\documentclass[a4paper,UKenglish,cleveref,numberwithinsect]{lipics-v2021}
\pdfoutput=1

\hideLIPIcs
\makeatletter
\AtBeginDocument{\nolinenumbers\let\@oddfoot\@empty\let\@evenfoot\@empty
  \hypersetup{pdfsubject={Preprint}}}
\makeatother

\title{Persistent Magnitude Homology for Quantitative Equational Theories}
\titlerunning{Persistent Magnitude Homology for Quantitative Equational Theories}

\author{Luciano Melodia}{Friedrich-Alexander Universität Erlangen-Nürnberg, Germany}{luciano.melodia@fau.de}{https://orcid.org/0000-0002-7584-7287}{}

\authorrunning{L. Melodia}

\Copyright{Luciano Melodia}



\supplement{The source code for the computations, checks, and figures.}
\supplementdetails[subcategory={Source Code}]{Software}{https://codeberg.org/Jiren/PersHomAlg}

\usepackage{mathtools}
\usepackage{booktabs}
\usepackage{colortbl}
\usepackage{tikz}
\usetikzlibrary{cd,arrows.meta,positioning,calc,fit,backgrounds}
\usepackage{sdproof}

\definecolor{sdRegion1}{HTML}{C49D97}
\definecolor{sdRegion2}{HTML}{BFB397}
\definecolor{sdRegion3}{HTML}{ACC9B3}
\definecolor{sdRegion4}{HTML}{AAD9DF}
\definecolor{sdRegion5}{HTML}{CFE0FD}
\definecolor{sdRegion6}{HTML}{FFE7FA}
\definecolor{sdAccent1}{HTML}{C74541}
\definecolor{sdAccent2}{HTML}{547E05}
\definecolor{sdAccent3}{HTML}{007E8A}
\definecolor{sdAccent4}{HTML}{7263C9}
\definecolor{sdRegionUnit}{HTML}{E5E5E5}
\definecolor{sdHint}{HTML}{815955}
\renewcommand\sdeq[1]{\;\overset{\text{\scriptsize\color{sdHint}#1}}{=}\;}

\definecolor{accent}{HTML}{B5470B}
\definecolor{accentbright}{HTML}{E07A1F}
\definecolor{accentsoft}{HTML}{F3C58E}
\definecolor{accentpale}{HTML}{FBF2E7}
\definecolor{ink}{HTML}{16171A}
\definecolor{inkmute}{HTML}{16171A}
\definecolor{complement}{HTML}{1F5E70}
\colorlet{tintleft}{accent!12}
\colorlet{tintright}{complement!12}
\colorlet{sdcodomain}{accent!10}
\colorlet{sddomain}{complement!10}
\colorlet{chasetint}{accentsoft}
\newcommand{\hlleft}[1]{{\setlength{\fboxsep}{1.5pt}\colorbox{tintleft}{\ensuremath{#1}}}}
\newcommand{\hlright}[1]{{\setlength{\fboxsep}{1.5pt}\colorbox{tintright}{\ensuremath{#1}}}}

\newlength{\lesjut}
\tikzset{
  les/.style 2 args={
    rounded corners=6pt,
    to path={(\tikztostart.east) -- ([xshift=\lesjut]\tikztostart.east)
             |- (#1) node[fill=white, inner sep=1.6pt] {$\scriptstyle #2$}
             -| ([xshift=-\lesjut]\tikztotarget.west)
             -- (\tikztotarget.west)}},
  trace/.style={-,draw=chasetint,line width=0.3pt,double=white,
                double distance=1.6pt,shorten <=-1.2pt,shorten >=-1.2pt},
  commutative diagrams/comm/.style={phantom,
    "{\color{inkmute}\circlearrowleft}" description}}

\numberwithin{equation}{section}
\allowdisplaybreaks

\crefformat{enumi}{(#2#1#3)}
\Crefformat{enumi}{(#2#1#3)}

\newcommand{\cat}[1]{\mathsf{#1}}
\newcommand{\Met}{\cat{Met}}
\newcommand{\QAlg}{\cat{QAlg}}
\newcommand{\Set}{\cat{Set}}
\newcommand{\Ab}{\cat{Ab}}
\newcommand{\Vect}{\cat{Vect}}
\newcommand{\SSet}{\cat{sSet}}
\newcommand{\Mod}{\operatorname{Mod}}
\newcommand{\Free}{\operatorname{Free}}
\newcommand{\colim}{\operatorname*{colim}}
\newcommand{\Ob}{\operatorname{Ob}}
\newcommand{\MH}{\operatorname{MH}}
\newcommand{\op}{\mathrm{op}}
\newcommand{\A}{\mathcal{A}}
\newcommand{\V}{\mathcal{V}}
\newcommand{\Rp}{[0,\infty]}
\newcommand{\VCat}{\V\text{-}\cat{Cat}}
\newcommand{\Vfin}{\V_{\mathrm{fin}}}
\newcommand{\Pfin}{\mathcal{P}_{\mathrm{fin}}}
\newcommand{\ar}{\operatorname{ar}}
\newcommand{\MC}{\operatorname{MC}}
\newcommand{\PH}{\operatorname{PH}}
\newcommand{\gr}{\operatorname{gr}}
\newcommand{\End}{\operatorname{End}}
\newcommand{\Nv}{\mathrm{N}}
\newcommand{\VR}{\operatorname{VR}}
\newcommand{\dnote}[1]{{\color{inkmute}\scriptstyle #1}}
\makeatletter
\newcommand{\twoheadrightarrowfill@}{\arrowfill@\relbar\relbar\twoheadrightarrow}
\newcommand{\xtwoheadrightarrow}[2][]{\ext@arrow 0359\twoheadrightarrowfill@{#1}{#2}}
\makeatother
\begin{document}

\maketitle

\begin{abstract}
A quantitative equational theory \(U\) reasons about terms that agree up to a
numerical error. It presents a free algebra \(T_UA\) over a metric space \(A\) of
generators, the terms of the syntax at the least distance the axioms derive, and
that metric is its semantic content. We give a functorial invariant of it, the
persistent magnitude homology of \(T_UA\): a barcode where the module is tame,
finite linear algebra where \(T_UA\) is finite, Lipschitz in each degree.
Magnitude homology is graded by length and knows nothing of
persistence, its persistent refinement nothing of where its bars begin and end,
yet the two are one construction: filtering the length nerve by sublevel sets of
the length yields the persistence module, and the associated graded of that
filtration is the magnitude complex. A long exact sequence exchanges them, and
each side gains what it lacked. Magnitude homology locates the critical values of
the barcode, so a graded computation lists the lengths at which an endpoint can
occur, and the barcode acquires a stability estimate of \((n+1)\delta\) in degree
\(n\) under a perturbation of size \(\delta\), and a computed perturbation
shows that the factor cannot be dropped. An inclusion of theories induces a
morphism of the presenting monads and, where the induced map is bijective and shortens no distance by more than
\(\delta\), a comparison of barcodes under the same bound, so a barcode movement
measures the metric-semantic strength of the added axioms. Four examples are
computed, one in every degree.
\end{abstract}

\section{Introduction}\label{sec:intro}
\looseness=-1
Programs, processes and distributions are often not equal but close, and a
\emph{quantitative equational theory} is a logic for that situation: its
judgements \(t=_\varepsilon s\) bound the error between two terms, and its free
algebras are extended metric spaces with nonexpansive operations, as Mardare,
Panangaden and Plotkin construct them
\cite[Def.~2.2, Def.~3.1, Thm.~6.1]{MardarePanangadenPlotkin2016QuantitativeAlgebraicReasoning}
\cite{BacciMardarePanangadenPlotkin2020QuantitativeEquationalReasoning,MardarePanangadenPlotkin2017Axiomatizability,MioSarkisVignudelli2022BeyondNonexpansive}.
Such theories axiomatise metric semantics for nondeterminism and probability
\cite{BacciMardarePanangadenPlotkin2018MarkovProcesses,MioVignudelli2020MonadsQuantitativeTheories}
and for effects that carry
numerical error \cite{BacciMardarePanangadenPlotkin2024SumTensor}. What a theory \(U\) presents is its free
algebra \(T_UA\) over a metric space \(A\) of generators, the terms of the syntax
at the least distance the axioms derive, so a theorem of \(U\) is a bound on a
distance in \(T_UA\) and that metric is the semantic content of \(U\).

\looseness=-1
Every extended metric space is a category enriched in the monoidal poset
\((\Rp,\ge,+,0)\), as Lawvere observes
\cite[p.~1]{Lawvere1973MetricSpacesGeneralizedLogic}, and every category enriched in that base carries magnitude homology
\cite{HepworthWillerton2017Categorifying,Leinster2013MagnitudeMetricSpaces}, a homology
theory whose classes are graded by the total length of a chain of points, as Leinster
and Shulman define it \cite[\S5]{LeinsterShulman2021MagnitudeHomology}, so a
quantitative equational theory acquires magnitude homology. Otter's filtration of the enriched nerve
\cite[Thm.~9.4]{Otter2022MagnitudeMeetsPersistence} refines the last link:
the homology becomes a persistence module of abelian groups
\cite{EdelsbrunnerLetscherZomorodian2002TopologicalPersistence,ZomorodianCarlsson2005ComputingPersistentHomology},
and over a field, under the tameness condition of \Cref{sec:barcodes}, it
decomposes into a barcode \cite[Thm.~1.1]{CrawleyBoevey2015Decomposition}. The two homology theories in that chain are not two constructions but
one, and four results follow:
\begin{enumerate}[A]
\item The magnitude complex in \(\Met\) is the associated graded of
  the length filtration on the normalised chains of its nerve
  (\Cref{def:magnitude,prop:mc-explicit}). Magnitude homology is the level-set
  theory of the length and its persistent refinement the sublevel-set theory.
\item A long exact sequence ties the two together (\Cref{thm:les}). It
  forces every endpoint of the barcode to be a length at which magnitude
  homology is nonzero (\Cref{cor:critical}): magnitude homology locates the
  critical values.
\item Moving every distance by at most \(\delta\) moves the persistence
  module by at most \((n+1)\delta\) in degree \(n\) (\Cref{thm:distortion}),
  and in \Cref{ex:distortion} one perturbation moves the barcode by more than
  \(\delta\). The invariant is Lipschitz in each degree, not nonexpansive.
\item An inclusion of theories induces a morphism of presenting monads and
  hence a comparison of persistence modules (\Cref{thm:theorymap}). Under the
  bijectivity condition of \Cref{thm:theorymap}\Cref{itm:boundedInterleaving}
  the barcode moves by at most \(n+1\) times the maximal shortening of a
  distance. This measures the metric-semantic strength of the added axioms.
\end{enumerate}

\looseness=-1
For a test space \(A\) and a degree \(n\) this compares theories
quantitatively: a nonzero interleaving distance between \(\PH_n(T_UA)\) and
\(\PH_n(T_{U'}A)\) certifies that the axioms of \(U'\smallsetminus U\) identify
terms or shorten a distance, and where \(q_A\) is bijective a barcode movement
of size \(\beta\) certifies a shortening by at least \(\beta/(n+1)\)
(\Cref{def:impact}). No completeness is claimed.

\section{The Base of Enrichment}\label{sec:base}
Let \(\V\) be the category with
\[
\Ob\V=\Rp,\qquad
\V(a,b)=\begin{cases}
          \{\bullet\} & \text{when } a\ge b,\\
          \varnothing & \text{when } a<b.
        \end{cases}
\]
Identities and composites are the unique morphisms available, and they exist
because \(a\ge a\), and because \(a\ge b\) and \(b\ge c\) give \(a\ge c\). Equip
\(\V\) with the tensor \(\otimes\colon\Ob\V\times\Ob\V\to\Ob\V\), \((a,b) \mapsto a+b\),
the unit \(I=0\), and the operation
\(\ominus\colon\Ob\V\times\Ob\V\to\Ob\V\) given by
\begin{equation}\label{eq:ominus}
c\ominus a=\begin{cases}
             0 & \text{when } c\le a,\\
             c-a & \text{when } a<c<\infty,\\
             \infty & \text{when } a<c=\infty.
           \end{cases}
\end{equation}
The three cases are mutually exclusive and exhaust \(\Ob\V\times\Ob\V\):
the negation of \(c\le a\) is \(a<c\), which splits into \(c<\infty\) and
\(c=\infty\).

\looseness=-1
Unwind the definition of a \(\V\)-category \(X\)
\cite[Ch.~1]{Kelly1982BasicConcepts}. It consists of a set \(\Ob X\), an object
\(X(x,y)\in\Ob\V\) for each pair \(x,y\in\Ob X\), a morphism
\(j_x\colon I\to X(x,x)\) for each \(x\), and a morphism
\(\mu_{x,y,z}\colon X(y,z)\otimes X(x,y)\to X(x,z)\) for each triple, subject to
associativity and unit axioms. Write \(d(x,y)=X(x,y)\). Then \(j_x\) exists
if and only if \(0\ge d(x,x)\), which forces \(d(x,x)=0\), and \(\mu_{x,y,z}\)
exists if and only if \(d(y,z)+d(x,y)\ge d(x,z)\). Both are unique where they
exist, so neither is data beyond \(d\), and the associativity and unit axioms
are equations between parallel morphisms of \(\V\), hence hold. A
\(\V\)-category is therefore a set \(\Ob X\) with a function
\(d\colon\Ob X\times\Ob X\to\Rp\) satisfying \(d(x,x)=0\) and the triangle
inequality, and nothing further is imposed: distances may be infinite, may fail
to be symmetric, and may vanish off the diagonal. These are the extended
quasi-pseudometric spaces of
Lawvere~\cite[p.~1]{Lawvere1973MetricSpacesGeneralizedLogic},
sometimes called Lawvere metric spaces.
A \(\V\)-functor \(F\colon X\to Y\) is a map \(f\colon\Ob X\to\Ob Y\) together
with a morphism \(f_{x,x'}\colon X(x,x')\to Y(fx,fx')\) for each pair, that is,
with \(d(x,x')\ge d(fx,fx')\), subject to compatibility with the compositions
and the identities of \(X\) and \(Y\). That morphism is again unique where it
exists, and the two compatibilities are equations between parallel morphisms of
the thin category \(\V\), hence hold, so a \(\V\)-functor is the same thing
as a nonexpansive map. Wherever one is named below it carries a capital letter,
and its two underlying maps the matching small one. Write \(\VCat\) for the
category with \(\Ob \VCat \) the small \(\V\)-categories and \(\VCat(X,Y)\)
the set of \(\V\)-functors \(X\to Y\). The identity \(1_{X}\) is the identity
function on \(\Ob X\) together with \((1_{X})_{x,x'}=1_{X(x,x')}\), and the
composite of \(F\colon X\to Y\) with \(G\colon Y\to Z\) is \(x\mapsto gfx\)
on objects together with
\((GF)_{x,x'}=g_{fx,fx'}\circ f_{x,x'}\colon X(x,x')\to Z(gfx,gfx')\) on
hom-objects. The associativity and unitality of this composition are again
equations between parallel morphisms of \(\V\). A \(\V\)-functor \(X\to Y\)
is a function \(\Ob X\to\Ob Y\) with a family of morphisms of \(\V\) indexed
by \(\Ob X\times\Ob X\), so \(\VCat(X,Y)\) is a set.

\begin{lemma}\label{lem:base}
The category \(\V\) is complete, cocomplete, symmetric monoidal and closed, its
internal hom-functor \((-) \ominus a\) being right adjoint to \(a \otimes (-)\), written as \(a \otimes (-) \dashv (-) \ominus a\). Its unit
\(I=0\) is terminal, so \(\V\) is semicartesian, and its tensor is not the
cartesian product.
\end{lemma}

\looseness=-1
A quantitative algebra has an underlying \(\V\)-category, and that
\(\V\)-category is symmetric and separated: the distance function is a metric,
allowed the value \(\infty\)
\cite[Def.~3.1]{MardarePanangadenPlotkin2016QuantitativeAlgebraicReasoning}.
Write \(\Met\) for the full subcategory of \(\VCat\) whose objects are the
\(\V\)-categories \(X\) with \(d(x,y)=d(y,x)\) for all \(x,y\in\Ob X\), and with
\(x=y\) whenever \(d(x,y)=0\). Thus \(\Met(X,Y)=\VCat(X,Y)\), both sets
being the nonexpansive maps \(X\to Y\). Write \(\Vfin\)
for the full subcategory of \(\V\) on \([0,\infty)\subset\Ob\V\). Its enriched
categories are the \([0,\infty)\)-categories of the magnitude literature. Since
\([0,\infty)\) contains \(I=0\) and is closed under \(+\), the category
\(\Vfin\) is a symmetric monoidal subcategory of \(\V\) and is again
semicartesian. It is neither complete nor cocomplete, since \(\sup\mathbb{N}=\infty\) and
\(\inf\varnothing=\infty\) both lie outside \(\Ob\Vfin\). The distinction
matters: the lengths that index everything from \Cref{sec:nerve} on run over
\(\Ob\Vfin\) and not over all of \(\Ob\V\). The magnitude complex of
Leinster and Shulman~\cite[Def.~3.3]{LeinsterShulman2021MagnitudeHomology} is likewise graded by
finite lengths, although their base of enrichment is \([0,\infty]\).

\looseness=-1
Finally, on a \(\V\)-category \(X\) the relation that holds of \(x\) and
\(y\) when \(d(x,y)<\infty\) and \(d(y,x)<\infty\) is an equivalence relation
on \(\Ob X\): reflexive because \(d(x,x)=0\), symmetric by its form, and
transitive by the triangle inequality in both directions. The full
subcategory of \(X\) spanned by any one class of this relation is again a
\(\Vfin\)-category.

\section{Quantitative Equational Theories and Their Monad}\label{sec:theories}
A \emph{signature} is a pair \((\Omega,\ar)\) with \(\Omega\) a finite set of
operation symbols and \(\ar\colon\Omega\to\mathbb{N}\). Put
\(\Omega_n=\ar^{-1}\{n\}\), so \(\Omega=\bigsqcup_{n\in\mathbb{N}}\Omega_n\) with
\(\Omega_n=\varnothing\) for all but finitely many \(n\). Fix a set \(\Xi\) of
variables with \(|\Xi|=\aleph_0\). An \emph{\(\Omega\)-algebra} is a set \(A\)
with a map \(f^A\colon A^n\to A\) for each \(n\in\mathbb{N}\) and each
\(f\in\Omega_n\). An \emph{\(\Omega\)-homomorphism} \(h\colon A\to B\) is a map
with \(h(f^A(a_1,\dots,a_n))=f^B(ha_1,\dots,ha_n)\) for all \(f\in\Omega_n\)
and \(a_1,\dots,a_n\in A\). On a set \(W\) the terms are built in grades, one grade
for each nesting depth of operations, and \(T_\Omega W=\bigcup_{k\ge0}T^{k}_\Omega W\)
collects the grades produced by the recursion, whose base is empty and whose step
adjoins \(W\) together with one further layer of operations applied to the grade
that has already been built,
\begin{equation}\label{eq:terms}
  T^{0}_\Omega W=\varnothing,
  \qquad
  T^{k+1}_\Omega W=W\sqcup\bigsqcup_{n\in\mathbb{N}}\Omega_n\times\bigl(T^{k}_\Omega W\bigr)^{n},
\end{equation}
whose elements are the
\emph{\(\Omega\)-terms over \(W\)}, where \(\sqcup\) and \(\bigsqcup\)
denote disjoint unions and \(\times\) the cartesian product. Write \(\eta_Ww\) for the copy of \(w\in W\)
in a left summand and \(f(t_1,\dots,t_n)\) for the copy of \((f,t_1,\dots,t_n)\)
in a right one. For an \(\Omega\)-algebra \(A\) and a map
\(g\colon W\to A\), let the extension \(g^\sharp\colon T_\Omega W\to A\) of the map \(g\)
be given by the recursion below,
\begin{equation}\label{eq:extension}
  g^\sharp\eta_Ww=gw,
  \qquad
  g^\sharp f(t_1,\dots,t_n)=f^{A}(g^\sharp t_1,\dots,g^\sharp t_n),
\end{equation}
and for a map \(r\colon W\to W'\) of sets put
\(T_\Omega r=(\eta_{W'}\circ r)^\sharp\) and
\(\mu_W=(1_{T_\Omega W})^\sharp\colon T_\Omega T_\Omega W\to T_\Omega W\), the
map that flattens a term of terms. \Cref{lem:term-monad} makes these definitions
well posed:

\begin{lemma}\label{lem:term-monad}
Let \((\Omega,\ar)\) be a signature, let \(W\) and \(W'\) be sets, let \(A\) be
an \(\Omega\)-algebra and let \(g\colon W\to A\) be a map, with \(T_\Omega\),
\(\eta\), \((-)^\sharp\) and \(\mu\) as in \eqref{eq:terms} and
\eqref{eq:extension} above.
\begin{enumerate}
\item\label{itm:tm-grades} The grades \(T^{k}_\Omega W\) increase, every \(\Omega\)-term over \(W\)
  lies in one of them, and \eqref{eq:terms} holds with \(T_\Omega W\) in place of
  both grades. Hence every \(\Omega\)-term over \(W\) is \(\eta_Ww\) for a
  unique \(w\in W\), or \(f(t_1,\dots,t_n)\) for a unique \(n\in\mathbb{N}\), one
  \(f\in\Omega_n\) and one tuple \((t_1,\dots,t_n)\in(T_\Omega W)^{n}\), and never both.
  The map \(\eta_W\colon W\to T_\Omega W\) is injective, and \(T_\Omega W\) is an
  \(\Omega\)-algebra with \(f^{T_\Omega W}(t_1,\dots,t_n)=f(t_1,\dots,t_n)\).
  Each \(t_i\) occurs at a strictly earlier grade than \(f(t_1,\dots,t_n)\), so
  the recursion \eqref{eq:extension} defines \(g^\sharp\) on all of
  \(T_\Omega W\) and determines it uniquely.
\item\label{itm:tm-free} The map \(g^\sharp\) is the unique \(\Omega\)-homomorphism with
  \(g^\sharp\circ\eta_W=g\), that is, the unique filler of
  \[
  \begin{tikzcd}[row sep=1.7em, column sep=3.0em]
    W \arrow[r, "\eta_{W}"] \arrow[dr, "g"'] & T_\Omega W \arrow[d, dashed, "{\exists!\ \Omega\text{-hom }g^\sharp}"]\\
    & A,
  \end{tikzcd}
  \]
  \looseness=-1
  in which only the filler is required to be a homomorphism, so \(T_\Omega W\)
  is free on \(W\).
\item\label{itm:tm-functor} The assignment \(T_\Omega\colon\Set\to\Set\) is a functor, and on a map
  \(r\colon W\to W'\) it renames variables,
  \[
    T_\Omega r\circ\eta_W=\eta_{W'}\circ r,
    \qquad
    T_\Omega r\,f(t_1,\dots,t_n)=f(T_\Omega r\,t_1,\dots,T_\Omega r\,t_n),
  \]
  the first of the two equations being the naturality of the unit \(\eta\) of
  \eqref{eq:terms}.
\item\label{itm:tm-monad} \((T_\Omega,\eta,\mu)\) is a monad on \(\Set\), and
  \((-)^\sharp\) is its Kleisli extension: \(g^\sharp=\mu_{W'}\circ T_\Omega g\)
  for \(A=T_\Omega W'\).
\end{enumerate}
\end{lemma}

Take the set \(W=\Xi\) and regard \(\Xi\subseteq T_\Omega\Xi\) along \(\eta_\Xi\). A
\emph{substitution} is a map \(\sigma\colon\Xi\to T_\Omega\Xi\), an endomorphism
of \(\Xi\) in the Kleisli category of \(T_\Omega\). It acts on terms as its
Kleisli extension \(\sigma^\sharp\colon T_\Omega\Xi\to T_\Omega\Xi\), and on
subsets of \(T_\Omega\Xi\) elementwise.
The three sets of syntactic data are
\[
\mathcal{E}=T_\Omega\Xi\times\mathbb{Q}_{\ge 0}\times T_\Omega\Xi,\qquad
\mathcal{I}=\Pfin(\mathcal{E})\times\mathcal{E},\qquad
\mathcal{I}_{\mathrm{b}}=\Pfin(\Xi\times\mathbb{Q}_{\ge 0}\times\Xi)
  \times\mathcal{E},
\]
where \(\Pfin\) denotes the finite powerset. Elements of \(\mathcal{E}\) are
\emph{quantitative equations}, and \(t=_\varepsilon s\) abbreviates the triple
\((t,\varepsilon,s) \in \mathcal{E}\). Elements of \(\mathcal{I}\) are \emph{quantitative
inferences}, and \(\Gamma\vdash\varphi\) abbreviates the pair
\((\Gamma,\varphi)\in\mathcal{I}\). Both abbreviations are syntax, not relations. Since
\(\Xi\subseteq T_\Omega\Xi\), the \emph{basic} inferences form a subset
\(\mathcal{I}_{\mathrm{b}}\subseteq\mathcal{I}\): they are those whose
hypotheses relate variables only, and no compound terms
\cite[\S2]{MardarePanangadenPlotkin2016QuantitativeAlgebraicReasoning}.

\begin{definition}\label{def:deduction}
A \emph{deducibility relation} is a relation
\({\vdash}\subseteq 2^{\mathcal{E}}\times\mathcal{E}\) closed under the rules
below, stated for all \(n\in\mathbb{N}\), \(f\in\Omega_n\),
\(t,s,u,t_i,s_i\in T_\Omega\Xi\),
\(\varepsilon,\varepsilon'\in\mathbb{Q}_{\ge 0}\),
\(\varphi,\psi\in\mathcal{E}\) and \(\Gamma,\Gamma'\subseteq\mathcal{E}\). The
hypothesis sets are arbitrary subsets of \(\mathcal{E}\), and are not required to
be finite:
\begin{description}
\item[Reflexivity:] \(\varnothing\vdash t=_0 t\).
\item[Symmetry:] \(\{t=_\varepsilon s\}\vdash s=_\varepsilon t\).
\item[Triangle inequality:] \(\{t=_\varepsilon s,\;s=_{\varepsilon'}u\}
  \vdash t=_{\varepsilon+\varepsilon'}u\).
\item[Monotonicity:] \(\{t=_\varepsilon s\}\vdash t=_{\varepsilon+\varepsilon'}s\) for
  \(\varepsilon'>0\).
\item[Archimedean:] \(\{t=_{\varepsilon'}s\mid\varepsilon'>\varepsilon\}
  \vdash t=_\varepsilon s\).
\item[Nonexpansiveness:] \(\{t_i=_\varepsilon s_i\mid 1\le i\le n\}\vdash
  f(t_1,\dots,t_n)=_\varepsilon f(s_1,\dots,s_n)\).
\item[Substitution:] if \(\Gamma\vdash t=_\varepsilon s\) then
  \(\sigma^\sharp\Gamma\vdash\sigma^\sharp t=_\varepsilon\sigma^\sharp s\) for
  every substitution \(\sigma\).
\item[Cut:] if \(\Gamma\vdash\varphi\) for every \(\varphi\in\Gamma'\) and
  \(\Gamma'\vdash\psi\), then \(\Gamma\vdash\psi\).
\item[Assumption:] if \(\varphi\in\Gamma\) then \(\Gamma\vdash\varphi\).
\end{description}
\end{definition}

\looseness=-1
Every rule of \Cref{def:deduction} is a Horn clause over the membership
predicate of \({\vdash}\): finitely or infinitely many memberships as
premisses and one membership as conclusion, the clause form of
Horn~\cite{Horn1951DirectUnions}, taken infinitary as in
Jurka, Milius and Urbat~\cite[Def.~3.2]{JurkaMiliusUrbat2024AlgebraicReasoningRelationalStructures}.
Horn clauses survive arbitrary intersection: let
\(({\vdash_i})_{i\in I}\) be deducibility relations and
\({\vdash}=\bigcap_{i\in I}{\vdash_i}\). A rule asserting an element
unconditionally puts it in each \({\vdash_i}\), hence in \({\vdash}\), and a
rule asserting an element from hypotheses lying in \({\vdash}\) finds them in
each \({\vdash_i}\), so its conclusion lies in \({\vdash}\). Moreover
\(2^{\mathcal{E}}\times\mathcal{E}\) is a deducibility relation, every rule
holding because every element is a member, so for
\(S\subseteq\mathcal{I}_{\mathrm{b}}\) the deducibility relations containing
\(S\) form a nonempty family, and their intersection \({\vdash_S}\) is again
one, contains \(S\), lies inside each member, and is the unique least one. The \emph{quantitative equational theory induced by} \(S\) is
\(U=({\vdash_S})\cap\mathcal{I}\), and the elements of \(S\) are its \emph{axioms}
\cite[Def.~2.2]{MardarePanangadenPlotkin2016QuantitativeAlgebraicReasoning}.
Two operations act here and both are needed: closure passes from \(S\) to
\({\vdash_S}\), and intersection with \(\mathcal{I}\) then discards the pairs
whose hypothesis set is infinite. Such pairs occur, since the Archimedean rule
places the pair with the infinite hypothesis set
\(\{t=_{\varepsilon'}s\mid\varepsilon'>\varepsilon\}\) in every deducibility
relation. Derivations may therefore pass through pairs outside \(U\), while
\(S\) and \(U\) hold finite hypothesis sets only
\cite[\S2]{MardarePanangadenPlotkin2016QuantitativeAlgebraicReasoning}. The
index \(\varepsilon\) is rational, and the distances it induces are real, as
\Cref{thm:free} shows.

\begin{definition}\label{def:qalg}
A \emph{quantitative algebra} is a triple \((A,\Omega^A,d^A)\) in which
\(\Omega^A=(f^A)_{f\in\Omega}\) is a family of maps
\(f^A\colon A^{\ar f}\to A\), one for each operation symbol, making
\((A,\Omega^A)\) an \(\Omega\)-algebra, and \(d^A\colon A\times A\to\Rp\) is
an extended metric on \(A\) making every operation nonexpansive: for \(f\in\Omega_n\) and \(a_i,b_i\in A\),
if \(d^A(a_i,b_i)\le\varepsilon\) for \(1\le i\le n\) then
\(d^A(f^A(a_1,\dots,a_n),f^A(b_1,\dots,b_n))\le\varepsilon\)
\cite[Def.~3.1]{MardarePanangadenPlotkin2016QuantitativeAlgebraicReasoning}. An
\emph{assignment} is a map \(\iota\colon\Xi\to A\), and it interprets terms
along \(\iota^\sharp\colon T_\Omega\Xi\to A\). The algebra \emph{satisfies}
\((\Gamma,t=_\varepsilon s)\in\mathcal{I}\) if every \(\iota\) with
\(d^A(\iota^\sharp t',\iota^\sharp s')\le\varepsilon'\) for all
\((t'=_{\varepsilon'}s')\in\Gamma\) also has
\(d^A(\iota^\sharp t,\iota^\sharp s)\le\varepsilon\), and we then write
\(A\models\varphi\)
\cite[Def.~4.1]{MardarePanangadenPlotkin2016QuantitativeAlgebraicReasoning}.
Let \(\QAlg\) be the category with the quantitative algebras over
\((\Omega,\ar)\) as objects and, as morphisms \(h\colon A\to B\), the
\(\Omega\)-homomorphisms with \(d^B(ha,ha')\le d^A(a,a')\) for all
\(a,a'\in A\). For a quantitative equational theory \(U\) put \cite[Def.~4.2]{MardarePanangadenPlotkin2016QuantitativeAlgebraicReasoning}:
\[
\Ob\Mod(U)=\{A\in\Ob\QAlg\mid A\models\varphi\ \text{for all}\ \varphi\in U\},
\qquad \Mod(U)(A,B)=\QAlg(A,B).
\]
\end{definition}

\looseness=-1
Satisfaction is well defined: each assignment \(\iota\) has a unique
extension \(\iota^\sharp\), so both compared bounds are determined by \(\iota\)
alone, and the quantifier ranges over the set \(A^\Xi\). Composites and
identities of such morphisms are again morphisms, so \(\QAlg\) is a category and
\(\Mod(U)\) is the full subcategory of it cut out by the axioms of \(U\).

\begin{theorem}\label{thm:free}
Let \((\Omega,\ar)\) be a signature, let \(\Xi\) be as in \Cref{def:deduction},
and let \(U\) be the quantitative equational theory induced by a set
\(S\subseteq\mathcal{I}_{\mathrm{b}}\) of basic inferences in the sense of that
definition.
\begin{enumerate}
\item\label{itm:free-metric} Define \(d_U\colon T_\Omega\Xi\times T_\Omega\Xi\to\Rp\) by
  \(d_U(t,s)=\inf\{\varepsilon\mid(\varnothing,t=_\varepsilon s)\in U\}\), with
  \(\inf\varnothing=\infty\). Then \(d_U\) is an extended pseudometric on
  \(T_\Omega\Xi\), the relation \(d_U=0\) is an \(\Omega\)-congruence, and the
  quotient by that relation lies in \(\Mod(U)\)
  \cite[\S5, Thm.~5.1]{MardarePanangadenPlotkin2016QuantitativeAlgebraicReasoning}.
\item\label{itm:free-sound} The rules of \Cref{def:deduction} are sound and complete for \(\Mod(U)\)
  \cite[Thm.~5.2]{MardarePanangadenPlotkin2016QuantitativeAlgebraicReasoning}.
\item \label{itm:forgetfulfunctor} The forgetful functor \(G_U\colon\Mod(U)\to\Met\) has a left adjoint
  \(\Free_U\colon\Met\to\Mod(U)\)
  \cite[Thm.~6.1]{MardarePanangadenPlotkin2016QuantitativeAlgebraicReasoning}
  and is monadic
  \cite[Cor.~4.6, Prop.~4.7]{MardareGhaniRischel2025MetricEquationalTheories}.
  The induced monad \(T_U=G_U\Free_U\colon\Met\to\Met\) is the \emph{metric term
  monad} of \(U\)
  \cite[\S6]{MardarePanangadenPlotkin2016QuantitativeAlgebraicReasoning}, and
  \(\Mod(U)\) is equivalent to the category of \(T_U\)-algebras over \(\Met\).
\end{enumerate}
\end{theorem}

\looseness=-1
Part \Cref{itm:free-metric} reads the free algebra as \(\Omega\)-terms identified when
the axioms derive distance \(0\), at the distance \(d_U(t,s)\) that is the least
cost at which \(U\) derives \(t=_\varepsilon s\). More axioms mean more
derivations, so they only shorten distances or identify terms.

\looseness=-1
Part \Cref{itm:forgetfulfunctor} needs a word on provenance. Mardare,
Panangaden and Plotkin construct the left adjoint and name the monad, but prove
no comparison with the \(T_U\)-algebras. Monadicity is stated for metric
equational theories, whose operations carry metric arities, and a signature with
discrete arities yields a quantitative equational theory with the same notion of
model \cite[\S8]{MardareGhaniRischel2025MetricEquationalTheories}, so the
statement transfers. Finiteness of \(\Omega\) and \(|\Xi|=\aleph_0\) serve the
computations below, not the structure theory: the objects of \Cref{thm:free}
survive arbitrary arities, with a variety theorem
\cite[Ex.~3.2(3), Thm.~5.3]{MiliusUrbat2019EquationalAxiomatization}, also over
relational structures
\cite[Ex.~4.19]{JurkaMiliusUrbat2024AlgebraicReasoningRelationalStructures},
present \(\omega_1\)-accessible monads over the \(1\)-bounded metric spaces
\cite[Ex.~3.5, Rem.~4.20]{FordMiliusSchroder2021MonadsRelationalStructures}, and
the monads on \(\Met\) presented by a variety of quantitative algebras
\cite{Adamek2022VarietiesQuantitativeAlgebrasMonads,AdamekDostalVelebil2023StronglyFinitaryMonads,Rosicky2021MetricMonads}
are the \(1\)-basic ones
\cite[Thm.~40]{Adamek2026OneBasicMonads}.

\section{The Length Nerve and Its Two Homologies}\label{sec:nerve}
Both homology theories come from one simplicial object, filtered by length.

Write \(\Delta\) for the category whose objects are the finite ordinals
\([n]=\{0<1<\dots<n\}\), \(n\in\mathbb{N}\), and whose morphisms are the
order-preserving maps. For \(n\ge 1\) and \(0\le i\le n\) the \emph{coface}
\(\delta^i\colon[n-1]\to[n]\), and for \(n\ge 0\) and \(0\le i\le n\) the
\emph{codegeneracy} \(\sigma^i\colon[n+1]\to[n]\), are
\begin{equation}\label{eq:cofaces}
  \delta^i(k)=\begin{cases}k,&k<i,\\ k+1,&k\ge i,\end{cases}
  \qquad
  \sigma^i(k)=\begin{cases}k,&k\le i,\\ k-1,&k>i,\end{cases}
\end{equation}
so \(\delta^i\) omits the value \(i\) and \(\sigma^i\) takes it twice. A
\emph{simplicial set} is a functor \(K\colon\Delta^{\op}\to\Set\) and a
\emph{simplicial map} a natural transformation between two such. They form the
category \(\SSet\). Write \(K_n=K[n]\),
\(d_i=K(\delta^i)\colon K_n\to K_{n-1}\) and
\(s_i=K(\sigma^i)\colon K_n\to K_{n+1}\) for \(0\le i\le n\). These sets and
operators determine \(K\), and they obey the simplicial identities, the
relations that \eqref{eq:cofaces} satisfies in \(\Delta\)
\cite[Defs.~1.1 and 2.1]{May1967SimplicialObjects}. A
\emph{simplicial subset} of \(K\) is a subfunctor, that is, a family of subsets
\(L_n\subseteq K_n\), one for every \(n\), such that the inclusion
\(K(\alpha)L_n\subseteq L_m\) holds for every morphism
\(\alpha\colon[m]\to[n]\) of \(\Delta\).

Let \(X\) be an object of \(\Met\), and write
\(\lvert-\rvert\colon\Delta\to\Set\) for the functor that forgets the order and
retains only the underlying set. The \emph{length nerve} of \(X\) is the
composite
\begin{align}\label{eq:nerve}
  &\Nv(X)\colon\Delta^{\op}\xrightarrow{\;\lvert-\rvert^{\op}\;}\Set^{\op}
    \xrightarrow{\;\Set(-,\Ob X)\;}\Set,
  \\[0.35em]\nonumber
  &\hspace{2.4em}
  \begin{tikzcd}[row sep=2.1em, column sep=2.8em, ampersand replacement=\&]
    {[m]} \arrow[r, maps to] \arrow[d, "\alpha"']
      \& {\Set\bigl([m],\Ob X\bigr)}\\
    {[n]} \arrow[r, maps to]
      \& {\Set\bigl([n],\Ob X\bigr)}.
        \arrow[u, "{\substack{\Nv(X)(\alpha)\\[0.2em]
             \dnote{\mathbf{x}\,\longmapsto\,\mathbf{x}\circ\alpha}}}"']
  \end{tikzcd}
\end{align}
A composite of functors is a functor, so \(\Nv(X)\) is a simplicial set with
\(\Nv(X)_n=\Set\bigl([n],\Ob X\bigr)\). An \(n\)-simplex is thus a map
\(\mathbf{x}\colon[n]\to\Ob X\), written
\(\mathbf{x}=\langle x_0,\dots,x_n\rangle\) with \(x_i=\mathbf{x}(i)\).
Substituting the cofaces and codegeneracies \eqref{eq:cofaces} into
\eqref{eq:nerve} gives
\[
  (d_i\mathbf{x})_k=\begin{cases}x_k,&k<i,\\ x_{k+1},&k\ge i,\end{cases}
  \qquad
  (s_i\mathbf{x})_k=\begin{cases}x_k,&k\le i,\\ x_{k-1},&k>i,\end{cases}
\]
that is, the operator \(d_i\) deletes the entry \(x_i\) from the tuple, and
\(s_i\) repeats it,
\[
  d_i\mathbf{x}=\langle x_0,\dots,x_{i-1},x_{i+1},\dots,x_n\rangle,
  \qquad
  s_i\mathbf{x}=\langle x_0,\dots,x_{i-1},x_i,x_i,x_{i+1},\dots,x_n\rangle .
\]
The \emph{length} of an \(n\)-simplex \(\mathbf{x}\) and the \emph{sublevel
sets} of the length function are
\begin{equation}\label{eq:length}
  \lambda\mathbf{x}=\sum_{i=1}^{n}d(x_{i-1},x_i)\in\Rp,
  \qquad
  \bigl(\Nv(X)^{\le\ell}\bigr)_n=\bigl\{\mathbf{x}\in\Nv(X)_n\bigm|\lambda\mathbf{x}\le\ell\bigr\},
\end{equation}
the empty sum being \(0\), so that \(\lambda\) vanishes on \(\Nv(X)_0\). The
value \(\infty\) is allowed, since distances in \(\Met\) are, while the
parameter \(\ell\) runs over the finite values \([0,\infty)\). These
are the enriched nerve and the length of
Otter~\cite[Defs.~4.1 and 4.2]{Otter2022MagnitudeMeetsPersistence}, whose filtered
simplicial set is the assignment \(\ell\mapsto\Nv(X)^{\le\ell}\).
\Cref{lem:filtration} shows that each \(\Nv(X)^{\le\ell}\) is a simplicial
subset of \(\Nv(X)\) and that the assignment \(\Nv(-)^{\le-}\) is functorial in
\(\ell\) and in the space \(X\).

\begin{lemma}\label{lem:filtration}
Let \(F\colon X\to Y\) be nonexpansive, with \(X\) and \(Y\) objects of
\(\Met\).
\begin{enumerate}
\item\label{itm:len-ops-stmt} For every \(n\ge 0\), every
  \(\mathbf{x}\in\Nv(X)_n\) and every \(0\le i\le n\), the degeneracy preserves
  the length, \(\lambda(s_i\mathbf{x})=\lambda\mathbf{x}\), and for \(n\ge 1\)
  the face does not raise it, so that \(\lambda(d_i\mathbf{x})\le\lambda\mathbf{x}\).
\item\label{itm:filt-functor} Each \(\Nv(X)^{\le\ell}\) is a simplicial subset of \(\Nv(X)\), and
  \(\ell\le\ell'\) gives \(\Nv(X)^{\le\ell}\subseteq\Nv(X)^{\le\ell'}\), so
  \(\Nv(X)^{\le-}\colon\Vfin^{\op}\to\SSet\) is a functor, which is the
  filtration of \(\Nv(X)\) by length.
\item\label{itm:filt-map} \(\Nv(F)\colon\Nv(X)\to\Nv(Y)\), \(\langle x_0,\dots,x_n\rangle\mapsto
  \langle fx_0,\dots,fx_n\rangle\), is a simplicial map satisfying
  \(\lambda\circ\Nv(F)\le\lambda\), hence restricts to
  \(\Nv(F)^{\le\ell}\colon\Nv(X)^{\le\ell}\to\Nv(Y)^{\le\ell}\), naturally
  in \(\ell\).
\item\label{itm:filt-nerve} For \(G\colon Y\to Z\) in \(\Met\), \(\Nv(1_X)=1_{\Nv(X)}\) and
  \(\Nv(GF)=\Nv(G)\circ\Nv(F)\), so \(\Nv(-)^{\le-}\colon\Met\to\SSet^{\Vfin^{\op}}\)
  is a functor.
\end{enumerate}
\end{lemma}

Write \(\mathbb{Z}[\Nv(X)_n]\) for the free abelian group on the \(n\)-simplices
of the length nerve. The \emph{normalised chain complex} \(C_\bullet(X)\) is the
quotient by the subgroup on the degenerate tuples,
\begin{equation}\label{eq:normalised}
  D_n(X)=\bigl\langle s_i\mathbf{y}\bigm|
    \mathbf{y}\in\Nv(X)_{n-1},\ 0\le i\le n-1\bigr\rangle_{\mathbb{Z}},
  \qquad
  C_n(X)=\mathbb{Z}[\Nv(X)_n]\big/D_n(X),
\end{equation}
with \(D_0(X)=0\) and \(\partial=\sum_{i=0}^{n}(-1)^id_i\). The subgroups
\(D_\bullet(X)\) form a subcomplex, since for every \(n\ge 1\), every
\(\mathbf{y}\in\Nv(X)_{n-1}\) and every \(0\le i\le n-1\) the simplicial
identities give
\begin{equation}\label{eq:degenerate}
  \partial s_i\mathbf{y}
    =\sum_{j=0}^{i-1}(-1)^js_{i-1}d_j\mathbf{y}
     +(-1)^i\mathbf{y}+(-1)^{i+1}\mathbf{y}
     +\sum_{j=i+2}^{n}(-1)^js_id_{j-1}\mathbf{y},
\end{equation}
where the two middle terms cancel and every remaining term is degenerate, so
\(\partial D_n(X)\subseteq D_{n-1}(X)\). As \(D_n(X)\) is spanned by a subset of
the basis \(\Nv(X)_n\), the quotient \(C_n(X)\) is free on the complementary
subset, the \emph{nondegenerate} tuples, those with \(x_{i-1}\neq x_i\) for
\(1\le i\le n\). No homology is lost in this quotient:
\(\mathbb{Z}[\Nv(X)_\bullet]\) is the direct sum of \(D_\bullet(X)\) and a
subcomplex carried isomorphically onto \(C_\bullet(X)\), and \(D_\bullet(X)\) is
acyclic, so the quotient map is a quasi-isomorphism
\cite[Lemmas 14.23.6 and 14.23.9]{StacksProject}.
For every length \(\ell\in[0,\infty)\) the filtration of \(C_\bullet(X)\) by
length has as its stage at \(\ell\) the subgroup
\begin{equation}\label{eq:filtered-chains}
  F_\ell C_n(X)=\bigl\langle\mathbf{x}\bigm|\mathbf{x}\in\Nv(X)_n
    \ \text{nondegenerate},\ \lambda\mathbf{x}\le\ell\bigr\rangle_{\mathbb{Z}}
    \subseteq C_n(X),
\end{equation}
spanned by the nondegenerate tuples of length at most \(\ell\). It is the
normalised chain complex of the simplicial subset \(\Nv(X)^{\le\ell}\) of
\eqref{eq:length}, a subcomplex by \Cref{lem:filtration}, and \(\ell\le\ell'\) gives an inclusion of
complexes \(F_\ell C_\bullet(X)\subseteq F_{\ell'}C_\bullet(X)\). Reading the
homology of every grade in degree \(n\) turns that chain of inclusions into
persistent magnitude homology:
\[
\PH_n(X)\colon\Vfin^{\op}\to\Ab,
\qquad
\begin{tikzcd}[row sep=2.1em, column sep=2.8em, ampersand replacement=\&]
  \ell \arrow[r, maps to] \arrow[d, "\le"']
    \& H_n\bigl(F_\ell C_\bullet(X)\bigr)
       \arrow[d, "{\substack{\PH_n(X)(\ell\le\ell')\\[0.2em]
            \dnote{[z]\,\longmapsto\,[z]}}}"]\\
  \ell' \arrow[r, maps to]
    \& H_n\bigl(F_{\ell'}C_\bullet(X)\bigr).
\end{tikzcd}
\]
Its value on the unique morphism \(\ell\le\ell'\) of \(\Vfin^{\op}\) is the map
induced on homology by that inclusion: a cycle \(z\) of \(F_\ell C_n(X)\) is a
cycle of \(F_{\ell'}C_n(X)\), and the map sends the class of \(z\) in the
smaller complex to its class in the larger one. These induced maps are the
\emph{transition maps} of the persistence module. This is Otter's blurred
magnitude homology \cite[\S9, Thm.~9.4]{Otter2022MagnitudeMeetsPersistence},
\cite{GovcHepworth2021PersistentMagnitude}, the sublevel condition \(\lambda\mathbf{x}\le\ell\) being the blur that
Leinster and Shulman ask for \cite[\S8]{LeinsterShulman2021MagnitudeHomology},
a version that notices approximate equalities in the triangle inequality. A
nonexpansive \(F\colon X\to Y\) induces a chain map \(C_\bullet(F)\) that
preserves grades of filtrations, so it induces a map
\(\PH_n(F)(\ell)\colon\PH_n(X)(\ell)\to\PH_n(Y)(\ell)\) at each length. These
maps commute with the transition maps, so \(\PH_n(F)\) is a morphism
\(\PH_n(X)\to\PH_n(Y)\) of persistence modules.

\begin{corollary}\label{cor:degree-zero}
For \(X\) in \(\Met\) and \(\ell\in[0,\infty)\), the group \(\PH_0(X)(\ell)\) is
free abelian on the components of the graph with vertex set \(\Ob X\) and an
edge \(xy\) whenever \(d(x,y)\le\ell\). Hence \(\PH_0(X)\) is the degree-zero
persistent homology of the Vietoris--Rips filtration of \(X\)
\cite{Carlsson2009TopologyData}, and for finite \(\Ob X\) its barcode over a
field is the single-linkage dendrogram
\cite{CarlssonMemoli2010HierarchicalClustering}.
\end{corollary}

\looseness=-1
Above degree zero the invariant and Vietoris--Rips persistence part company,
because their filtrations do \cite{Cho2019QuantalesPersistenceMagnitude}: a finite nonempty \(\sigma\subseteq\Ob X\) is a
simplex of \(\VR_\ell(X)\) when \(\max_{x,y\in\sigma}d(x,y)\le\ell\), whereas
\(\mathbf{x}\in\bigl(\Nv(X)^{\le\ell}\bigr)_n\) when
\(\sum_{i=1}^{n}d(x_{i-1},x_i)\le\ell\).
The graded pieces of this filtration carry magnitude homology
\cite{AsaoIzumihara2021GeometricApproachGraphMagnitudeHomology,KanetaYoshinaga2021MagnitudeHomologyOrderComplexes}.

\begin{definition}\label{def:magnitude}
Let \(X\) be an object of \(\Met\) and \(\ell\in[0,\infty)\). Write
\(F_{<\ell}C_\bullet(X)\) for the union \(\bigcup_{\ell'<\ell}F_{\ell'}C_\bullet(X)\)
of the earlier stages of the filtration \eqref{eq:filtered-chains}, a subcomplex
of \(F_\ell C_\bullet(X)\) that is zero at \(\ell=0\), and \(\gr_\ell\) for the
associated graded piece at \(\ell\), the degreewise quotient of the stage at
\(\ell\) by the union of the earlier stages, and put
\[
\MC_{\bullet,\ell}(X)=\gr_\ell C_\bullet(X)
  =F_\ell C_\bullet(X)\big/F_{<\ell}C_\bullet(X),
\qquad
\MH_{n,\ell}(X)=H_n\bigl(\MC_{\bullet,\ell}(X)\bigr).
\]
We call \(n\) the degree of \(\MH_{n,\ell}(X)\) and \(\ell\) its length, a
bigrading by degree and by length. Write
\(\kappa_{\ell'}\colon F_{\ell'}C_\bullet(X)\hookrightarrow F_{<\ell}C_\bullet(X)\), for
\(\ell'<\ell\), and \(\iota_\ell\colon F_{<\ell}C_\bullet(X)\hookrightarrow F_\ell C_\bullet(X)\)
for the inclusions and
\(\pi_\ell\colon F_\ell C_\bullet(X)\twoheadrightarrow\MC_{\bullet,\ell}(X)\) for the quotient
map. The categories, the objects and the maps around this definition assemble,
for every \(\ell'<\ell\), into the commuting diagram
\begin{equation*}\label{eq:mag-pipeline}
\begin{tikzcd}[row sep=2em, column sep=3em, ampersand replacement=\&,
               labels={inner sep=2.2pt}]
  \SSet\colon
    \&[-2.5em]\Nv(X)^{\le\ell'}
      \arrow[r, hook, "\dnote{\mathbf{x}\,\mapsto\,\mathbf{x}}"']
      \arrow[d, "C_\bullet"']
    \&\textstyle\bigcup_{\ell''<\ell}\Nv(X)^{\le\ell''}
      \arrow[r, hook, "\dnote{\mathbf{x}\,\mapsto\,\mathbf{x}}"']
      \arrow[d, "C_\bullet"]
    \&\Nv(X)^{\le\ell}
      \arrow[d, "C_\bullet"]
    \&\\
  \mathrm{Ch}(\Ab)\colon
    \&[-2.5em]F_{\ell'}C_\bullet(X)
      \arrow[r, hook, "\kappa_{\ell'}", "\dnote{z\,\mapsto\,z}"']
      \arrow[d, "H_n"']
    \&F_{<\ell}C_\bullet(X)
      \arrow[r, hook, "\iota_\ell", "\dnote{z\,\mapsto\,z}"']
      \arrow[d, "H_n"]
    \&[0.9em]F_\ell C_\bullet(X)
      \arrow[r, two heads, "\pi_\ell", "\dnote{z\,\mapsto\,z+F_{<\ell}}"']
      \arrow[d, "H_n"]
    \&[1.1em]\MC_{\bullet,\ell}(X)
      \arrow[d, "H_n"]\\
  \Ab\colon
    \&[-2.5em]\PH_n(X)(\ell')
      \arrow[r, "(\kappa_{\ell'})_*", "\dnote{[z]\,\mapsto\,[z]}"']
    \&H_n\bigl(F_{<\ell}C_\bullet(X)\bigr)
      \arrow[r, "(\iota_\ell)_*", "\dnote{[z]\,\mapsto\,[z]}"']
    \&\PH_n(X)(\ell)
      \arrow[r, "(\pi_\ell)_*", "\dnote{[z]\,\mapsto\,[\pi_\ell z]}"']
    \&\MH_{n,\ell}(X),
\end{tikzcd}
\end{equation*}
\looseness=-1
whose rows live in the categories named on their left and whose columns are
the stages before \(\ell\), their union, the stage at \(\ell\) and its quotient,
taken at the level of chains. The top squares commute because all four maps
send a basis simplex to itself, the lower squares because each bottom map is
the \(H_n\)-image of the chain map above it.
\end{definition}

\begin{proposition}\label{prop:mc-explicit}
Let \(X\) be an object of \(\Met\), let \(n\ge 0\) and \(\ell\in[0,\infty)\), and write
\(x\preceq y\preceq z\) when \(d(x,y)+d(y,z)=d(x,z)\). The group
\(\MC_{n,\ell}(X)\) is free abelian on the nondegenerate tuples
\(\mathbf{x}\in\Nv(X)_n\) with \(\lambda\mathbf{x}=\ell\), and the differential
induced on it by the differential \(\partial\) of \(C_\bullet(X)\) is
\begin{equation}\label{eq:differential}
\partial=\sum_{i=1}^{n-1}(-1)^{i}\partial_i\colon
\MC_{n,\ell}(X)\longrightarrow\MC_{n-1,\ell}(X),
\end{equation}
\looseness=-1
where \(\partial_i\) deletes the entry \(x_i\) when
\(x_{i-1}\preceq x_i\preceq x_{i+1}\) and sends \(\mathbf{x}\) to \(0\)
otherwise. The outer faces \(d_0\) and \(d_n\) vanish on the associated graded,
and the differential has bidegree \((-1,0)\).
\end{proposition}

\looseness=-1
The complex of \Cref{prop:mc-explicit} is the magnitude complex of Leinster and
Shulman, so \(\MH_{n,\ell}\) is their magnitude homology
\cite[\S3, \S5]{LeinsterShulman2021MagnitudeHomology}, which generalises the
magnitude homology of graphs of
Hepworth and Willerton~\cite[Def.~2.4]{HepworthWillerton2017Categorifying}. Their route to it is a
different one: they take the Hochschild homology of the enriched nerve with
coefficients in the diagram constant at the unit, which exists because \(\V\) is
semicartesian (\Cref{lem:base}). Either way the differential sees of the
triangle inequality only whether it holds as an equality
\cite{KanetaYoshinaga2021MagnitudeHomologyOrderComplexes}, and nothing of how
slack it is. Two consequences follow, both visible in \Cref{sec:examples}:

\begin{corollary}\label{cor:no-betweenness}
Let \(X\) be an object of \(\Met\), let \(n\ge 0\) and let \(\ell\in[0,\infty)\).
Suppose no \(x,y,z\in\Ob X\) with \(y\notin\{x,z\}\) satisfy
\(x\preceq y\preceq z\). Then the differential \eqref{eq:differential} is zero
and \(\MH_{n,\ell}(X)\) is free abelian on the nondegenerate tuples of length
\(\ell\).
\end{corollary}

\begin{lemma}\label{lem:digon}
Let \(X\) be an object of \(\Met\) and let \(x\neq y\) in \(\Ob X\) with
\(e=d(x,y)<\infty\). The chain \(\langle x,y\rangle+\langle y,x\rangle\) is a
cycle of \(F_eC_1(X)\) and the boundary of
\(\langle x,y,x\rangle\in F_{2e}C_2(X)\).
\end{lemma}

\looseness=-1
Every pair at finite distance therefore carries a degree-one cycle born no later
than that distance and dead no later than twice it, and these are the short bars
filling the degree-one row of \Cref{fig:barcodes}. On a space all of whose
distances equal \(e\) every such class is born at \(e\), since \(F_\ell C_1(X)\)
vanishes below that length, and is dead by \(2e\).

\begin{lemma}\label{lem:complex}
Let \(F\colon X\to Y\) be nonexpansive, with \(X\) and \(Y\) objects of
\(\Met\).
\begin{enumerate}
\item\label{itm:mc-functor} Sending \(\mathbf{x}\) to \(\langle fx_0,\dots,fx_n\rangle\) when
  \(d(fx_{i-1},fx_i)=d(x_{i-1},x_i)\) for all \(1\le i\le n\), and to \(0\)
  otherwise,
  defines \(\MC_{n,\ell}(F)\) and makes
  \(\MC_{\bullet,\ell}\colon\Met\to\mathrm{Ch}(\Ab)\) and
  \(\MH_{n,\ell}\colon\Met\to\Ab\) functors, one such pair for each degree
  \(n\ge0\) and each length \(\ell\in[0,\infty)\).
\item\label{itm:mc-split} For \(X\) in \(\Met\) the finite-distance equivalence relation of
  \Cref{sec:base} reads \(x\sim y\) if and only if \(d(x,y)<\infty\). Write
  \(\pi_0X=\Ob X/{\sim}\) for the set of classes and, for \(c\in\pi_0X\),
  write \(X_c\) for the subspace with \(\Ob X_c=c\) and
  \(d_{X_c}=d|_{c\times c}\). Then each \(X_c\) has finite distances and
  \begin{equation}\label{eq:splitting}
  \MC_{\bullet,\ell}(X)\;\cong\;\bigoplus_{c\in\pi_0X}\MC_{\bullet,\ell}(X_c)
  \end{equation}
  as complexes, for every \(\ell\in[0,\infty)\).
\end{enumerate}
\end{lemma}

\looseness=-1
The space underlying the free algebra on \(A\) is \(T_UA\), an object of
\(\Met\) by \Cref{thm:free}\Cref{itm:forgetfulfunctor}, hence separated and
symmetric as \Cref{prop:mc-explicit} and \eqref{eq:splitting} require, and
\eqref{eq:splitting} splits off its components, so every length below is finite
and the smaller base \(\Vfin\) suffices.

\section{The Comparison Theorem}\label{sec:comparison}
Magnitude homology is graded by length and the persistence module is filtered by
it, and \Cref{def:magnitude} makes the first the associated graded of the
second. For a field \(\mathbb{F}\) write
\(\PH_n(X;\mathbb{F})(\ell)=H_n\bigl(F_\ell C_\bullet(X)\otimes_{\mathbb{Z}}\mathbb{F}\bigr)\),
with the transition maps induced by the inclusions of the filtration, and
\(\MH_{n,\ell}(X;\mathbb{F})=H_n\bigl(\MC_{\bullet,\ell}(X)\otimes_{\mathbb{Z}}\mathbb{F}\bigr)\).
The barcode with its births and deaths, and tameness, are those of
\Cref{sec:barcodes}.

\begin{theorem}\label{thm:les}
Let \(X\) be an object of \(\Met\), let \(\ell\in[0,\infty)\) and put
\(\PH_n(X)(\ell^-)=\colim_{\ell'<\ell}\PH_n(X)(\ell')\). There is a long exact
sequence of abelian groups, natural in \(X\),
\[
\begin{tikzcd}[row sep=2.8em, column sep=2.6em]
  \cdots \arrow[r, "{\partial_{n+1,\ell}}"]
    & \PH_n(X)(\ell^-) \arrow[r, "{\tau_{n,\ell}}"]
    & \PH_n(X)(\ell) \arrow[r, "{\rho_{n,\ell}}"]
      \arrow[d, phantom, ""{name=Z, coordinate}]
    & \MH_{n,\ell}(X) \arrow[dll, les={Z}{\partial_{n,\ell}}]\\
  & \PH_{n-1}(X)(\ell^-) \arrow[r, "{\tau_{n-1,\ell}}"]
    & \PH_{n-1}(X)(\ell) \arrow[r, "{\rho_{n-1,\ell}}"]
    & \cdots,
\end{tikzcd}
\]
in which \(\tau_{n,\ell}\) is the transition map of the persistence module,
\(\rho_{n,\ell}\) is induced by the quotient onto the associated graded, and
\(\partial_{n,\ell}\) is the connecting homomorphism.
\end{theorem}

\begin{proof}
Fix \(\ell\in[0,\infty)\), and let
\(\iota_\ell\colon F_{<\ell}C_\bullet(X)\hookrightarrow F_\ell C_\bullet(X)\) and
\(\pi_\ell\colon F_\ell C_\bullet(X)\twoheadrightarrow\MC_{\bullet,\ell}(X)\) be the
inclusion and the quotient map of \Cref{def:magnitude}.
\begin{itemize}
\item \textbf{The Sequence Is Short Exact.} Both maps are chain maps, since
  \(\partial\) carries \(F_{<\ell}C_\bullet(X)\) and \(F_\ell C_\bullet(X)\) into
  themselves by \Cref{lem:filtration}. In each degree \(n\) the group
  \(F_{<\ell}C_n(X)\) is spanned by a subset of the basis that spans
  \(F_\ell C_n(X)\), so \(\iota_\ell\) is injective, \(\pi_\ell\) is surjective
  as a quotient map, and \(\ker\pi_\ell=F_{<\ell}C_n(X)=\operatorname{im}\iota_\ell\)
  by the definition of the quotient. Hence
  \begin{equation}\label{eq:ses}
    0\longrightarrow F_{<\ell}C_\bullet(X)
      \xhookrightarrow{\ \iota_\ell\ }F_\ell C_\bullet(X)
      \xtwoheadrightarrow{\ \pi_\ell\ }\MC_{\bullet,\ell}(X)\longrightarrow 0
  \end{equation}
  is a short exact sequence of chain complexes of free abelian groups.
\item \textbf{The Long Exact Sequence.} A short exact sequence of complexes has
  a long exact homology sequence, with connecting homomorphism supplied by the
  snake lemma \cite[Lemma 12.13.6]{StacksProject}. Write \(\varphi_*\) for the
  map that a chain map \(\varphi\) induces on homology. Applied to
  \eqref{eq:ses} it runs through the groups
  \(H_n\bigl(F_{<\ell}C_\bullet(X)\bigr)\), \(H_n\bigl(F_\ell C_\bullet(X)\bigr)\)
  and \(H_n\bigl(\MC_{\bullet,\ell}(X)\bigr)\), in this order and for
  every \(n\), along the maps \((\iota_\ell)_*\), \((\pi_\ell)_*\) and the
  connecting map \(\partial_{n,\ell}\), which lowers the homological degree by one.
\item \textbf{The Three Terms.} The middle and right terms are the two
  invariants by definition, and the left term is a filtered colimit. Write
  \(\kappa_{\ell'}\colon F_{\ell'}C_\bullet(X)\hookrightarrow F_{<\ell}C_\bullet(X)\) for the
  inclusion of the stage \(\ell'<\ell\). The complexes \(F_{\ell'}C_\bullet(X)\)
  and these inclusions form a filtered system whose colimit is the union
  \(F_{<\ell}C_\bullet(X)\) of \Cref{def:magnitude}, and homology commutes with
  filtered colimits \cite[Lemma 10.8.8]{StacksProject}. The three terms are
  therefore
  \begin{equation}\label{eq:les-terms}
  \begin{aligned}
    H_n\bigl(F_\ell C_\bullet(X)\bigr)&=\PH_n(X)(\ell),\qquad
    H_n\bigl(\MC_{\bullet,\ell}(X)\bigr)=\MH_{n,\ell}(X),\\
    H_n\bigl(F_{<\ell}C_\bullet(X)\bigr)
      &=\operatorname*{colim}_{\ell'<\ell}H_n\bigl(F_{\ell'}C_\bullet(X)\bigr)
       =\operatorname*{colim}_{\ell'<\ell}\PH_n(X)(\ell')=\PH_n(X)(\ell^-),
  \end{aligned}
  \end{equation}
  the second line by that same lemma. Putting
  \(\rho_{n,\ell}=(\pi_\ell)_*\) and \(\tau_{n,\ell}=(\iota_\ell)_*\) in degree
  \(n\) turns this sequence into the sequence of the statement, term by term.
\item \textbf{\(\tau_{n,\ell}\) Is the Transition Map.} For \(\ell'<\ell\) the
  inclusions compose, \(\iota_\ell\circ\kappa_{\ell'}=\iota_{\ell',\ell}\), the
  right-hand side being the inclusion of \(F_{\ell'}C_\bullet(X)\) into
  \(F_\ell C_\bullet(X)\). Applying \(H_n\) and \eqref{eq:les-terms},
  \[
    \tau_{n,\ell}\circ(\kappa_{\ell'})_*=(\iota_{\ell',\ell})_*
      =\PH_n(X)(\ell'\le\ell),
  \]
  which is the transition map from \(\ell'\) to \(\ell\). Every class of the
  colimit is \((\kappa_{\ell'})_*[z]\) for some \(\ell'<\ell\) and some cycle
  \(z\in F_{\ell'}C_n(X)\), so \(\tau_{n,\ell}\) is the map induced on the
  colimit by these transition maps, and on classes it is given by the
  assignment \([z]\mapsto[z]\).
\item \textbf{Naturality in \(X\).} Let \(F\colon X\to Y\) be nonexpansive. By
  \Cref{lem:filtration}\Cref{itm:filt-map} the chain map \(C_\bullet(F)\) preserves both stages
  of the filtration, and by \Cref{lem:complex}\Cref{itm:mc-functor} it descends to the associated
  graded, so it carries \eqref{eq:ses} for \(X\) to \eqref{eq:ses} for \(Y\).
  The long exact sequence of a short exact sequence is natural for such a map
  \cite[Lemmas 12.5.18 and 12.13.6]{StacksProject}. On the transition maps this naturality
  is the commutativity of the left square below, which the right square
  evaluates on one class of the upper left corner,
  \[
  \begin{tikzcd}[row sep=2.2em, column sep=2.6em]
    \PH_n(X)(\ell^-) \arrow[r, "{\tau^X_{n,\ell}}"]
      \arrow[d, "{\PH_n(F)(\ell^-)}"']
      & \PH_n(X)(\ell) \arrow[d, "{\PH_n(F)(\ell)}"]\\
    \PH_n(Y)(\ell^-) \arrow[r, "{\tau^Y_{n,\ell}}"']
      & \PH_n(Y)(\ell),
  \end{tikzcd}
  \qquad
  \begin{tikzcd}[row sep=2.2em, column sep=2.2em]
    {[z]} \arrow[r, maps to] \arrow[d, maps to]
      & {[z]} \arrow[d, maps to]\\
    {\bigl[C_\bullet(F)z\bigr]} \arrow[r, maps to]
      & {\bigl[C_\bullet(F)z\bigr]},
  \end{tikzcd}
  \]
  where both composites send \([z]\) to \(\bigl[C_\bullet(F)z\bigr]\), so the
  left square commutes. The same argument on \(\rho\) and on \(\partial\) makes
  the whole sequence natural in \(X\).
\item \looseness=-1\textbf{Coefficients in a Field.} By \Cref{prop:mc-explicit} the group
  \(\MC_{n,\ell}(X)\) is free abelian in every degree, so \eqref{eq:ses} splits
  degreewise and stays exact after \(-\otimes_{\mathbb{Z}}\mathbb{F}\) for every
  field \(\mathbb{F}\), and the first four parts above, run on the tensored
  sequence, return the sequence with \(\mathbb{F}\)-coefficients.\qedhere
\end{itemize}
\end{proof}

\begin{corollary}\label{cor:critical}
Let \(X\) be an object of \(\Met\), let \(n\ge 0\) and let \(\ell\in[0,\infty)\).
If \(\MH_{n,\ell}(X)=0=\MH_{n+1,\ell}(X)\), then
\(\PH_n(X)(\ell^-)\to\PH_n(X)(\ell)\) is an isomorphism. For finite \(X\) over a
field \(\mathbb{F}\), where every bar is an interval \([b,b')\) or \([b,\infty)\), every
birth and every death in the barcode of \(\PH_n(X;\mathbb{F})\) therefore lies in
\(\{\ell\mid\MH_{n,\ell}(X;\mathbb{F})\neq 0\}\cup\{\ell\mid\MH_{n+1,\ell}(X;\mathbb{F})\neq 0\}\).
\end{corollary}

\begin{figure}[t]
\centering
\scalebox{0.82}{\begin{tikzpicture}[x=1.0421mm,y=1mm,font=\small]

\begin{scope}[on background layer]
  \fill[accentpale] (0,-3.0) rectangle (28,39.0);
  \fill[accentpale] (56,-3.0) rectangle (84,39.0);
\end{scope}

\draw[complement,line width=0.7pt,dash pattern=on 0.9pt off 1.3pt,
      preaction={draw,line width=2.0pt,white},
      -{Straight Barb[length=1.5mm,width=1.6mm]}] (0,19.9)--(0,31.9);
\draw[accent,line width=0.7pt,dash pattern=on 0.9pt off 1.3pt,
      preaction={draw,line width=2.0pt,white},
      -{Straight Barb[length=1.5mm,width=1.6mm]}] (28,9.9)--(28,23.9);
\draw[accent,line width=0.7pt,dash pattern=on 0.9pt off 1.3pt,
      preaction={draw,line width=2.0pt,white},
      -{Straight Barb[length=1.5mm,width=1.6mm]}] (56,14.9)--(56,23.9);

\node[left,inner sep=2pt,font=\footnotesize] at (-1.5,33.5) {$\MH_{0,\ell}(X;\mathbb{F})$};
\draw[accentsoft,line width=0.4pt] (0,33.5)--(112,33.5);
\fill[complement] (-1.05,32.45) rectangle (1.05,34.55);
\node[above right,inner sep=1.4pt,font=\footnotesize,complement] at (1.4,34.2) {$\mathbb{F}^{3}$};
\draw[complement,line width=0.7pt] (84,32.2)--(84,34.8);
\node[above,inner sep=1.8pt,font=\footnotesize,complement] at (84,34.8) {$0$};

\node[left,inner sep=2pt,font=\footnotesize] at (-1.5,25.5) {$\MH_{1,\ell}(X;\mathbb{F})$};
\draw[accentsoft,line width=0.4pt] (0,25.5)--(112,25.5);
\fill[accent] (26.95,24.45) rectangle (29.05,26.55);
\node[above,inner sep=2.0pt,font=\footnotesize,accent] at (28,26.6) {$\mathbb{F}^{2}$};
\fill[accent] (54.95,24.45) rectangle (57.05,26.55);
\node[above,inner sep=2.0pt,font=\footnotesize,accent] at (56,26.6) {$\mathbb{F}^{2}$};
\draw[accent,line width=0.7pt] (84,24.2)--(84,26.8);
\node[above,inner sep=1.8pt,font=\footnotesize,accent] at (84,26.8) {$0$};

\node[left,inner sep=2pt,font=\footnotesize] at (-1.5,13.5) {$\PH_{0}(X;\mathbb{F})$};
\draw[ink,line width=0.9pt] (0,18.5)--(105,18.5);
\draw[ink,line width=0.9pt,dotted] (105.8,18.5)--(112,18.5);
\fill[complement] (0,18.5) circle (0.95);
\draw[ink,line width=0.9pt] (0,13.5)--(56,13.5);
\fill[complement] (0,13.5) circle (0.95);
\draw[accent,line width=0.7pt,fill=white] (56,13.5) circle (0.95);
\draw[ink,line width=0.9pt] (0,8.5)--(28,8.5);
\fill[complement] (0,8.5) circle (0.95);
\draw[accent,line width=0.7pt,fill=white] (28,8.5) circle (0.95);

\node[inner sep=1pt,font=\footnotesize,inkmute] at (14,2.4) {$B_{0}$};
\node[inner sep=1pt,font=\footnotesize,inkmute] at (42,2.4) {$B_{1}$};
\node[inner sep=1pt,font=\footnotesize,inkmute] at (70,2.4) {$B_{2}$};
\node[inner sep=1pt,font=\footnotesize,inkmute] at (98,2.4) {$B_{3}$};

\draw[ink,line width=0.5pt] (0,-3.0)--(112,-3.0);
\foreach \bx/\bl in {0/0,28/1,56/2,84/3}{%
  \draw[ink,line width=0.5pt] (\bx,-3.0)--(\bx,-4.2);
  \node[below,inner sep=1.5pt,font=\footnotesize] at (\bx,-4.2) {$\bl$};}
\node[right,inner sep=2pt,font=\footnotesize] at (112,-3.0) {$\ell$};
\node[below,inner sep=1.5pt,font=\footnotesize,inkmute] at (42,-6.4)
     {$\Lambda=\{0,1,2,3\}$};

\begin{scope}[font=\footnotesize]
  \fill[complement] (1.2,47.5) circle (1.0);
  \node[right,inner sep=2.2pt,inkmute] at (2.6,47.5) {Birth};
  \draw[accent,line width=0.7pt,fill=white] (19.0,47.5) circle (1.0);
  \node[right,inner sep=2.2pt,inkmute] at (20.4,47.5) {Death};
\end{scope}

\draw[accentsoft,line width=0.4pt,rounded corners=1.5pt] (62,41.5) rectangle (113,55.5);
\begin{scope}[shift={(66,50.5)}]
  \draw[inkmute,line width=0.5pt] (0,0)--(42,0);
  \foreach \px/\pl in {0/{x_{0}},14/{x_{1}},42/{x_{2}}}{%
    \fill[ink] (\px,0) circle (0.9);
    \node[above,inner sep=2.2pt,font=\footnotesize] at (\px,0.9) {$\pl$};}
  \draw[inkmute,line width=0.4pt] (0,-1.8)--(0,-3.6)--(14,-3.6)--(14,-1.8);
  \node[below,inner sep=1.0pt,font=\footnotesize] at (7,-3.6) {$1$};
  \draw[inkmute,line width=0.4pt] (14,-1.8)--(14,-3.6)--(42,-3.6)--(42,-1.8);
  \node[below,inner sep=1.0pt,font=\footnotesize] at (28,-3.6) {$2$};
\end{scope}
\end{tikzpicture}}
\caption{\looseness=-1 The running example of the proof of \Cref{cor:critical}: \(X\) in
\(\Met\) with \(\Ob X=\{x_0,x_1,x_2\}\), \(d(x_0,x_1)=1\), \(d(x_1,x_2)=2\),
\(d(x_0,x_2)=3\), degree \(n=0\), over a field \(\mathbb{F}\). The bands are
the blocks \(B_0,\dots,B_3\) of the set \(\Lambda=\{0,1,2,3\}\) of lengths of
nondegenerate simplices of degree at most \(n+1\), on each of which the
module is constant. The three half-open bars are the barcode \([0,\infty)\),
\([0,2)\), \([0,1)\), a filled endpoint a birth, a hollow one a death, the
dotted tail the bar that never dies. Squares mark the two magnitude groups where they do not vanish,
with dimensions, and each arrow runs from a barcode endpoint to the group
whose nonvanishing it forces: the birth at \(0\) to \(\MH_{0,\ell}\), the
deaths to \(\MH_{1,\ell}\). At \(\ell=3\) a block ends but the barcode does
not change, and ticks record the value \(0\), the differential of
\Cref{prop:mc-explicit} carrying \(\MC_{2,3}(X)\) isomorphically onto
\(\MC_{1,3}(X)\).}
\label{fig:critical}
\end{figure}
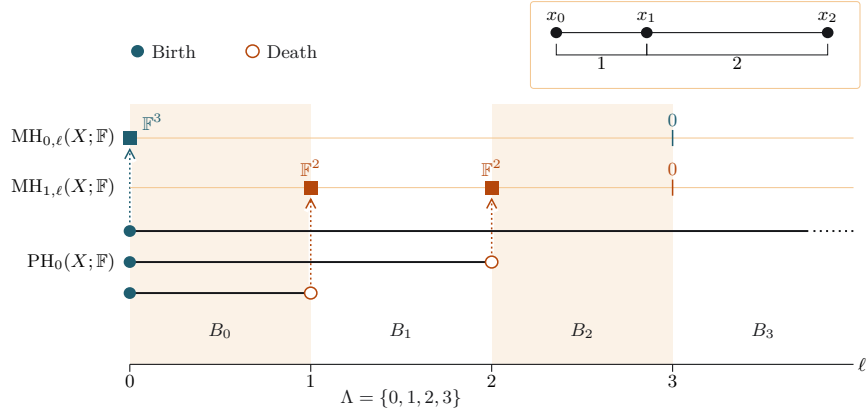

\looseness=-1
For finite \(X\) over a field the sequence counts endpoints with multiplicity:
exactness at \(\PH_n(X)(\ell)\), at \(\MH_{n,\ell}(X)\) and at
\(\PH_{n-1}(X)(\ell^-)\) gives
\(\dim\MH_{n,\ell}(X;\mathbb{F})=\dim\operatorname{coker}\tau_{n,\ell}
+\dim\ker\tau_{n-1,\ell}\), with \(\tau_{-1,\ell}=0\), and the table
\eqref{eq:endpoint-table} in the proof of \Cref{cor:critical} evaluates the
first term as the number of bars of \(\PH_n(X;\mathbb{F})\) born at \(\ell\) and
the second as the number of bars of \(\PH_{n-1}(X;\mathbb{F})\) dying at \(\ell\).

\section{Persistence Modules, Barcodes and Stability}\label{sec:barcodes}
The indexing category is \(\Vfin^{\op}\), the
objects \([0,\infty)\) of \Cref{sec:base} with one morphism \(\ell\to\ell'\)
when \(\ell\le\ell'\) and none otherwise. Let \(\A\) be an abelian category. A
\emph{persistence module} in \(\A\) is a functor \(M\colon\Vfin^{\op}\to\A\)
\cite{ChazalDeSilvaGlisseOudot2016Structure,Oudot2015PersistenceTheory}, with
\emph{transition maps} \(M(\ell\le\ell')\colon M\ell\to M\ell'\). Persistence
modules and their natural transformations form the functor category
\(\A^{\Vfin^{\op}}\), abelian because \(\A\) is abelian and kernels and
cokernels are computed pointwise in \(\ell\)
\cite[p.~614]{BubenikScott2014Categorification}.
Fix a field \(\mathbb{F}\) and take \(\A=\Vect_{\mathbb{F}}\). A subset \(J\subseteq[0,\infty)\) is
an \emph{interval} if it is nonempty and \(\ell\le m\le\ell'\) with
\(\ell,\ell'\in J\) forces \(m\in J\). The \emph{interval module}
\(\mathbb{F}_J\colon\Vfin^{\op}\to\Vect_{\mathbb{F}}\) has \((\mathbb{F}_J)\ell=\mathbb{F}\) for \(\ell\in J\) and \(0\)
otherwise, with every transition map between nonzero terms the identity of
\(\mathbb{F}\). Call \(M\) \emph{tame} if the dimension \(\dim_{\mathbb{F}}M\ell\) is finite for every
\(\ell\in[0,\infty)\).

\begin{theorem}\label{thm:barcode}
Let \(\mathbb{F}\) be a field and let \(M\colon\Vfin^{\op}\to\Vect_{\mathbb{F}}\) be tame. Then
\(M\cong\bigoplus_{J\in\mathcal{B}}\mathbb{F}_J\) for a multiset \(\mathcal{B}\) of
intervals, and \(\mathcal{B}\) is determined by \(M\) up to bijection.
\end{theorem}

\looseness=-1
It is the \emph{barcode} \cite{CarlssonZomorodianCollinsGuibas2005PersistenceBarcodes,Ghrist2008Barcodes}
of \(M\), and its elements are the \emph{bars}. For
a bar \(J\in\mathcal{B}\) call \(\inf J\) a \emph{birth} of \(M\), and call
\(\sup J\), when it is finite, a \emph{death} of \(M\). Tameness is a genuine
restriction: \(T_UA\) is infinite for many theories of interest, and
\(F_\ell C_n\) then has infinite rank. What is finite, and what each example
checks, is the dimension of the homology in each degree, immediate for the
finite spaces of \Cref{sec:examples}.

\begin{definition}\label{def:interleaving}
\looseness=-1
Let \(M,N\colon\Vfin^{\op}\to\A\) and let \(\delta\in[0,\infty)\). A
\emph{\(\delta\)-interleaving}
\cite{ChazalCohenSteinerGlisseGuibasOudot2009Proximity} of \(M\) and \(N\) is a
pair of families
\(\varphi_\ell\colon M\ell\to N(\ell+\delta)\) and
\(\psi_\ell\colon N\ell\to M(\ell+\delta)\), natural in \(\ell\), with
\(\psi_{\ell+\delta}\varphi_\ell=M(\ell\le\ell+2\delta)\) and
\(\varphi_{\ell+\delta}\psi_\ell=N(\ell\le\ell+2\delta)\) for all \(\ell\). Put
\(d_I(M,N)=\inf\{\delta\mid M\ \text{and}\ N\ \text{are}\ \delta\text{-interleaved}\}\),
with \(\inf\varnothing=\infty\). This is the interleaving distance
\(d_I\colon\Ob(\A^{\Vfin^{\op}})\times\Ob(\A^{\Vfin^{\op}})\to\Rp\) of
\cite[Ch.~4]{ChazalDeSilvaGlisseOudot2016Structure}, there indexed by the real
line and here \(\Vfin^{\op}\).
\end{definition}

\begin{proposition}\label{prop:interleaving-vcat}
\looseness=-1
The function \(d_I\) vanishes on the diagonal, is symmetric, and satisfies the
triangle inequality. Persistence modules in \(\A\) therefore carry the structure
of a symmetric \(\V\)-category in the sense of \Cref{sec:base}. It is separated
only after quotienting by the relation \(d_I=0\).
\end{proposition}

\begin{theorem}\label{thm:distortion}
Let \(X\) and \(Y\) be objects of \(\Met\) with the same underlying set and
\(d_X\le d_Y+\delta\) as well as \(d_Y\le d_X+\delta\) pointwise, an inequality
in \(\Rp\) that avoids subtracting infinities. Then \(\PH_n(X)\) and
\(\PH_n(Y)\) are \((n+1)\delta\)-interleaved, so
\(d_I(\PH_nX,\PH_nY)\le(n+1)\delta\).
\end{theorem}

\looseness=-1
The factor \(n+1\) counts the distances that \(\lambda\) adds in simplicial
degree \(n+1\), the top degree needed to compute \(H_n\). \Cref{sec:examples} exhibits a
perturbation whose effect on the barcode grows linearly across the computed
degrees. In particular \(\PH_n\) is Lipschitz but not nonexpansive, so it is
not a \(\V\)-functor for the distance \(\sup|d_X-d_Y|\). For tame modules over
a field, extended by zero to \(\mathbb{R}\), the interleaving distance is the
bottleneck distance \cite{CohenSteinerEdelsbrunnerHarer2007Stability} of the
barcodes of \Cref{thm:barcode}
\cite{BauerLesnick2015InducedMatchings,Lesnick2015InterleavingDistance},
\cite[Ch.~5]{ChazalDeSilvaGlisseOudot2016Structure}, and
\Cref{sec:examples} reports it as such.

\section{The Composite Construction}\label{sec:composite}
\begin{theorem}\label{thm:composite}
Let \(U\) be a quantitative equational theory over a signature
\((\Omega,\ar)\), let \(T_U\colon\Met\to\Met\) be the metric term monad it
induces by \Cref{thm:free}, and let \(n\ge 0\).
\begin{enumerate}
\item\label{itm:comp-ph} The assignment \(A\mapsto\PH_n(T_UA)\) is a functor
  \(\PH^U_n\colon\Met\to\Ab^{\Vfin^{\op}}\), namely the composite
  \[
  \PH^U_n\colon\Met\xrightarrow{\;T_U\;}\Met
       \xrightarrow{\;\Nv(-)^{\le-}\;}\SSet^{\Vfin^{\op}}
       \xrightarrow{\;H_nC_\bullet\;}\Ab^{\Vfin^{\op}}.
  \]
\item\label{itm:comp-gr} For every \(A\) in \(\Met\) and every \(\ell\in[0,\infty)\) the homology
  of \(\gr_\ell C_\bullet(T_UA)\) is the magnitude homology
  \(\MH_{n,\ell}(T_UA)\) of the free algebra, naturally in \(A\).
\end{enumerate}
\end{theorem}

\looseness=-1
Over a field \(\mathbb{F}\) put \(\PH^U_n(A;\mathbb{F})(\ell)=H_n\bigl(F_\ell C_\bullet(T_UA)\otimes_{\mathbb{Z}}\mathbb{F}\bigr)\).
When this module is tame, \Cref{thm:barcode} attaches to it a barcode, the
invariant we compute, a function of the axioms of \(U\) and of the generating
space \(A\).

\begin{theorem}\label{thm:theorymap}
Let \(U\subseteq U'\) be quantitative equational theories over one signature
\((\Omega,\ar)\), with units \(\eta^U\) and \(\eta^{U'}\) for the two metric
term monads \(T_U\) and \(T_{U'}\) of \Cref{thm:free}.
\begin{enumerate}
\item\label{itm:tmap-comparison} Every algebra of \(\Mod(U')\) lies in \(\Mod(U)\), so \(\Free_{U'}A\)
  is an object of \(\Mod(U)\) for every \(A\) in \(\Met\). There is one and
  only one natural transformation \(q\colon T_U\Rightarrow T_{U'}\) such that
  \(q_A\circ\eta^U_A=\eta^{U'}_A\) holds and \(q_A\colon T_UA\to T_{U'}A\)
  underlies a morphism \(\Free_UA\to\Free_{U'}A\) of \(\Mod(U)\), for every
  \(A\) in \(\Met\). It is a morphism of monads, and each of its components is
  a surjective and nonexpansive map \(q_A\colon T_UA\twoheadrightarrow T_{U'}A\) of \(\Met\).
\item\label{itm:tmap-homology} Applying \(\PH_n\) to \(q\) gives a natural transformation
  \(\PH_n(q)\colon\PH^U_n\Rightarrow\PH^{U'}_n\).
\item\label{itm:boundedInterleaving} Suppose \(q_A\) is bijective and
  \(d^{T_UA}(t,s)\le d^{T_{U'}A}(q_At,q_As)+\delta\) for all \(t,s\in T_UA\).
  Then \(d_I\bigl(\PH^U_n(A),\PH^{U'}_n(A)\bigr)\le(n+1)\delta\).
\end{enumerate}
\end{theorem}

The map \(q_A\) is the semantic effect of adding axioms, and part
\Cref{itm:boundedInterleaving} bounds the motion of the barcode under it. Two
quantities name the two sides of that bound.

\begin{definition}\label{def:impact}
Let \(U\subseteq U'\) be quantitative equational theories over one signature,
let \(A\) be an object of \(\Met\), let \(q_A\colon T_UA\twoheadrightarrow T_{U'}A\)
be the component of \Cref{thm:theorymap} and let \(n\ge0\). If \(q_A\) is
bijective, the \emph{semantic perturbation radius} of the inclusion at \(A\) is
\[
\delta_A(U\subseteq U')=\inf\bigl\{\delta\in[0,\infty)\bigm|
  d^{T_UA}(t,s)\le d^{T_{U'}A}(q_At,q_As)+\delta\ \text{for all}\ t,s\in T_UA\bigr\},
\]
with \(\inf\varnothing=\infty\), and otherwise \(\delta_A(U\subseteq U')=\infty\).
The \emph{degree-\(n\) semantic impact} of the inclusion at \(A\) is
\(\operatorname{Impact}_{n,A}(U\subseteq U')=d_I\bigl(\PH^U_n(A),\PH^{U'}_n(A)\bigr)\).
\end{definition}

\looseness=-1
Every \(\delta\) in the set of \Cref{def:impact} satisfies the hypotheses of
part \Cref{itm:boundedInterleaving}, so
\begin{equation}\label{eq:impact}
\operatorname{Impact}_{n,A}(U\subseteq U')\le(n+1)\,\delta_A(U\subseteq U').
\end{equation}
That set is an intersection of rays \([c,\infty)\), so a finite radius is
attained. Read contrapositively, \(\operatorname{Impact}_{n,A}(U\subseteq U')
>(n+1)\varepsilon\) forces \(\delta_A(U\subseteq U')>\varepsilon\): the axioms
of \(U'\smallsetminus U\) identify two terms of \(T_UA\) or shorten some
distance by more than \(\varepsilon\). A barcode movement of size \(\beta\) at
a bijective \(q_A\) therefore certifies that some distance shrinks by at least
\(\beta/(n+1)\), and a bijective isometry \(q_A\) has radius and impact \(0\).
This is the sense in which the invariant measures axiomatic strength: it
measures how far the added axioms compress the free algebra, a metric property
of the semantics and not the length or difficulty of the derivations behind it,
functorially in \(A\) by \Cref{thm:composite}, robustly by \Cref{thm:distortion},
and by finite linear algebra when both free algebras are finite. It is a test,
not a decision procedure: no converse to \eqref{eq:impact} is proved.

\section{Computations}\label{sec:examples}
All barcodes below are over \(\mathbb{F}_2\) and were computed from the
definitions in \Cref{sec:nerve}. Three independent checks run
against it:\looseness=-1\relax{} \(\PH_0\) against single-linkage clustering, \(\MH_{1,\ell}\)
against its generators, the pairs at distance \(\ell\) with no point between
them, and the barcode endpoints against the critical lengths of
\Cref{cor:critical}. \Cref{fig:pipeline} draws every stage of the pipeline for
the free semilattice on the four marks of a Golomb ruler.

\begin{figure}[t]
\centering
\begin{tikzpicture}[x=1mm,y=1mm,font=\small,
  pt/.style={circle,fill=ink,inner sep=0pt,minimum size=1.5mm},
  sub/.style={circle,draw=inkmute,fill=white,line width=0.35pt,
              inner sep=0pt,minimum size=4.6mm,font=\footnotesize},
  tag/.style={anchor=north west,font=\bfseries\color{accent}},
  blurb/.style={anchor=north,font=\small\color{ink},align=center,text width=30mm},
  op/.style={-{Stealth[length=1.5mm]},accent,line width=0.55pt}]

\useasboundingbox (-7,-97) rectangle (133,-2);

\begin{scope}[shift={(0,0)}]
  \node[tag] at (0,-4) {(a) E1};
  \begin{scope}[shift={(19,-25)},scale=1.15]
    \foreach \k in {0,...,5}{\coordinate (v\k) at (\k*60:11.5);}
    \draw[inkmute,line width=0.35pt,opacity=0.5] (v0)--(v3) (v1)--(v4) (v2)--(v5);
    \draw[complement,line width=0.5pt] (v0)--(v2) (v1)--(v3) (v2)--(v4)
                                       (v3)--(v5) (v4)--(v0) (v5)--(v1);
    \draw[accent,line width=0.85pt] (v0)--(v1)--(v2)--(v3)--(v4)--(v5)--cycle;
    \foreach \k in {0,...,5}{\node[pt] at (v\k) {};}
    \node[font=\footnotesize,accent,fill=white,inner sep=1.1pt]
      at (30:13.6) {$1$};
    \node[font=\footnotesize,complement,fill=white,inner sep=1.1pt]
      at (2.9,5.0) {$\sqrt3$};
    \node[font=\footnotesize,inkmute,fill=white,inner sep=1.1pt]
      at (3.4,0) {$2$};
  \end{scope}
  \node[blurb] at (16,-42) {No operations: $T_UA=A$.};
\end{scope}

\begin{scope}[shift={(33.5,0)}]
  \node[tag] at (0,-4) {(b) E2};
  \begin{scope}[shift={(9,-36)},scale=1.15]
    \node[sub] (A)  at (0,0)    {$a$};
    \node[sub] (B)  at (9,0)    {$b$};
    \node[sub] (C)  at (18,0)   {$c$};
    \node[sub] (AB) at (0,9)    {$ab$};
    \node[sub] (AC) at (9,9)    {$ac$};
    \node[sub] (BC) at (18,9)   {$bc$};
    \node[sub] (T)  at (9,18)   {$abc$};
    \draw[inkmute,line width=0.3pt,opacity=0.7]
      (A)--(AB) (B)--(AB) (A)--(AC) (C)--(AC) (B)--(BC) (C)--(BC)
      (AB)--(T) (AC)--(T) (BC)--(T);
    \draw[op,bend left=30] (A) to
      node[midway,left,font=\footnotesize,accent,inner sep=1pt] {$\vee$} (AB);
  \end{scope}
  \node[blurb] at (16,-42) {Join is union. Every $d_H=1$.};
\end{scope}

\begin{scope}[shift={(68,0)}]
  \node[tag] at (0,-4) {(c) E3};
  \begin{scope}[shift={(6,-15)},scale=1.1]
    \node[sub] (X)  at (0,0)   {$x$};
    \node[sub] (Y)  at (0,-10) {$y$};
    \node[sub] (XY) at (20,-5) {$x{\vee}y$};
    \draw[op,bend right=34] (X) to (XY);
    \draw[op,bend right=24] (Y) to
      node[pos=0.45,below,font=\footnotesize,accent,inner sep=1.4pt] {$\vee$} (XY.205);
    \draw[complement,line width=0.55pt,bend left=34,
          {Stealth[length=1.2mm]}-{Stealth[length=1.2mm]}] (X) to
      node[midway,above,font=\footnotesize,complement,inner sep=1.4pt]
      {$\le\tfrac25$} (XY);
  \end{scope}
  \begin{scope}[shift={(3,-35)}]
    \draw[ink,line width=0.4pt] (0,0)--(29,0);
    \foreach \x/\l in {0/0,20.5/1,29/{\sqrt2}}
      {\draw[ink,line width=0.4pt] (\x,0)--(\x,-1.0);
       \node[font=\footnotesize,below,inner sep=1.2pt] at (\x,-1.0) {$\l$};}
    \foreach \x in {20.5,29}{\fill[inkmute] (\x,3.0) circle[radius=0.7];}
    \foreach \x in {8.2,16.4}{\fill[accent] (\x,6.6) circle[radius=0.7];}
    \node[font=\footnotesize,inkmute,anchor=east,inner sep=1.2pt] at (-0.6,3.0) {$U$};
    \node[font=\footnotesize,accent,anchor=east,inner sep=1.2pt] at (-0.6,6.6) {$U\mathrlap{'}$};
    \draw[{Stealth[length=1.2mm]}-,accent,line width=0.45pt] (8.6,5.9)--(20.1,3.7);
    \draw[{Stealth[length=1.2mm]}-,accent,line width=0.45pt] (16.8,5.9)--(28.6,3.7);
  \end{scope}
  \node[blurb] at (16,-42) {The axiom shortens $1,\sqrt2$ to $\tfrac25,\tfrac45$.};
\end{scope}

\begin{scope}[shift={(100.5,0)}]
  \node[tag] at (0,-4) {(d) E4};
  \begin{scope}[shift={(19,-25)},scale=1.15]
    \foreach \k in {0,...,5}{\coordinate (w\k) at (\k*60:10.5);}
    \foreach \k in {0,...,5}{\fill[accentbright,opacity=0.3] (w\k) circle[radius=2.2];}
    \draw[inkmute,line width=0.4pt,opacity=0.45,dashed]
      (w0)--(w1)--(w2)--(w3)--(w4)--(w5)--cycle;
    \draw[accent,line width=0.7pt]
      ($(w0)+(1.0,0.5)$)--($(w1)+(-0.8,0.9)$)--($(w2)+(0.6,-1.0)$)--
      ($(w3)+(-1.0,-0.4)$)--($(w4)+(0.9,-0.7)$)--($(w5)+(-0.5,1.0)$)--cycle;
    \foreach \k in {0,...,5}{\node[pt] at (w\k) {};}
    \node[font=\footnotesize,accentbright,anchor=east] at (-9.8,-7.6) {$\pm\delta$};
  \end{scope}
  \node[blurb] at (16,-42) {Same points, metric moved by $\delta$.};
\end{scope}

\begin{scope}[shift={(4,-81)}]
  \input{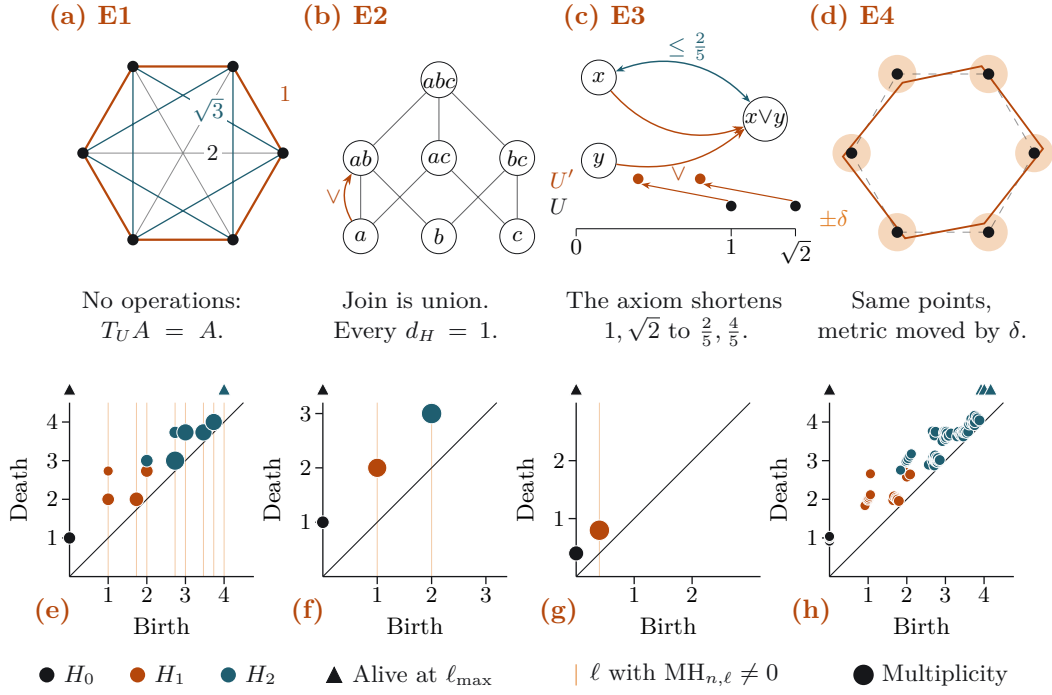}
\end{scope}
\foreach \x/\l in {0/e,33.5/f,67/g,100.5/h}
  {\node[anchor=east,font=\bfseries\color{accent}] at (\x+5.5,-85.6) {(\l)};}

\begin{scope}[shift={(1,-94)}]
  \foreach \k/\c/\n in {0/inkmute/0,1/accent/1,2/complement/2}
    {\fill[\c] (\k*12,0) circle[radius=1.0];
     \node[anchor=west,font=\small\color{ink},inner sep=1.6pt]
       at (\k*12+1.6,0) {$H_{\n}$};}
  \filldraw[inkmute] (37,-1.0)--(39,-1.0)--(38,1.0)--cycle;
  \node[anchor=west,font=\small\color{ink},inner sep=1.6pt] at (39.8,0)
    {Alive at $\ell_{\max}$};
  \draw[accentbright,line width=0.6pt,opacity=0.6] (70,-1.4)--(70,1.4);
  \node[anchor=west,font=\small\color{ink},inner sep=1.6pt] at (71.2,0)
    {$\ell$ with $\mathrm{MH}_{n,\ell}\neq0$};
  \fill[inkmute] (108,0) circle[radius=1.35];
  \node[anchor=west,font=\small\color{ink},inner sep=1.6pt] at (109.8,0)
    {Multiplicity};
\end{scope}

\end{tikzpicture}
\caption{\looseness=-1 The four computations of \Cref{sec:examples}, each algebra above its
diagram. \textbf{(a)}~The hexagon, distance classes \(1,\sqrt3,2\) in three
weights. The signature is empty, so the free algebra is the space.
\textbf{(b)}~The free quantitative semilattice on the equilateral triangle:
the seven nonempty subsets under inclusion, the join is union.
\textbf{(c)}~The axiom \(\vdash x\vee y=_{2/5}x\), read in every model as
\(d(x\vee y,x)\le\tfrac25\), caps the cost of a join, drawn for two elements
\(x\) and \(y\). With commutativity it bounds \(d(x\vee y,y)\) too, and the
triangle inequality through the join gives \(d(x,y)\le\tfrac45\), so the free
semilattice on the square collapses its Hausdorff distances \(1,\sqrt2\), row
\(U\), to \(\tfrac25,\tfrac45\), row \(U'\). \textbf{(d)}~The hexagon of
\textbf{(a)}, each distance moved by at most \(\delta=1/10\): dashed the original,
haloes of radius \(\delta\). \textbf{(e)--(h)}~One point per class, each on a
vertical rule by \Cref{cor:critical}. In \textbf{(e)} every rule carries one,
an equality. The uniform metric of \textbf{(b)} puts every class on
\((n,n+1)\), the axiom of \textbf{(c)} pulls the diagram towards the origin,
the perturbation of \textbf{(d)} scatters \textbf{(e)} by an amount growing
with the degree, and \textbf{(h)} omits its multiplicities, too many to draw
legibly.}
\label{fig:spaces}
\end{figure}

\begin{example}\label{ex:hexagon}
Let \(X\) be the six vertices of a regular hexagon on the unit circle, with the
Euclidean metric, so the distances are \(1,\sqrt3,2\). In degrees \(n\le2\) the
group \(\MH_{n,\ell}(X)\) is nonzero for
\(\ell\in\{0,1,\sqrt3,2,1+\sqrt3,3,2\sqrt3,2+\sqrt3,4\}\) and for no other
length, and every endpoint of the barcode of \(\PH_n(X)\) in those degrees is
one of these lengths or a length carried by degree \(3\), as
\Cref{cor:critical} requires. This is endpoint detection at work.
\end{example}

\begin{example}\label{ex:semilattice}
\looseness=-1
Let \(U\) be the theory of quantitative semilattices, with one binary
\(\vee\) and the axioms of associativity, commutativity and idempotence, and let
\(A\) be the equilateral triangle of side length \(1\). Then \(T_UA\) is the set of
seven nonempty subsets of \(A\) under the Hausdorff metric, and that metric is
\emph{uniform}: \(d(S,T)=1\) whenever \(S\neq T\). Indeed \(d\le 1\) since all
distances in \(A\) are \(1\), and for \(S\neq T\) a point of the symmetric
difference has distance \(1\) to the other set.

\looseness=-1
Consequently \(\lambda\langle S_0,\dots,S_n\rangle=n\) on nondegenerate tuples,
and \(S\preceq S'\preceq S''\) never holds, since \(1+1\neq 1\), so
\Cref{cor:no-betweenness} gives
\(\MH_{n,\ell}(T_UA)=\mathbb{Z}^{7\cdot 6^{n}}\) for \(\ell=n\) and \(0\)
otherwise, while \Cref{lem:digon} produces \(21\) of the degree-one classes, one per pair. The complex \(C_\bullet(T_UA)\) is the normalised chain complex of the
\(0\)-coskeletal simplicial set on seven vertices, which is contractible, so
writing \(b_n\) for the number of finite bars of \(\PH_n\) we get
\(b_n=7\cdot6^{n}-b_{n-1}\) with \(b_{-1}=1\), whence \(b_n=6^{n+1}\). Every bar
is \([n,n+1)\), and \(\PH_0\) carries one further essential bar. For the uniform
space on \(m\) points the same argument gives \((m-1)^{n+1}\) bars, a barcode in
closed form in every degree.
\end{example}

\begin{example}\label{ex:theorymap}
\looseness=-1
Let \(A\) be the unit square and \(U\) again the theory of quantitative
semilattices, so \(T_UA\) has fifteen points and distances \(\{1,\sqrt2\}\). Let
\(U'=U\cup\{(\varnothing,\ x\vee y=_{2/5}x)\}\). Closing the Hausdorff metric
under the axiom, nonexpansiveness of \(\vee\) and the triangle inequality
gives \(T_{U'}A\) with the same fifteen points and distances \(\{2/5,4/5\}\), so
\(q_A\) is bijective and \(\delta_A(U\subseteq U')=\sqrt2-2/5\), the largest
amount by which the axiom shortens a distance. The number of finite bars in
degree \(1\) drops from \(196\) to \(146\), and the bottleneck distance between
the two barcodes, the impact \(\operatorname{Impact}_{n,A}(U\subseteq U')\) of
the single axiom, is \(1/2\) in degrees \(0\) and \(1\), well inside the bounds
\((n+1)\delta_A(U\subseteq U')\) of \eqref{eq:impact}, which are \(1.01\) and
\(2.03\) here.
\end{example}

\begin{example}\label{ex:distortion}
Moving every distance of the hexagon of \Cref{ex:hexagon} by at most
\(\delta=1/10\), in a pattern for which the triangle inequality survives, leaves the number of bars unchanged and moves the barcode by a
bottleneck distance of \(0.082\), \(0.164\) and \(0.246\) in degrees \(0,1,2\).
These sit below the bound \((n+1)\delta\) of \Cref{thm:distortion} and grow
linearly in the degree, exceeding \(\delta\) from degree \(1\) on, so the factor
\(n+1\) of \Cref{thm:distortion} cannot be replaced by a constant below
\(2.45\) uniformly in the degree, the linear growth matches it, and \(\PH_n\)
is Lipschitz but not nonexpansive.
\end{example}

\begin{figure}[tb]
\centering
\input{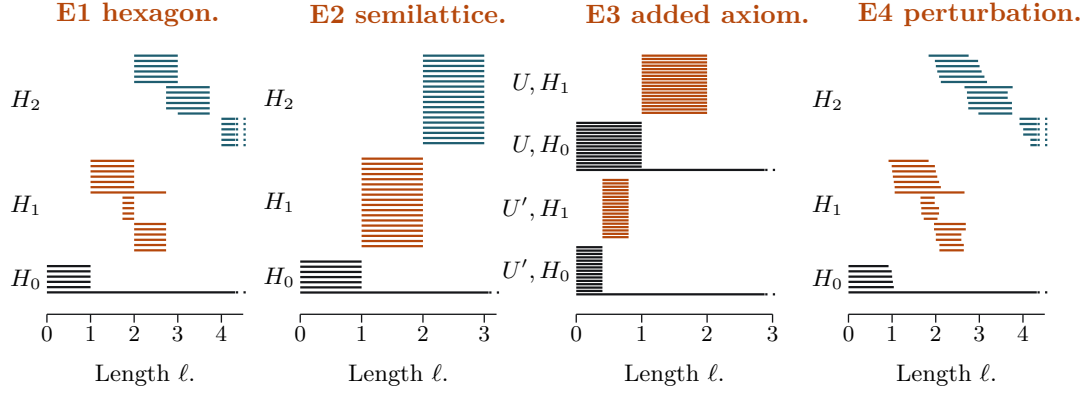}
\caption{Barcodes of \(\PH_n\) over \(\mathbb{F}_2\) for
\Cref{ex:hexagon,ex:semilattice,ex:theorymap,ex:distortion}, the four
computations of \Cref{fig:spaces}, each panel drawing at most \(18\) bars per
degree, the longest first, colour coding the degree. In E3 the two blocks are
the barcodes before and after the axiom \(\vdash x\vee y=_{2/5}x\), in degrees
\(0\) and \(1\). A bar dotted at its right end has not died by the largest drawn
length \(\ell_{\max}\).}
\label{fig:barcodes}
\end{figure}

\section{What the Invariant Sees}\label{sec:outlook}
\looseness=-1
The barcode is an invariant of the metric semantics a theory presents.
Four limits delimit it, each the consequence of a result proved above:

\begin{enumerate}
\item\looseness=-1 The invariant is blind to slack in the triangle inequality at a single length.
By \Cref{prop:mc-explicit} a face survives on the associated graded only when
\(d(x,y)+d(y,z)=d(x,z)\) holds, so \(\MH_{n,\ell}\) cannot distinguish a space
where the equation almost holds from one where it fails badly. The
filtration repairs this in one direction only: \Cref{thm:les} lets a class
persist across an interval of lengths, and the length of the resulting bar is
the quantitative datum that a single graded piece cannot carry.

\item\looseness=-1 The invariant is not stable in the sense usually asked of persistent homology
\cite{ChazalCohenSteinerGlisseGuibasOudot2009Proximity,ChazalDeSilvaOudot2014PersistenceStability,CohenSteinerEdelsbrunnerHarer2007Stability}.
By \Cref{ex:distortion} no constant below \(2.45\) covers every degree, and the
length accumulates \(n\) summands in degree \(n\), so this is no artefact of the
proof. A degree-independent bound needs a different filtering function, and
the natural candidate, the maximum of the consecutive distances, is the
Vietoris--Rips filtration, whose associated graded is not the magnitude complex.

\item\looseness=-1 The invariant is a function of the theory as much as of the space.
By \Cref{ex:theorymap} the bound \eqref{eq:impact} is far from attained, \(1/2\)
against \(1.01\) and \(2.03\), and whether the gap closes for some class of
theories we leave open. Which monads on \(\Met\) present a variety is settled,
the \(1\)-basic ones \cite[Thm.~40]{Adamek2026OneBasicMonads}, but which
\emph{morphisms} of such monads come from an inclusion of theories, what
\Cref{thm:theorymap} needs, is not, nor over relational structures
\cite[\S6]{JurkaMiliusUrbat2024AlgebraicReasoningRelationalStructures}.

\item\looseness=-1 The invariant is a test, not a decision procedure:
\Cref{thm:theorymap}\Cref{itm:boundedInterleaving} and \eqref{eq:impact} need
\(q_A\) bijective, \Cref{thm:barcode} needs tameness and a field, no converse
to \eqref{eq:impact} is proved, so equal barcodes are not shown to force
isometric free algebras, and what is measured is metric-semantic strength, not
a proof-theoretic quantity.
\end{enumerate}

\looseness=-1
Two problems follow. The filtration of \Cref{sec:nerve} is bounded below,
exhaustive for finite distances and of finitely many values per degree on a
finite space, so it carries a spectral sequence with \(E^1\) page the magnitude
homology, of which \Cref{thm:les} is the exact-couple shadow, abutting to the
homology of the contractible full length nerve, so every magnitude class of
positive degree supports or receives a nonzero differential on some page, and
which page we do not know. And tameness does not follow from finiteness of
the generating space, \(T_UA\) being infinite for many \(U\). A criterion for it
stated on \(U\) rather than on \(T_UA\) would open the barcode
to the theories of chief interest, among them the interpolative barycentric
algebras of Mardare, Panangaden and Plotkin~\cite{MardarePanangadenPlotkin2016QuantitativeAlgebraicReasoning}.

\bibliography{main}

\clearpage
\appendix

\section{Proofs of \texorpdfstring{\Cref{sec:base}}{Section 2}}\label{app:base}

\begin{proof}[Proof of \Cref{lem:base}]~
\begin{itemize}
\item \textbf{Every Diagram Commutes.} Call a
  category \emph{thin} when each of its hom-sets has at most one element. In a
  thin category two morphisms with the same source and the same target are two
  elements of one such set, hence equal. The category \(\V\) is thin by definition. Let \(\varphi\) and \(\psi\) be the composites along two directed paths with source \(a\) and target \(b\) through
  \(c_{1},\dots,c_{m}\) and \(e_{1},\dots,e_{n}\) of \(\Ob\V\). Both are
  elements of \(\V(a,b)\), so they agree:
  \[
  \begin{tikzcd}[row sep=0.9em, column sep=1.7em, cramped]
    & c_{1} \arrow[r] & \cdots \arrow[r] & c_{m} \arrow[dr] & \\
    a \arrow[ur] \arrow[dr]
      \arrow[rrrr, phantom, "{\dnote{\varphi=\psi\ \text{in}\ \V(a,b)}}"]
    & & & & b\mathrlap{\,.} \\
    & e_{1} \arrow[r] & \cdots \arrow[r] & e_{n} \arrow[ur] &
  \end{tikzcd}
  \]
  Every diagram in \(\V\), and every diagram of maps between sets with at most
  one element, therefore commutes. This settles every coherence
  condition: the associativity and unit axioms of a \(\V\)-category,
  the pentagon, the unit triangle, the symmetry and the hexagon, the naturality
  squares of \(\alpha,\lambda,\rho,\sigma\), the commutation of a cone leg with
  the diagram, and the naturality squares of the adjunction in the part
  \emph{Closed} below are all such diagrams.
\item \textbf{Limits and Colimits.} Let \(D\colon J\to\V\) be a diagram with
  \(J\) small and put \(A=\{Dj\mid j\in\Ob J\}\subseteq\Rp\). A cone on \(D\)
  with vertex \(p\) is a family \(\pi_{j}\colon p\to Dj\) with
  \(Du\circ\pi_{j}=\pi_{j'}\) for every \(u\colon j\to j'\) of \(J\). The family
  exists when \(p\ge Dj\) for all \(j\), that is, when
  \(p\ge\sup A\), and the commutation is free. Dually a cocone
  with vertex \(q\) is a family \(\iota_{j}\colon Dj\to q\), which exists
  when \(Dj\ge q\) for all \(j\), that is, when \(\inf A\ge q\). Both
  bounds are such vertices, since \(\sup A\ge Dj\ge\inf A\) for every
  \(j\), and comparison morphisms are unique:
  \[
  \begin{tikzcd}[row sep=1.6em, column sep=2.0em]
    & & p \arrow[d, dashed, "\exists!", "{\dnote{p\ge\sup A}}"'] & & \\[0.7em]
    & & \sup A
        \arrow[dll, bend right=22, "\pi_{j}"']
        \arrow[d, "\pi_{j'}"']
        \arrow[drr, bend left=22, "\pi_{j''}"]
    & & \\
    Dj \arrow[rr, "Du"'] & & Dj' \arrow[rr, "Du'"'] & & Dj'' \\
    & & \inf A \arrow[d, dashed, "\exists!"', "{\dnote{\inf A\ge q}}"] & & \\[0.7em]
    & & q\mathrlap{\,.} & &
    \arrow[from=3-1, to=4-3, bend right=22, "\iota_{j}"']
    \arrow[from=3-3, to=4-3, "\iota_{j'}"']
    \arrow[from=3-5, to=4-3, bend left=22, "\iota_{j''}"]
  \end{tikzcd}
  \]
  The middle row stands for an arbitrary \(D\). Only the set \(A\) enters the
  argument, and not the shape of \(J\). Hence \(\lim_J D=\sup A\) and
  \(\colim_J D=\inf A\). Every subset of \(\Rp\) has a supremum and an infimum in
  \(\Rp\), so \(\V\) is complete and cocomplete. The empty \(J\) gives the
  terminal object \(\sup\varnothing=0\) and the initial object
  \(\inf\varnothing=\infty\).
\item \textbf{Symmetric Monoidal.} A morphism \((u,v)\colon(a,b)\to(a',b')\) of
  \(\V\times\V\) is a pair of relations \(a\ge a'\) and \(b\ge b'\), which add to
  \(a+b\ge a'+b'\). The unique morphism this names is \(u\otimes v\).
  Preservation of identities and of composites is an equation between parallel
  morphisms, so the tensor \(\otimes\colon\V\times\V\to\V\) is a functor of two
  variables:
  \[
  \begin{tikzcd}[row sep=1.3em, column sep=2.8em]
    {(a,b)} \arrow[r, "{(u,v)}"] \arrow[d, mapsto, "\otimes"'] & {(a',b')} \arrow[d, mapsto, "\otimes"]\\
    {a+b} \arrow[r, "{u\otimes v}"'] & {a'+b'}\mathrlap{\,.}
  \end{tikzcd}
  \]
  For all \(a,b,c\in\Ob\V\) the equalities \((a+b)+c=a+(b+c)\), \(0+a=a=a+0\)
  and \(a+b=b+a\) hold in \(\Rp\). The associator \(\alpha_{a,b,c}\), the
  unitors \(\lambda_{a},\rho_{a}\) and the flip map \(\sigma_{a,b}\)
  have equal source and target and are the identity of that object. They are
  natural, each naturality square being a diagram in \(\V\). The pentagon and
  the unit triangle read
  \[
  \begin{gathered}
  \hbox{\small\begin{tikzcd}[row sep=2.0em, column sep=0.4em, cramped]
    & & {(a\otimes b)\otimes(c\otimes d)} \arrow[drr, "{\alpha_{a,b,c\otimes d}}"] & & \\
    {((a\otimes b)\otimes c)\otimes d}
      \arrow[urr, "{\alpha_{a\otimes b,c,d}}"]
      \arrow[dr, "{\alpha_{a,b,c}\otimes 1_{d}}"']
    & & {\dnote{a+b+c+d}} & & {a\otimes(b\otimes(c\otimes d))} \\
    & {(a\otimes(b\otimes c))\otimes d} \arrow[rr, "{\alpha_{a,b\otimes c,d}}"']
    & & {a\otimes((b\otimes c)\otimes d),} \arrow[ur, "{1_{a}\otimes\alpha_{b,c,d}}"'] &
  \end{tikzcd}}\\[1.2ex]
  \begin{tikzcd}[row sep=1.2em, column sep=2.0em, cramped]
    {(a\otimes I)\otimes b} \arrow[rr, "{\alpha_{a,I,b}}"] \arrow[ddr, "{\rho_{a}\otimes 1_{b}}"']
    & & {a\otimes(I\otimes b)} \arrow[ddl, "{1_{a}\otimes\lambda_{b}}"] \\
    & {\dnote{a+b}} & \\
    & {a\otimes b}. &
  \end{tikzcd}
  \end{gathered}
  \]
  Every arrow is the identity of its vertex, so both commute. The same holds
  for \(\sigma_{b,a}\circ\sigma_{a,b}=1_{a\otimes b}\) and for the hexagon.
  Hence \((\V,\otimes,I)\) is symmetric monoidal.
\item \textbf{Closed.} Fix \(a\in\Ob\V\). We claim the adjunction in the form
  of the equivalence
  \begin{equation}\label{eq:closed}
    a+b\ge c\quad\text{if and only if}\quad b\ge c\ominus a,
    \qquad b,c\in\Ob\V,
  \end{equation}
  and we check it separately in each of the three defining cases of
  \eqref{eq:ominus}:
  \begin{itemize}
  \item \(a<c=\infty\), so \(a<\infty\) and \(c\ominus a=\infty\): each side
    holds if and only if \(b=\infty\).
  \item \(c\le a\), so \(c\ominus a=0\): both \(a+b\ge a\ge c\) and \(b\ge 0\)
    hold for every \(b\).
  \item \(a<c<\infty\), so \(a<\infty\) and \(c\ominus a=c-a\): for \(b<\infty\)
    subtract \(a\) on both sides, and for \(b=\infty\) both sides hold.
  \end{itemize}
  \textit{Three consequences of \eqref{eq:closed}.} First,
  \(c\mapsto c\ominus a\) is a functor: taking \(b=c\ominus a\) in
  \eqref{eq:closed} gives \(a+(c\ominus a)\ge c\), so \(c\ge c'\) yields
  \(a+(c\ominus a)\ge c'\), and \eqref{eq:closed} at \(c'\) with the same \(b\)
  returns \(c\ominus a\ge c'\ominus a\). Second, \(\V(a\otimes b,c)\) and
  \(\V(b,c\ominus a)\) are both \(\{\bullet\}\) or both \(\varnothing\), so
  \[
    \theta_{b,c}\colon\V(a\otimes b,c)\longrightarrow\V(b,c\ominus a),
    \qquad\bullet\longmapsto\bullet,
  \]
  is a bijection, the empty one when both sets are empty. Third, all four
  vertices of
  \[
  \begin{tikzcd}[row sep=1.7em, column sep=3.4em]
    {\V(a\otimes b,c)} \arrow[r, "{\theta_{b,c}}"]
      \arrow[d, "{f\,\mapsto\,t\circ f\circ(1_{a}\otimes w)}"']
    & {\V(b,c\ominus a)} \arrow[d, "{g\,\mapsto\,(t\ominus a)\circ g\circ w}"]\\
    {\V(a\otimes b',c')} \arrow[r, "{\theta_{b',c'}}"']
    & {\V(b',c'\ominus a)}
  \end{tikzcd}
  \]
  have at most one element, so the square commutes for every \(w\colon b'\to b\)
  and \(t\colon c\to c'\), and \(\theta\) is natural in both variables. Hence
  \(\V\) is closed, with the adjunction
  \[
  \begin{tikzcd}[column sep=3.6em]
    \V \arrow[r, bend left=28, "{a\otimes(-)}"] \arrow[r, phantom, "{\scriptstyle\bot}"]
    & \V. \arrow[l, bend left=28, "{(-)\ominus a}"]
  \end{tikzcd}
  \]
\item \textbf{Semicartesian.} Every \(a\in\Ob\V\) satisfies \(a\ge 0\), so
  \(\V(a,I)=\{\bullet\}\): a morphism \(a\to I\) exists, and it is unique. The unit \(I=0\) is therefore terminal, in agreement with the empty limit
  \(\sup\varnothing=0\) computed in \emph{Limits and Colimits}, and \(\V\) is
  semicartesian.
\item \textbf{Tensor Against Product.} Let \(J\) be the discrete category on two
  objects and \(D\colon J\to\V\) the diagram with values \(a\) and \(b\). Then
  \(A=\{a,b\}\), so the second item computes the product of two objects as their
  supremum, \(a\times b=\lim_J D=\sup\{a,b\}=\max(a,b)\), with the limit cone, a
  competing pair \(u,v\) and the comparison morphism drawn below:
  \[
  \begin{tikzcd}[row sep=1.9em, column sep=2.4em]
    & p \arrow[dl, bend right=18, "u"'] \arrow[dr, bend left=18, "v"] \arrow[d, dashed, "\exists!"] & \\
    a & {\max(a,b)} \arrow[l, "\pi_{a}"] \arrow[r, "\pi_{b}"'] & b\mathrlap{\,.}
  \end{tikzcd}
  \]
  The projections exist because \(\max(a,b)\ge a\) and \(\max(a,b)\ge b\). A
  pair \(u,v\) as displayed says \(p\ge a\) and \(p\ge b\), hence
  \(p\ge\max(a,b)\), and this relation is the unique comparison morphism.
  At \(a=b=1\) this gives \(1\times 1=\max(1,1)=1\), whereas \(1\otimes 1=1+1=2\).
  The objects \(1\) and \(2\) are not isomorphic in \(\V\), since an isomorphism
  would require \(1\ge 2\).\qedhere
\end{itemize}
\end{proof}

\section{Proofs of \texorpdfstring{\Cref{sec:theories}}{Section 3}}\label{app:theories}

\begin{proof}[Proof of \Cref{lem:term-monad}]~
\begin{enumerate}
\item Realise the coproducts of \eqref{eq:terms}, writing
  \(\iota_0,\iota_1\) for the injections of a binary disjoint union:
  \[
    X\sqcup Y=\bigl(\{0\}\times X\bigr)\cup\bigl(\{1\}\times Y\bigr),
    \qquad
    \bigsqcup_{n\in\mathbb{N}}\Omega_n\times Z^{n}
      =\bigcup_{n\in\mathbb{N}}\Omega_n\times Z^{n},
  \]
  and abbreviating \(\eta_Ww=\iota_0w\) as well as
  \(f(t_1,\dots,t_n)=\iota_1(f,t_1,\dots,t_n)\).
  \begin{itemize}
  \item \textbf{The Summands Are Pairwise Disjoint.} For the binary coproduct and
    for \(m\ne n\),
    \begin{align*}
      \bigl(\{0\}\times X\bigr)\cap\bigl(\{1\}\times Y\bigr)
        &=\bigl(\{0\}\cap\{1\}\bigr)\times\bigl(X\cap Y\bigr)=\varnothing,\\
      \bigl(\Omega_m\times Z^{m}\bigr)\cap\bigl(\Omega_n\times Z^{n}\bigr)
        &=\ar^{-1}\bigl(\{m\}\cap\{n\}\bigr)\times\bigl(Z^{m}\cap Z^{n}\bigr)
         =\varnothing.
    \end{align*}
  \item \textbf{The Grades Increase.} \(T^{0}_\Omega W=\varnothing\subseteq T^{1}_\Omega W\) starts an
    induction. \(Z\mapsto W\sqcup\bigsqcup_n\Omega_n\times Z^n\)
    of \eqref{eq:terms} is monotone in \(Z\), so for \(k\ge 1\) the inclusion
    \(T^{k-1}_\Omega W\subseteq T^{k}_\Omega W\) gives
    \[
      T^{k}_\Omega W
        =W\sqcup\bigsqcup_{n}\Omega_n\times\bigl(T^{k-1}_\Omega W\bigr)^{n}
        \subseteq W\sqcup\bigsqcup_{n}\Omega_n\times\bigl(T^{k}_\Omega W\bigr)^{n}
        =T^{k+1}_\Omega W.
    \]
  \item \textbf{Finite Tuples at a Single Grade.}
    \((t_1,\dots,t_n)\in(T_\Omega W)^{n}\) has \(t_i\in T^{k_i}_\Omega W\) for
    some \(k_i\), and \(n\) is finite, so all \(t_i\) lie in
    \(T^{k}_\Omega W\) for \(k=\max\{k_1,\dots,k_n\}\), with \(\max\varnothing=0\).
    Hence
    \begin{equation}\label{eq:tuples}
      \bigcup_{k\ge 0}\bigl(T^{k}_\Omega W\bigr)^{n}
        =\Bigl(\bigcup_{k\ge 0}T^{k}_\Omega W\Bigr)^{n}
        =\bigl(T_\Omega W\bigr)^{n}.
    \end{equation}
  \item \textbf{Union Is a Fixed Point of \eqref{eq:terms}.} Since \(T^{0}_\Omega W=\varnothing\), the union telescopes,
    \begin{align*}
      T_\Omega W
        &=\bigcup_{k\ge 0}T^{k+1}_\Omega W
         =\bigcup_{k\ge 0}\Bigl(\iota_0W\cup\bigcup_{n}
             \iota_1\bigl(\Omega_n\times(T^{k}_\Omega W)^{n}\bigr)\Bigr)\\
        &=\iota_0W\cup\bigcup_{n}\iota_1\Bigl(\Omega_n\times
             \bigcup_{k\ge 0}\bigl(T^{k}_\Omega W\bigr)^{n}\Bigr)
         \overset{\eqref{eq:tuples}}{=}
         \iota_0W\cup\bigcup_{n}\iota_1\bigl(\Omega_n\times(T_\Omega W)^{n}\bigr)\\
        &=W\sqcup\bigsqcup_{n\in\mathbb{N}}\Omega_n\times\bigl(T_\Omega W\bigr)^{n}.
      \end{align*}
  \item \textbf{Terms Are Read in One Way.} By the fixed
    point, the sets
    \(\iota_0W\) and \(\iota_1(\Omega_n\times(T_\Omega W)^{n})\), \(n\in\mathbb{N}\),
    are pairwise disjoint and cover \(T_\Omega W\), and \(\iota_0\), \(\iota_1\)
    are injective. So \(\eta_W\) is injective, no \(\eta_Ww\) is an
    \(f(t_1,\dots,t_n)\), and \(f(t_1,\dots,t_n)=f'(t'_1,\dots,t'_{n'})\) forces
    \(f=f'\), then \(n=\ar f=n'\), then \(t_i=t'_i\). The operations
    \(f^{T_\Omega W}(t_1,\dots,t_n)=f(t_1,\dots,t_n)\) are therefore well defined
    and make \(T_\Omega W\) an \(\Omega\)-algebra.
  \item \textbf{The Recursion \eqref{eq:extension} Terminates.} Let the grade of
    \(t\) be the least \(k\) with \(t\in T^{k}_\Omega W\). A term of grade
    \(k+1\) lying in the right summand of (\ref{eq:terms}) is \(f(t_1,\dots,t_n)\) with
    \(t_i\in T^{k}_\Omega W\), so each \(t_i\) has grade at most \(k\) and the
    second clause of \eqref{eq:extension} calls \(g^\sharp\) at strictly smaller
    grade. At the base, \(T^{1}_\Omega W=W\sqcup\Omega_0\), because \(\varnothing^{0}\) is a singleton
    and \(\varnothing^{n}=\varnothing\) for \(n\ge 1\), and there the first clause
    fixes \(g^\sharp\) on \(\iota_0W\) while the second at \(n=0\) gives
    \(g^\sharp\bigl(f^{T_\Omega W}(\,)\bigr)=f^{A}(\,)\) with no call at all. Induction on the grade
    therefore defines \(g^\sharp\) on all of \(T_\Omega W\) and determines it
    uniquely.
  \end{itemize}
\item \textbf{Freeness.} The two clauses of \eqref{eq:extension} say that \(g^\sharp\circ\eta_W=g\)
  and, since \(f^{T_\Omega W}=f\), that \(g^\sharp\) is an
  \(\Omega\)-homomorphism. An \(\Omega\)-homomorphism \(h\) with
  \(h\circ\eta_W=g\) satisfies both clauses, so \(h=g^\sharp\) by
  \Cref{itm:tm-grades}. That universal property is the freeness asserted here.
\item \textbf{Functoriality.} Each \(\leadsto\) below is an instance of \Cref{itm:tm-free}: its
  two sides are \(\Omega\)-homomorphisms, being a \((-)^\sharp\), an identity or
  a composite. The line above it compares them after \(\eta_W\).
  \begin{align}
    T_\Omega r\circ\eta_W
      &=(\eta_{W'}\circ r)^\sharp\circ\eta_W
       \overset{\eqref{eq:extension}}{=}\eta_{W'}\circ r ,\label{eq:eta-nat}\\
    T_\Omega r\,f(t_1,\dots,t_n)
      &\overset{\eqref{eq:extension}}{=}
       f^{T_\Omega W'}\bigl(T_\Omega r\,t_1,\dots,T_\Omega r\,t_n\bigr)
       \overset{\cref{itm:tm-grades}}{=}
       f\bigl(T_\Omega r\,t_1,\dots,T_\Omega r\,t_n\bigr).\\\nonumber
    T_\Omega 1_W\circ\eta_W
      &\overset{\eqref{eq:eta-nat}}{=}\eta_W=1_{T_\Omega W}\circ\eta_W,\\
    \leadsto\ T_\Omega 1_W
      &\overset{\cref{itm:tm-free}}{=}1_{T_\Omega W},\nonumber\\[3pt]
    T_\Omega r'\circ T_\Omega r\circ\eta_W
      &\overset{\eqref{eq:eta-nat}}{=}T_\Omega r'\circ\eta_{W'}\circ r
       \overset{\eqref{eq:eta-nat}}{=}\eta_{W''}\circ r'\circ r
       \overset{\eqref{eq:eta-nat}}{=}T_\Omega(r'\circ r)\circ\eta_W,\\
    \leadsto\ T_\Omega r'\circ T_\Omega r
      &\overset{\cref{itm:tm-free}}{=}T_\Omega(r'\circ r).\nonumber
  \end{align}
\item \textbf{Monadicity.}
\begin{align}
    \mu_W\circ\eta_{T_\Omega W}
      &=(1_{T_\Omega W})^\sharp\circ\eta_{T_\Omega W}
       \overset{\eqref{eq:extension}}{=}1_{T_\Omega W},\label{eq:unit-l}\\[2pt]
    \mu_W\circ T_\Omega\eta_W\circ\eta_W
      &\overset{\eqref{eq:eta-nat}}{=}\mu_W\circ\eta_{T_\Omega W}\circ\eta_W
       \overset{\eqref{eq:unit-l}}{=}\eta_W
       =1_{T_\Omega W}\circ\eta_W,\nonumber\\
    \leadsto\ \mu_W\circ T_\Omega\eta_W
      &\overset{\cref{itm:tm-free}}{=}1_{T_\Omega W}.\label{eq:unit-r}\\[2pt]
    \mu_W\circ\mu_{T_\Omega W}\circ\eta_{T_\Omega T_\Omega W}
      &\overset{\eqref{eq:unit-l}}{=}\mu_W
       \overset{\eqref{eq:unit-l}}{=}\mu_W\circ\eta_{T_\Omega W}\circ\mu_W
       \overset{\eqref{eq:eta-nat}}{=}
         \mu_W\circ T_\Omega\mu_W\circ\eta_{T_\Omega T_\Omega W},\nonumber\\
    \leadsto\ \mu_W\circ\mu_{T_\Omega W}
      &\overset{\cref{itm:tm-free}}{=}\mu_W\circ T_\Omega\mu_W.\label{eq:assoc}\\[2pt]
    \mu_{W'}\circ T_\Omega T_\Omega r\circ\eta_{T_\Omega W}
      &\overset{\eqref{eq:eta-nat}}{=}
       \mu_{W'}\circ\eta_{T_\Omega W'}\circ T_\Omega r
       \overset{\eqref{eq:unit-l}}{=}T_\Omega r
       \overset{\eqref{eq:unit-l}}{=}T_\Omega r\circ\mu_W\circ\eta_{T_\Omega W},\nonumber\\
    \leadsto\ \mu_{W'}\circ T_\Omega T_\Omega r
      &\overset{\cref{itm:tm-free}}{=}T_\Omega r\circ\mu_W .\label{eq:mu-nat}
  \end{align}
  By \eqref{eq:eta-nat} and \eqref{eq:mu-nat} both \(\eta\) and \(\mu\) are
  natural, so \eqref{eq:unit-l}, \eqref{eq:unit-r} and \eqref{eq:assoc} are
  identities between natural transformations of endofunctors of \(\Set\), and the
  graphical calculus draws them. Read a diagram from top to bottom and its wires
  from left to right as the composite is written, so that the left wire of
  \(T_\Omega\circ T_\Omega\) is the outer factor. The region is \(\Set\), a wire
  is \(T_\Omega\), a dot with nothing above it is \(\eta\), and a merge of two
  wires is \(\mu\):
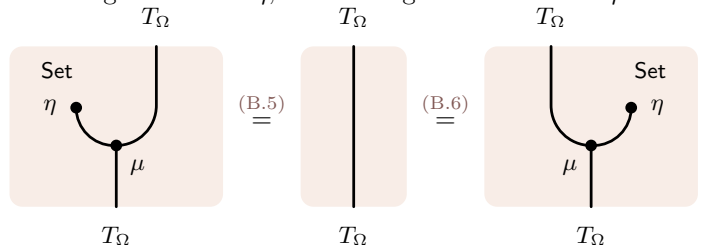
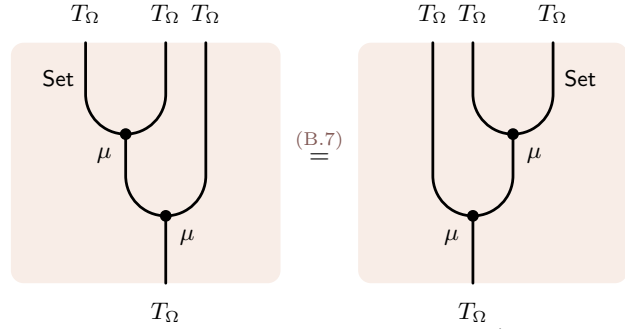
  \begin{center}
  \renewcommand\sdxunit{1.06}\renewcommand\sdyunit{1.06}
  \(
  \begin{sdiag}{2}
    \sdcanvas{sdcodomain}{-0.45}{1.45}
    \sdlink{0}{\sdgoldup}{0.5}{\sdgold}\sdvertex{0}{\sdgoldup}{\eta}{left}
    \sdprong{1}{0.5}{\sdgold}\sdstem{0.5}{\sdgold}{0.5}
    \sdvertexplain{0.5}{\sdgold}\sdside{0.5}{\sdgold}{below right}{\mu}
    \sdtop{1}{T_\Omega}\sdbot{0.5}{T_\Omega}\sdcat{-0.24}{\Set}
  \end{sdiag}
  \sdeq{\eqref{eq:unit-l}}
  \begin{sdiag}{2}
    \sdcanvas{sdcodomain}{-0.28}{0.28}
    \sdwire{0}{0}\sdtop{0}{T_\Omega}\sdbot{0}{T_\Omega}
  \end{sdiag}
  \sdeq{\eqref{eq:unit-r}}
  \begin{sdiag}{2}
    \sdcanvas{sdcodomain}{-0.45}{1.45}
    \sdlink{1}{\sdgoldup}{0.5}{\sdgold}\sdvertex{1}{\sdgoldup}{\eta}{right}
    \sdprong{0}{0.5}{\sdgold}\sdstem{0.5}{\sdgold}{0.5}
    \sdvertexplain{0.5}{\sdgold}\sdside{0.5}{\sdgold}{below left}{\mu}
    \sdtop{0}{T_\Omega}\sdbot{0.5}{T_\Omega}\sdcat{1.24}{\Set}
  \end{sdiag}
  \,,
  \)

  \(
  \begin{sdiag}{3}
    \sdcanvas{sdcodomain}{-0.55}{2.05}
    \sdprong{0}{0.5}{\sdlvl{0.38}}\sdprong{1}{0.5}{\sdlvl{0.38}}
    \sdvertexplain{0.5}{\sdlvl{0.38}}\sdside{0.5}{\sdlvl{0.38}}{below left}{\mu}
    \sdlink{0.5}{\sdlvl{0.38}}{1}{\sdlvl{0.72}}%
    \sdprong{1.5}{1}{\sdlvl{0.72}}
    \sdstem{1}{\sdlvl{0.72}}{1}
    \sdvertexplain{1}{\sdlvl{0.72}}\sdside{1}{\sdlvl{0.72}}{below right}{\mu}
    \sdtop{0}{T_\Omega}\sdtop{1}{T_\Omega}\sdtop{1.5}{T_\Omega}\sdbot{1}{T_\Omega}
    \sdcat{-0.34}{\Set}
  \end{sdiag}
  \sdeq{\eqref{eq:assoc}}
  \begin{sdiag}{3}
    \sdcanvas{sdcodomain}{-0.05}{2.55}
    \sdprong{2}{1.5}{\sdlvl{0.38}}\sdprong{1}{1.5}{\sdlvl{0.38}}
    \sdvertexplain{1.5}{\sdlvl{0.38}}\sdside{1.5}{\sdlvl{0.38}}{below right}{\mu}
    \sdlink{1.5}{\sdlvl{0.38}}{1}{\sdlvl{0.72}}
    \sdprong{0.5}{1}{\sdlvl{0.72}}
    \sdstem{1}{\sdlvl{0.72}}{1}
    \sdvertexplain{1}{\sdlvl{0.72}}\sdside{1}{\sdlvl{0.72}}{below left}{\mu}
    \sdtop{0.5}{T_\Omega}\sdtop{1}{T_\Omega}\sdtop{2}{T_\Omega}\sdbot{1}{T_\Omega}
    \sdcat{2.34}{\Set}
  \end{sdiag}
  \,.
  \)
  \end{center}
  So \((T_\Omega,\eta,\mu)\) is a monad. Finally, for a map
  \(g\colon W\to T_\Omega W'\) of sets,
  \begin{align*}
    \mu_{W'}\circ T_\Omega g\circ\eta_W
      &\overset{\eqref{eq:eta-nat}}{=}\mu_{W'}\circ\eta_{T_\Omega W'}\circ g
       \overset{\eqref{eq:unit-l}}{=}g
       \overset{\eqref{eq:extension}}{=}g^\sharp\circ\eta_W\\
    \leadsto\ \mu_{W'}\circ T_\Omega g
      &\overset{\cref{itm:tm-free}}{=}g^\sharp .\qedhere
  \end{align*}
\end{enumerate}
\end{proof}

\section{Proofs of \texorpdfstring{\Cref{sec:nerve}}{Section 4}}\label{app:nerve}

\begin{proof}[Proof of \Cref{lem:filtration}]~
\begin{enumerate}
\item\label{itm:len-ops} Throughout, \(\mathbf{x}\in\Nv(X)_n\), and \(n\ge 1\)
  wherever a face occurs, since \(d_i\) is defined on \(\Nv(X)_n\) only then.
  The first item below bounds the length under an arbitrary reindexing, and the
  two after it evaluate the length on the faces and on the degeneracies.
  \begin{itemize}
  \item \textbf{Reindexing.} Let \(\alpha\colon[m]\to[n]\) in \(\Delta\).
    Iterating the triangle inequality along \(x_a,x_{a+1},\dots,x_b\) gives
    \(d(x_a,x_b)\le\sum_{k=a+1}^{b}d(x_{k-1},x_k)\) for \(a\le b\), and
    \(\alpha\) is order-preserving, so the index blocks
    \(\{\alpha(j-1)+1,\dots,\alpha(j)\}\) for \(1\le j\le m\) are pairwise
    disjoint subsets of \(\{1,\dots,n\}\). Summing that inequality over these
    blocks, and adding the summands of \(\lambda\mathbf{x}\) outside them, all of
    which are nonnegative, gives
    \begin{equation}\label{eq:reindex}
    \begin{aligned}
      \lambda(\mathbf{x}\circ\alpha)
        &=\sum_{j=1}^{m}d\bigl(x_{\alpha(j-1)},x_{\alpha(j)}\bigr) \le\sum_{j=1}^{m}\;\sum_{k=\alpha(j-1)+1}^{\alpha(j)}d(x_{k-1},x_k)\\
        &\le\sum_{k=1}^{n}d(x_{k-1},x_k)=\lambda\mathbf{x}.
    \end{aligned}
    \end{equation}
    By \eqref{eq:nerve} every face and every degeneracy is of the form
    \(\mathbf{x}\mapsto\mathbf{x}\circ\alpha\), the cases \(\alpha=\delta^i\)
    and \(\alpha=\sigma^i\) of \eqref{eq:cofaces}, so \eqref{eq:reindex} gives
    \(\lambda(d_i\mathbf{x})\le\lambda\mathbf{x}\) and
    \(\lambda(s_i\mathbf{x})\le\lambda\mathbf{x}\) at once. The next two items
    compute both lengths, which the statement needs for \(s_i\).
  \item \textbf{Faces.} By \eqref{eq:cofaces} the entries of \(d_i\mathbf{x}\)
    are \(x_0,\dots,x_{i-1},x_{i+1},\dots,x_n\), so
    \begin{equation}\label{eq:face-length}
    \begin{aligned}
      \lambda\mathbf{x}&=\sum_{j=1}^{n}d(x_{j-1},x_j),\\[0.4ex]
      \lambda(d_i\mathbf{x})&=
      \begin{cases}
        \displaystyle\sum_{j=2}^{n}d(x_{j-1},x_j),
          &i=0,\\[2.4ex]
        \displaystyle\sum_{j=1}^{i-1}d(x_{j-1},x_j)+d(x_{i-1},x_{i+1})
          +\sum_{j=i+2}^{n}d(x_{j-1},x_j),
          &0<i<n,\\[2.4ex]
        \displaystyle\sum_{j=1}^{n-1}d(x_{j-1},x_j),
          &i=n.
      \end{cases}
    \end{aligned}
    \end{equation}
    The outer faces drop one summand of \(\lambda\mathbf{x}\), namely
    \(d(x_0,x_1)\) for \(i=0\) and \(d(x_{n-1},x_n)\) for \(i=n\), so the first
    and the third case of \eqref{eq:face-length} may be written as
    \[
      \lambda\mathbf{x}=d(x_0,x_1)+\lambda(d_0\mathbf{x}),
      \qquad
      \lambda\mathbf{x}=\lambda(d_n\mathbf{x})+d(x_{n-1},x_n).
    \]
    An inner face, \(0<i<n\), instead replaces the two summands
    \(d(x_{i-1},x_i)\) and \(d(x_i,x_{i+1})\) by the single summand
    \(d(x_{i-1},x_{i+1})\), while all remaining summands agree. Addition on
    \(\Rp\) is monotone and
    \(d(x_{i-1},x_{i+1})\le d(x_{i-1},x_i)+d(x_i,x_{i+1})\), so
    \(\lambda(d_i\mathbf{x})\le\lambda\mathbf{x}\) in each of the three cases.
    If \(\lambda\mathbf{x}<\infty\), then every summand is finite, and
    cancelling the summands common to both lines of \eqref{eq:face-length}
    leaves the defect of the triangle inequality at the middle index
    \(i\), which is what the face \(d_i\) contracts,
    \begin{equation}\label{eq:face-defect}
      \lambda\mathbf{x}-\lambda(d_i\mathbf{x})
        =d(x_{i-1},x_i)+d(x_i,x_{i+1})-d(x_{i-1},x_{i+1})\ge 0
        \qquad(0<i<n),
    \end{equation}
    so an inner face preserves the length if and only if
    \(d(x_{i-1},x_{i+1})=d(x_{i-1},x_i)+d(x_i,x_{i+1})\). For \(n=1\) only the
    outer faces \(i=0\) and \(i=n\) occur.
  \item \textbf{Degeneracies.} By \eqref{eq:cofaces} the entries of
    \(s_i\mathbf{x}\) are \(x_0,\dots,x_i,x_i,\dots,x_n\), so
    \[
      \lambda(s_i\mathbf{x})
        =\sum_{j=1}^{i}d(x_{j-1},x_j)+d(x_i,x_i)+\sum_{j=i+1}^{n}d(x_{j-1},x_j)
        =\lambda\mathbf{x}+d(x_i,x_i)
        =\lambda\mathbf{x},
    \]
    because \(d(x,x)=0\) in every object of \(\Met\), by the unit axiom of a
    \(\V\)-category.
  \end{itemize}
\item\label{itm:len-sub} Let \(\ell\le\ell'\) in \([0,\infty)\), the set of
  objects of \(\Vfin^{\op}\), with sublevel sets as in \eqref{eq:length}.
  \begin{itemize}
  \item \textbf{Sublevel Sets Are Simplicial Subsets.} For every
    \(\alpha\colon[m]\to[n]\) of \(\Delta\) and every
    \(\mathbf{x}\in\bigl(\Nv(X)^{\le\ell}\bigr)_n\), the reindexed simplex has
    length at most \(\ell\) as well, since
    \[
      \lambda\bigl(\Nv(X)(\alpha)\mathbf{x}\bigr)
        =\lambda(\mathbf{x}\circ\alpha)
        \overset{\eqref{eq:reindex}}{\le}\lambda\mathbf{x}\le\ell ,
    \]
    so \(\Nv(X)(\alpha)\) carries \(\bigl(\Nv(X)^{\le\ell}\bigr)_n\) into
    \(\bigl(\Nv(X)^{\le\ell}\bigr)_m\) for every \(\alpha\), and
    \(\Nv(X)^{\le\ell}\) is a subfunctor of \(\Nv(X)\). The cases
    \(\alpha=\delta^i\) and \(\alpha=\sigma^i\) are all faces and all
    degeneracies, so this one computation settles the whole simplicial structure at
    once.
  \item \textbf{The Inclusions.} From \(\lambda\mathbf{x}\le\ell\le\ell'\) we get
    \(\bigl(\Nv(X)^{\le\ell}\bigr)_n\subseteq\bigl(\Nv(X)^{\le\ell'}\bigr)_n\) for every \(n\). Write
    \(\iota_{\ell,\ell'}\colon\Nv(X)^{\le\ell}\to\Nv(X)^{\le\ell'}\) for the
    inclusion. It is simplicial, because both sides carry the restrictions of
    the operators of \(\Nv(X)\), which the previous part provides.
  \item \textbf{Functoriality.} The category \(\Vfin^{\op}\) is a poset: it has one
    morphism \(\ell\to\ell'\) when \(\ell\le\ell'\) and none otherwise, so a
    functor out of it is an assignment on objects together with one on morphisms
    that respects identities and composites, and the pair of
    \[
      \Nv(X)^{\le-}(\ell)=\Nv(X)^{\le\ell},
      \qquad
      \Nv(X)^{\le-}(\ell\le\ell')=\iota_{\ell,\ell'},
    \]
    is well defined on both counts. It preserves identities and composites,
    \[
      \iota_{\ell,\ell}=1_{\Nv(X)^{\le\ell}},
      \qquad
      \iota_{\ell',\ell''}\circ\iota_{\ell,\ell'}=\iota_{\ell,\ell''}
      \quad(\ell\le\ell'\le\ell''),
    \]
    the first because the inclusion of a subset into itself is its identity, the
    second because a composite of inclusions of subsets is again an inclusion.
    Hence \(\Nv(X)^{\le-}\colon\Vfin^{\op}\to\SSet\) is a functor, and it is the
    filtration of \(\Nv(X)\) by length.
  \end{itemize}
\item\label{itm:len-map} Let \(F\colon X\to Y\) be nonexpansive and put
  \(\Nv(F)\mathbf{x}=f\circ\mathbf{x}\), which on entries is the map
  \(\langle x_0,\dots,x_n\rangle\mapsto\langle fx_0,\dots,fx_n\rangle\) appearing in
  the statement of the lemma.
  \begin{itemize}
  \item \textbf{\(\Nv(F)\) Is Simplicial.} For every \(\alpha\colon[m]\to[n]\)
    of \(\Delta\) and \(\mathbf{x}\in\Nv(X)_n\), associativity gives
    \[
      \Nv(F)\bigl(\Nv(X)(\alpha)\mathbf{x}\bigr)
        \overset{\eqref{eq:nerve}}{=}f\circ(\mathbf{x}\circ\alpha)
        =(f\circ\mathbf{x})\circ\alpha
        \overset{\eqref{eq:nerve}}{=}\Nv(Y)(\alpha)\bigl(\Nv(F)\mathbf{x}\bigr).
    \]
    The cases \(\alpha=\delta^i\) and \(\alpha=\sigma^i\) are all faces and all
    degeneracies, so this one computation gives \(d_i\Nv(F)=\Nv(F)d_i\) and
    \(s_i\Nv(F)=\Nv(F)s_i\) for every \(i\), and \(\Nv(F)\) is a simplicial map
    \(\Nv(X)\to\Nv(Y)\) between the underlying simplicial sets.
  \item \textbf{\(\Nv(F)\) Does Not Increase the Length.} Nonexpansiveness of
    \(F\) is \(d(fx,fx')\le d(x,x')\) for all \(x,x'\in\Ob X\). Summing this
    inequality over the \(n\) consecutive pairs of entries of \(\mathbf{x}\), and
    using that addition on \(\Rp\) is monotone, we obtain the estimate
    \[
      \lambda\bigl(\Nv(F)\mathbf{x}\bigr)
        =\sum_{j=1}^{n}d(fx_{j-1},fx_j)
        \le\sum_{j=1}^{n}d(x_{j-1},x_j)
        =\lambda\mathbf{x},
    \]
    so \(\Nv(F)\) restricts to
    \(\Nv(F)^{\le\ell}\colon\Nv(X)^{\le\ell}\to\Nv(Y)^{\le\ell}\) for every
    \(\ell\), and these restrictions are again simplicial maps, being restrictions
    of a simplicial map to simplicial subsets.
  \item \textbf{Naturality in \(\ell\).} The restrictions \(\Nv(F)^{\le\ell}\)
    commute with the inclusions of the filtration in \(X\) and in \(Y\).
    Writing \(\iota_{\ell',\infty}\) for the inclusion of the sublevel set at
    \(\ell'\) into the nerve itself, the sublevel set at \(\ell=\infty\) whose
    condition \(\lambda\mathbf{x}\le\infty\) is void, this is the
    commutativity in \(\SSet\) of the two squares below, which share their
    middle column,
    \[
    \begin{tikzcd}[row sep=2.0em, column sep=3.0em]
      \Nv(X)^{\le\ell} \arrow[r, hook, "\iota_{\ell,\ell'}"]
        \arrow[d, "{\Nv(F)^{\le\ell}}"'] &
      \Nv(X)^{\le\ell'} \arrow[r, hook, "\iota_{\ell',\infty}"]
        \arrow[d, "{\Nv(F)^{\le\ell'}}"] &
      \Nv(X) \arrow[d, "{\Nv(F)}"]\\
      \Nv(Y)^{\le\ell} \arrow[r, hook, "\iota_{\ell,\ell'}"'] &
      \Nv(Y)^{\le\ell'} \arrow[r, hook, "\iota_{\ell',\infty}"'] &
      \Nv(Y)\mathrlap{\,.}
    \end{tikzcd}
    \]
    Every horizontal arrow is an inclusion of a subset and every
    vertical arrow is a restriction of the one map \(\Nv(F)\), so the two
    composites of either square send a simplex \(\mathbf{x}\) of its upper left
    corner to \(f\circ\mathbf{x}\). Both squares therefore commute, and
    \(\Nv(F)^{\le-}\colon\Nv(X)^{\le-}\Rightarrow\Nv(Y)^{\le-}\) is a natural
    transformation of functors \(\Vfin^{\op}\to\SSet\).
  \end{itemize}
\item\label{itm:len-functor} Let \(G\colon Y\to Z\) in \(\Met\). For every
  \(n\ge 0\) and every \(\mathbf{x}\in\Nv(X)_n\), associativity and unitality of
  composition of maps give the two identities
  \[
    \Nv(1_X)\mathbf{x}=1_{\Ob X}\circ\mathbf{x}=\mathbf{x},
    \qquad
    \Nv(GF)\mathbf{x}=(g\circ f)\circ\mathbf{x}
      =g\circ(f\circ\mathbf{x})=\Nv(G)\bigl(\Nv(F)\mathbf{x}\bigr),
  \]
  and both identities restrict to the sublevel sets, since every map occurring
  there is a restriction of \(\Nv(1_X)\) or of \(\Nv(GF)\). Hence
  \(\Nv(-)^{\le-}\colon\Met\to\SSet^{\Vfin^{\op}}\) is a functor.\qedhere
\end{enumerate}
\end{proof}

\begin{proof}[Proof of \Cref{cor:degree-zero}]~
\begin{itemize}
\item \textbf{Chains in Degrees Zero and One.} A one-tuple is nondegenerate,
  the condition \(x_{i-1}\neq x_i\) for \(1\le i\le 0\) being empty, and its
  length is the empty sum
  \(\lambda\langle x\rangle=\sum_{i=1}^{0}d(x_{i-1},x_i)=0\le\ell\). A tuple
  \(\langle x,y\rangle\) is nondegenerate when \(x\neq y\), and
  \(\lambda\langle x,y\rangle=d(x,y)\). \eqref{eq:filtered-chains} reads
  \[
    F_\ell C_0(X)=\bigl\langle\langle x\rangle\bigm|x\in\Ob X\bigr\rangle_{\mathbb{Z}},
    \qquad
    F_\ell C_1(X)=\bigl\langle\langle x,y\rangle\bigm|
      x\neq y,\ d(x,y)\le\ell\bigr\rangle_{\mathbb{Z}},
  \]
  and both are free abelian, the first on \(\Ob X\) and the second on the set of
  those ordered pairs, each being spanned by a subset of the basis of
  nondegenerate tuples.
\item \textbf{The Quotient.} There are no chains in degree \(-1\), so
  \(\PH_0(X)(\ell)=F_\ell C_0(X)/\partial F_\ell C_1(X)\). On a generator of
  \(F_\ell C_1(X)\) the differential \(\partial=d_0-d_1\) gives
  \(\partial\langle x,y\rangle=\langle y\rangle-\langle x\rangle\), whence
  \begin{equation}\label{eq:ph0-quotient}
    \PH_0(X)(\ell)=\mathbb{Z}[\Ob X]\Big/
      \bigl\langle\langle y\rangle-\langle x\rangle\bigm|
        x\neq y,\ d(x,y)\le\ell\bigr\rangle_{\mathbb{Z}} .
  \end{equation}
\item \textbf{The Components.} Let \(G_\ell(X)\) be the graph with vertex set
  \(\Ob X\) and an edge \(xy\) whenever \(x\neq y\) and \(d(x,y)\le\ell\). The
  distance of \(\Met\) is symmetric, so this condition does not depend on the
  order of the two points. Write \(R\) for the subgroup of relations divided
  out in \eqref{eq:ph0-quotient} and \([x]\in\pi_0G_\ell(X)\) for the
  component of \(x\). All but finitely many of the integers \(a_x\) below are
  zero, and extending the values \(p\langle x\rangle=[x]\) additively from the
  one-tuple basis of \(\mathbb{Z}[\Ob X]\) defines the homomorphism of abelian
  groups
  \[
    p\colon\mathbb{Z}[\Ob X]\longrightarrow\mathbb{Z}[\pi_0G_\ell(X)],
    \qquad
    p\Bigl(\sum_{x\in\Ob X}a_x\langle x\rangle\Bigr)
      =\sum_{x\in\Ob X}a_x[x].
  \]
  \begin{itemize}
  \item \textbf{Surjectivity.} Every component \(c\in\pi_0G_\ell(X)\) is
    nonempty. Choosing a point \(x_c\in c\) for each \(c\) gives
    \(p\langle x_c\rangle=[x_c]=c\), so a general element of the target is
    hit, \(\sum_{c}b_c\,c=p\bigl(\sum_{c}b_c\langle x_c\rangle\bigr)\), and
    \(p\) is surjective.
  \item \textbf{Every Relation Dies.} A generator of \(R\) is
    \(\langle y\rangle-\langle x\rangle\) with \(x\neq y\) and
    \(d(x,y)\le\ell\). Then \(xy\) is an edge of \(G_\ell(X)\), its two
    endpoints lie in one component, \([y]=[x]\), and
    \(p(\langle y\rangle-\langle x\rangle)=[y]-[x]=0\). Hence
    \(R\subseteq\ker p\), and \(p\) descends to a surjection
    \(\bar p\colon\PH_0(X)(\ell)\twoheadrightarrow\mathbb{Z}[\pi_0G_\ell(X)]\),
    \([\langle x\rangle]\mapsto[x]\), on the quotient \eqref{eq:ph0-quotient}.
  \item \textbf{The Kernel Is No Larger.} The components form a basis of the
    target, so
    \(p\bigl(\sum_{x}a_x\langle x\rangle\bigr)=\sum_{c}\bigl(\sum_{x\in c}a_x\bigr)c\)
    vanishes if and only if \(\sum_{x\in c}a_x=0\) for every component \(c\).
    For an element of \(\ker p\) these sums vanish, so subtracting
    \(\sum_{c}\bigl(\sum_{x\in c}a_x\bigr)\langle x_c\rangle=0\) leaves the
    element unchanged and regroups it, one component at a time, as
    \[
      \sum_{x\in\Ob X}a_x\langle x\rangle
        =\sum_{c\in\pi_0G_\ell(X)}\sum_{x\in c}
          a_x\bigl(\langle x\rangle-\langle x_c\rangle\bigr).
    \]
    Each difference on the right has \(x\) and \(x_c\) in one component, and
    along a path \(x_c=v_0,v_1,\dots,v_k=x\) of \(G_\ell(X)\) it telescopes
    into relations,
    \[
      \langle x\rangle-\langle x_c\rangle
        =\sum_{r=1}^{k}\bigl(\langle v_r\rangle-\langle v_{r-1}\rangle\bigr)
        =\partial\sum_{r=1}^{k}\langle v_{r-1},v_r\rangle\in R,
    \]
    the tuples \(\langle v_{r-1},v_r\rangle\) lying in \(F_\ell C_1(X)\)
    because consecutive path vertices are distinct and at distance at most
    \(\ell\). So \(\ker p\subseteq R\), the two subgroups agree, and \(\bar p\)
    is injective with kernel \(\ker p/R=0\). The map \(\bar p\) is therefore an
    isomorphism, and \(\PH_0(X)(\ell)\cong\mathbb{Z}[\pi_0G_\ell(X)]\) is free
    abelian on the components of \(G_\ell(X)\).
  \end{itemize}
\item \textbf{The Vietoris--Rips Complex.} For \(\ell\in[0,\infty)\) let
  \(\VR_\ell(X)\) be the simplicial complex with vertex set \(\Ob X\) whose
  simplices are the finite nonempty subsets \(\sigma\subseteq\Ob X\) of diameter
  at most \(\ell\), that is, with \(d(x,y)\le\ell\) for all \(x,y\in\sigma\).
  Its \(1\)-skeleton is computed from the definition. A singleton \(\{x\}\) has
  diameter \(d(x,x)=0\le\ell\), so the vertices of \(\VR_\ell(X)\) are all of
  \(\Ob X\), the vertex set of \(G_\ell(X)\). A pair \(\{x,y\}\) with
  \(x\neq y\) has diameter
  \(\max\{d(x,x),d(x,y),d(y,x),d(y,y)\}=d(x,y)\), the distance being symmetric
  and zero on the diagonal, so \(\{x,y\}\) is a \(1\)-simplex of
  \(\VR_\ell(X)\) if and only if \(d(x,y)\le\ell\), that is, if and only if
  \(xy\) is an edge of \(G_\ell(X)\). The \(1\)-skeleton of \(\VR_\ell(X)\) is
  therefore the graph \(G_\ell(X)\). Sending the vertex \(x\) to the one-tuple
  \(\langle x\rangle\) makes the degree-zero simplicial chains of
  \(\VR_\ell(X)\) into \(\mathbb{Z}[\Ob X]\), the boundary of an edge
  \(xy\), in either order, into \(\pm(\langle y\rangle-\langle x\rangle)\),
  and the subgroup the boundaries span into the subgroup \(R\) of
  \emph{The Components}, so \(H_0(\VR_\ell(X))\cong\mathbb{Z}[\Ob X]/R
  =\PH_0(X)(\ell)\), and both are free abelian on \(\pi_0G_\ell(X)\) by that
  item. For naturality let \(\ell\le\ell'\) in \([0,\infty)\). A simplex of
  \(\VR_\ell(X)\) has diameter at most \(\ell\le\ell'\), so \(\VR_\ell(X)\) is
  a subcomplex of \(\VR_{\ell'}(X)\), and the isomorphisms just computed put
  the two transition maps into the square of abelian groups
  \[
  \begin{tikzcd}[row sep=2.0em, column sep=4.6em]
    \PH_0(X)(\ell)
      \arrow[r, "{\PH_0(X)(\ell\le\ell')}", "\dnote{[z]\,\mapsto\,[z]}"']
      \arrow[d, "\cong"']
    & \PH_0(X)(\ell')
      \arrow[d, "\cong"]\\
    H_0\bigl(\VR_\ell(X)\bigr)
      \arrow[r, "\dnote{[z]\,\mapsto\,[z]}"]
    & H_0\bigl(\VR_{\ell'}(X)\bigr)\mathrlap{\,,}
  \end{tikzcd}
  \]
  whose lower map is induced by the inclusion of that subcomplex. The square
  commutes by computation on generators: the classes \([\langle x\rangle]\)
  span \(\PH_0(X)(\ell)\), every one of the four maps fixes the underlying
  point \(x\), and both composites send \([\langle x\rangle]\) to the class of
  the vertex \(x\) in \(H_0(\VR_{\ell'}(X))\). The isomorphism is therefore
  natural in \(\ell\), and \(\PH_0(X)\) is the degree-zero persistent
  homology of the Vietoris--Rips filtration of \(X\).
\item \textbf{The Dendrogram.} Let \(\Ob X\) be finite, and for two points
  \(x,y\in\Ob X\) put
  \[
    u(x,y)=\min\Bigl\{\max_{1\le r\le k}d(v_{r-1},v_r)\Bigm|
      x=v_0,v_1,\dots,v_k=y\ \text{in}\ \Ob X\Bigr\},
  \]
  the smallest height at which a chain joins \(x\) to \(y\). A path of
  \(G_\ell(X)\) from \(x\) to \(y\) is a chain with all steps at most \(\ell\),
  and conversely, so \(x\) and \(y\) lie in one component of \(G_\ell(X)\) if and
  only if \(u(x,y)\le\ell\). The partition \(\pi_0G_\ell(X)\) is therefore the
  set of \(u\)-balls of radius \(\ell\), and \(\ell\mapsto\pi_0G_\ell(X)\) is the
  merge tree of single-linkage clustering, its dendrogram. Over a field \(\mathbb{F}\)
  the module \(\PH_0(X;\mathbb{F})\) is pointwise \(\mathbb{F}[\pi_0G_\ell(X)]\), so its
  barcode has one bar \([0,u\text{-height})\) for each merge and one bar
  \([0,\infty)\) for each component at large \(\ell\).
\begin{figure}[t]
\centering
\begin{tikzpicture}[x=1mm,y=1mm,font=\small,
  pt/.style={circle,fill=ink,inner sep=0pt,minimum size=1.5mm},
  ed/.style={accent,line width=0.85pt},
  ned/.style={ink,line width=0.4pt,dotted},
  lab/.style={font=\footnotesize,fill=white,inner sep=1.1pt},
  tag/.style={anchor=south west,font=\bfseries\color{accent}},
  ax/.style={ink,line width=0.4pt},
  cut/.style={accentsoft,line width=0.4pt,dashed}]

\begin{scope}
  \node[tag] at (0,24) {(a) $\VR_{3.5}(X)$.};
  \coordinate (a1) at (0,0);
  \coordinate (a2) at (15,17);
  \coordinate (a3) at (31,0);
  \coordinate (a4) at (48,-8);
  \fill[accentsoft,opacity=0.45] (a1)--(a2)--(a3)--cycle;
  \draw[ned] (a2)--(a4) node[lab,pos=0.45] {$4$};
  \draw[ned] (a1)--(a4) node[lab,pos=0.72] {$5$};
  \draw[ed] (a1)--(a2) node[lab,midway] {$1$};
  \draw[ed] (a2)--(a3) node[lab,midway] {$2.5$};
  \draw[ed] (a1)--(a3) node[lab,midway] {$3.5$};
  \draw[ed] (a3)--(a4) node[lab,midway] {$1.5$};
  \node[pt] at (a1) {}; \node[pt] at (a2) {};
  \node[pt] at (a3) {}; \node[pt] at (a4) {};
  \node[anchor=east,font=\footnotesize] at ($(a1)+(-1.4,0)$) {$x_1$};
  \node[anchor=south,font=\footnotesize] at ($(a2)+(0,1.4)$) {$x_2$};
  \node[anchor=south west,font=\footnotesize] at ($(a3)+(2.2,1.0)$) {$x_3$};
  \node[anchor=west,font=\footnotesize] at ($(a4)+(1.4,0)$) {$x_4$};
  \node[font=\scriptsize,color=accent] at (15.3,6.0)
    {$\{x_1,x_2,x_3\}$};
\end{scope}

\begin{scope}[shift={(76,-6)}]
  \node[tag] at (0,30) {(b) Dendrogram of $X$.};
  \draw[ax,-{Stealth[length=1.5mm]}] (-6,0) -- (-6,27.5)
    node[anchor=south,font=\footnotesize] {$\ell$};
  \foreach \h/\t in {0/0, 8/1, 12/1.5, 20/2.5}{
    \draw[ax] (-6.9,\h) -- (-6,\h);
    \node[anchor=east,font=\footnotesize] at (-7.2,\h) {$\t$};}
  \foreach \h in {8,12,20}{\draw[cut] (-6,\h) -- (44,\h);}
  \coordinate (b1) at (0,0);  \coordinate (b2) at (9,0);
  \coordinate (b3) at (22,0); \coordinate (b4) at (33,0);
  \draw[ed] (b1)|-(4.5,8)-|(b2);
  \draw[ed] (b3)|-(27.5,12)-|(b4);
  \draw[ed] (4.5,8)|-(16,20)-|(27.5,12);
  \draw[ed] (16,20)--(16,27);
  \foreach \p/\n in {b1/1, b2/2, b3/3, b4/4}{
    \node[pt] at (\p) {};
    \node[anchor=north,font=\footnotesize] at ($(\p)+(0,-1.2)$) {$x_\n$};}
  \node[anchor=west,font=\footnotesize,color=accent] at (36,3.6)
    {$\mathbb{Z}^{4}$};
  \node[anchor=west,font=\footnotesize,color=accent] at (36,10)
    {$\mathbb{Z}^{3}$};
  \node[anchor=west,font=\footnotesize,color=accent] at (36,16)
    {$\mathbb{Z}^{2}$};
  \node[anchor=west,font=\footnotesize,color=accent] at (36,24)
    {$\mathbb{Z}$};
\end{scope}
\end{tikzpicture}
\caption{The example in the proof of \Cref{cor:degree-zero}. \textbf{(a)}~The
complex \(\VR_{3.5}(X)\): solid its edges, shaded its \(2\)-simplex, dotted
the two pairs at distance greater than \(3.5\). \textbf{(b)}~The dendrogram
of \(X\), with the values of \(\PH_0(X)\) on the right.}
\label{fig:dendrogram}
\end{figure}
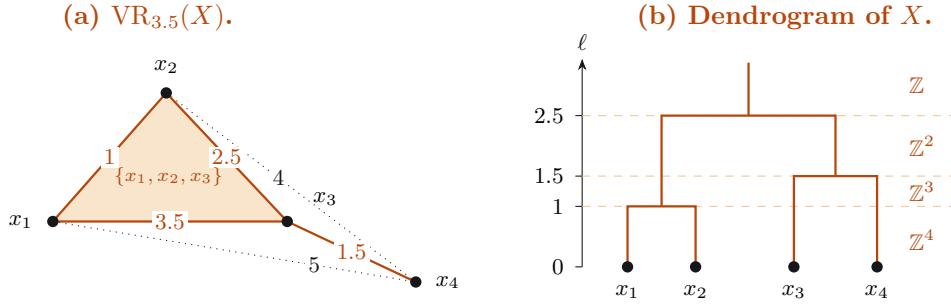

\item \textbf{An Example.} Let \(\Ob X=\{x_1,x_2,x_3,x_4\}\) carry the
  shortest-path metric of the weighted path with weights \(1\), \(2.5\) and
  \(1.5\), so that \(d(x_1,x_2)=1\), \(d(x_2,x_3)=2.5\), \(d(x_3,x_4)=1.5\),
  \(d(x_1,x_3)=3.5\), \(d(x_2,x_4)=4\) and \(d(x_1,x_4)=5\). At \(\ell=3.5\) four
  of the six pairs are edges of \(G_{3.5}(X)\). \Cref{fig:dendrogram} draws
  the complex \(\VR_{3.5}(X)\) and the dendrogram of \(X\).
  The bands on the right of \Cref{fig:dendrogram}(b) are the values of the
  persistence module,
  \[
    \PH_0(X)(\ell)\cong
    \begin{cases}
      \mathbb{Z}^{4}, & 0\le\ell<1,\\
      \mathbb{Z}^{3}, & 1\le\ell<1.5,\\
      \mathbb{Z}^{2}, & 1.5\le\ell<2.5,\\
      \mathbb{Z}, & 2.5\le\ell,
    \end{cases}
  \]
  with transition maps the surjections induced by the merges. Over a field the
  barcode is \([0,1)\), \([0,1.5)\), \([0,2.5)\) and \([0,\infty)\), the finite bars ending at the merge heights.\qedhere
\end{itemize}
\end{proof}

\begin{proof}[Proof of \Cref{prop:mc-explicit}]~
Throughout, \(\mathbf{x}=\langle x_0,\dots,x_n\rangle\) is a nondegenerate tuple
of \(\Nv(X)_n\) with \(\lambda\mathbf{x}=\ell\). That \(X\) is separated means
that \(d(x,y)=0\) forces \(x=y\) (\Cref{sec:base}), so distinct points of
\(\Ob X\) have positive distance. The length \(\ell\) is finite by hypothesis.
\begin{itemize}
\item \textbf{The Two Graded Groups.} By \eqref{eq:filtered-chains} the group
  \(F_\ell C_n(X)\) is spanned by the nondegenerate tuples of length at most
  \(\ell\), and a nondegenerate \(\mathbf{y}\) lies in
  \(F_{<\ell}C_n(X)=\bigcup_{\ell'<\ell}F_{\ell'}C_n(X)\) if and only if
  \(\lambda\mathbf{y}\le\ell'\) for some \(\ell'<\ell\), that is, if and only if
  \(\lambda\mathbf{y}<\ell\). Hence
  \[
    F_\ell C_n(X)=\bigl\langle\mathbf{y}\bigm|\lambda\mathbf{y}\le\ell\bigr\rangle_{\mathbb{Z}},
    \qquad
    F_{<\ell}C_n(X)=\bigl\langle\mathbf{y}\bigm|\lambda\mathbf{y}<\ell\bigr\rangle_{\mathbb{Z}},
  \]
  both spans taken over nondegenerate \(\mathbf{y}\in\Nv(X)_n\), which are basis
  elements of \(C_n(X)\). The first spanning set is the disjoint union of the
  second with \(\{\mathbf{y}\mid\lambda\mathbf{y}=\ell\}\), so the span splits as a
  direct sum and the quotient is the complementary summand,
  \begin{equation}\label{eq:mc-basis}
    F_\ell C_n(X)=F_{<\ell}C_n(X)\oplus
      \bigl\langle\mathbf{y}\bigm|\lambda\mathbf{y}=\ell\bigr\rangle_{\mathbb{Z}},
    \qquad
    \MC_{n,\ell}(X)\cong
      \bigl\langle\mathbf{y}\bigm|\lambda\mathbf{y}=\ell\bigr\rangle_{\mathbb{Z}} .
  \end{equation}
  A direct summand spanned by a subset of a basis is free on that subset, so
  \(\MC_{n,\ell}(X)\) is free abelian on the nondegenerate tuples of length
  \(\ell\).
\item \textbf{The Outer Faces.} Let \(n\ge 1\). The consecutive pairs of
  \(d_0\mathbf{x}\) are consecutive pairs of \(\mathbf{x}\), so
  \(d_0\mathbf{x}\) is again nondegenerate, and the case \(i=0\) of
  \eqref{eq:face-length} together with \(\ell<\infty\) gives
  \[
    \lambda(d_0\mathbf{x})=\lambda\mathbf{x}-d(x_0,x_1)=\ell-d(x_0,x_1)<\ell,
  \]
  because \(x_0\neq x_1\) forces \(d(x_0,x_1)>0\). Therefore
  \(d_0\mathbf{x}\in F_{<\ell}C_{n-1}(X)\) and, in \(\MC_{n-1,\ell}(X)\),
  \[
    [d_0\mathbf{x}]=d_0\mathbf{x}+F_{<\ell}C_{n-1}(X)=0 .
  \]
  The case \(i=n\) of \eqref{eq:face-length} gives
  \(\lambda(d_n\mathbf{x})=\ell-d(x_{n-1},x_n)<\ell\) in the same way, so
  \([d_n\mathbf{x}]=0\).
\item \textbf{The Inner Faces.} Let \(0<i<n\). Both lengths are finite, so
  \eqref{eq:face-defect} reads
  \[
    \ell-\lambda(d_i\mathbf{x})
      =d(x_{i-1},x_i)+d(x_i,x_{i+1})-d(x_{i-1},x_{i+1})\ge 0,
  \]
  and the two sides vanish together. Reading off the equality case and its
  negation,
  \[
    \lambda(d_i\mathbf{x})=\ell
      \iff d(x_{i-1},x_i)+d(x_i,x_{i+1})=d(x_{i-1},x_{i+1})
      \iff x_{i-1}\preceq x_i\preceq x_{i+1},
  \]
  the second equivalence being the definition of \(\preceq\), and otherwise
  \(\lambda(d_i\mathbf{x})<\ell\), which puts \(d_i\mathbf{x}\) into
  \(F_{<\ell}C_{n-1}(X)\) and makes \([d_i\mathbf{x}]=0\). In the equality case
  \(d_i\mathbf{x}\) is nondegenerate: its consecutive pairs are those of
  \(\mathbf{x}\) apart from \((x_{i-1},x_{i+1})\), and \(x_{i-1}=x_{i+1}\) would
  give
  \[
    d(x_{i-1},x_i)+d(x_i,x_{i-1})=d(x_{i-1},x_{i-1})=0,
  \]
  hence \(d(x_{i-1},x_i)=0\) and \(x_{i-1}=x_i\) by separatedness, against the
  nondegeneracy of \(\mathbf{x}\). So \(d_i\mathbf{x}\) is a basis element of
  \(\MC_{n-1,\ell}(X)\) by \eqref{eq:mc-basis}, and in every case
  \begin{equation}\label{eq:face-class}
    [d_i\mathbf{x}]=\partial_i\mathbf{x}=
    \begin{cases}
      \langle x_0,\dots,x_{i-1},x_{i+1},\dots,x_n\rangle,
        & x_{i-1}\preceq x_i\preceq x_{i+1},\\
      0, & \text{otherwise.}
    \end{cases}
  \end{equation}
\item \textbf{The Differential.} Faces do not raise the length by
  \Cref{lem:filtration}, so \(\partial\) carries \(F_\ell C_\bullet(X)\) and
  \(F_{<\ell}C_\bullet(X)\) into themselves and descends to the quotient. On a
  basis element \([\mathbf{x}]\) of \(\MC_{n,\ell}(X)\) the induced differential is
  computed by the two previous items,
  \[
    \partial[\mathbf{x}]
      =\Bigl[\sum_{i=0}^{n}(-1)^id_i\mathbf{x}\Bigr]
      =\sum_{i=0}^{n}(-1)^i[d_i\mathbf{x}]
      =\sum_{i=1}^{n-1}(-1)^i[d_i\mathbf{x}]
      \overset{\eqref{eq:face-class}}{=}\sum_{i=1}^{n-1}(-1)^i\partial_i\mathbf{x},
  \]
  the third equality by \([d_0\mathbf{x}]=[d_n\mathbf{x}]=0\), which is
  \eqref{eq:differential}. Each \(\partial_i\) lowers the degree by one and
  leaves the length at \(\ell\), so \(\partial\) has bidegree
  \((-1,0)\).\qedhere
\end{itemize}
\end{proof}

\begin{proof}[Proof of \Cref{cor:no-betweenness}]
Let \(\mathbf{x}\) be a nondegenerate tuple of length \(\ell\) in degree \(n\).
Step (1) is the differential \eqref{eq:differential} of
\Cref{prop:mc-explicit}. Step (2) reads \(x_i\neq x_{i-1}\) and
\(x_i\neq x_{i+1}\) off the nondegeneracy of \(\mathbf{x}\), so that
\(x_i\notin\{x_{i-1},x_{i+1}\}\) and the hypothesis forbids
\(x_{i-1}\preceq x_i\preceq x_{i+1}\), whence \(\partial_i\mathbf{x}=0\) for
every \(0<i<n\) by \eqref{eq:face-class}. Step (3) is the definition of
homology on a complex whose differential vanishes in every degree and every
length. Then
\[
  \partial\mathbf{x}
    \overset{\text{(1)}}{=}\sum_{i=1}^{n-1}(-1)^i\partial_i\mathbf{x}
    \overset{\text{(2)}}{=}0,
  \qquad
  \MH_{n,\ell}(X)
    \overset{\text{(3)}}{=}\ker\partial/\operatorname{im}\partial
    =\MC_{n,\ell}(X),
\]
and \(\MC_{n,\ell}(X)\) is free abelian on the nondegenerate tuples of length
\(\ell\) by \Cref{prop:mc-explicit}.
\end{proof}

\begin{proof}[Proof of \Cref{lem:digon}]
Both \(\langle x,y\rangle\) and \(\langle y,x\rangle\) are nondegenerate,
\(x\neq y\) being assumed, and so is \(\langle x,y,x\rangle\). The distance of
\(\Met\) is symmetric, so \(\lambda\langle x,y\rangle=\lambda\langle y,x\rangle=e\)
and \(\lambda\langle x,y,x\rangle=d(x,y)+d(y,x)=2e\), which places the two
chains in \(F_eC_1(X)\) and in \(F_{2e}C_2(X)\) by \eqref{eq:filtered-chains}.
On \(C_1(X)\) the differential is \(\partial=d_0-d_1\) with
\(d_0\langle u,v\rangle=\langle v\rangle\) and \(d_1\langle u,v\rangle=\langle u\rangle\),
on \(C_2(X)\) it is \(\partial=d_0-d_1+d_2\), and \(\langle x,x\rangle=0\) in
\(C_1(X)\), that tuple being degenerate, so
\begin{align*}
  \partial\bigl(\langle x,y\rangle+\langle y,x\rangle\bigr)
    &=\bigl(\langle y\rangle-\langle x\rangle\bigr)
      +\bigl(\langle x\rangle-\langle y\rangle\bigr)=0,\\
  \partial\langle x,y,x\rangle
    &=\langle y,x\rangle-\langle x,x\rangle+\langle x,y\rangle
    =\langle x,y\rangle+\langle y,x\rangle .\qedhere
\end{align*}
\end{proof}

\begin{proof}[Proof of \Cref{lem:complex}]~
For \(Z\) in \(\Met\) write \(B_{n,\ell}(Z)\) for the set of nondegenerate
\(\mathbf{z}\in\Nv(Z)_n\) with \(\lambda\mathbf{z}=\ell\), which is a basis of the
free abelian group \(\MC_{n,\ell}(Z)\) by \Cref{prop:mc-explicit}.
\begin{enumerate}
\item Both functors are induced by the chain map \(C_\bullet(F)\) that the
  simplicial map \(\Nv(F)\) of \Cref{lem:filtration}\Cref{itm:filt-map} induces on normalised
  chains, whose value, chain rule and functoriality are inherited on the graded
  pieces of the filtration at every length.
  \begin{itemize}
  \item \textbf{The Filtrations Are Preserved.} By \Cref{lem:filtration}\Cref{itm:filt-map} the
    map \(\Nv(F)\) is simplicial and satisfies \(\lambda\circ\Nv(F)\le\lambda\),
    so the chain map \(C_\bullet(F)\) on normalised chains obeys
    \[
      C_n(F)\bigl(F_\ell C_n(X)\bigr)\subseteq F_\ell C_n(Y),
      \qquad
      C_n(F)\bigl(F_{<\ell}C_n(X)\bigr)\subseteq F_{<\ell}C_n(Y),
    \]
    and \(\MC_{n,\ell}(F)\) is the map it induces on the quotients by
    \(F_{<\ell}C_n\),
    \[
    \begin{tikzcd}[row sep=1.7em, column sep=2.4em]
      F_{<\ell}C_n(X) \arrow[r, hook] \arrow[d, "{C_n(F)}"']
        & F_\ell C_n(X) \arrow[r, two heads] \arrow[d, "{C_n(F)}"]
        & \MC_{n,\ell}(X) \arrow[d, dashed, "{\MC_{n,\ell}(F)}"]\\
      F_{<\ell}C_n(Y) \arrow[r, hook]
        & F_\ell C_n(Y) \arrow[r, two heads]
        & \MC_{n,\ell}(Y)\mathrlap{\,.}
    \end{tikzcd}
    \]
  \item \textbf{Its Value on a Basis Element.} Let
    \(\mathbf{x}\in B_{n,\ell}(X)\) and put \(a_i=d(x_{i-1},x_i)\) and
    \(b_i=d(fx_{i-1},fx_i)\), so that \(b_i\le a_i\) for \(1\le i\le n\) by
    nonexpansiveness. All terms are finite, \(\ell\) being finite, and
    subtracting the two lengths term by term gives
    \[
      \sum_{i=1}^{n}(a_i-b_i)
        =\lambda\mathbf{x}-\lambda\bigl(\Nv(F)\mathbf{x}\bigr),
      \qquad a_i-b_i\ge 0 ,
    \]
    a finite sum of nonnegative reals, which vanishes if and only if every
    summand vanishes. With \(\lambda\mathbf{x}=\ell\) this gives the criterion for
    the class of the image to survive,
    \[
      \lambda\bigl(\Nv(F)\mathbf{x}\bigr)=\ell
        \iff a_i=b_i\ \text{for all}\ 1\le i\le n
        \implies b_i=a_i>0,
    \]
    the last step by nondegeneracy of \(\mathbf{x}\) and separatedness of
    \(X\), so that \(\Nv(F)\mathbf{x}\in B_{n,\ell}(Y)\) in that case. As
    \(\lambda(\Nv(F)\mathbf{x})\le\ell\) throughout, the class of
    \(\Nv(F)\mathbf{x}\) in \(\MC_{n,\ell}(Y)\) is nonzero if and only if the \(a_i\) and the \(b_i\)
    agree in every index \(1\le i\le n\), so that
    \[
      \MC_{n,\ell}(F)[\mathbf{x}]=
      \begin{cases}
        [\langle fx_0,\dots,fx_n\rangle],
          & d(fx_{i-1},fx_i)=d(x_{i-1},x_i)\ \text{for}\ 1\le i\le n,\\
        0, & \text{otherwise.}
      \end{cases}
    \]
  \item \textbf{Chain Map.} Simpliciality of \(\Nv(F)\) gives
    \(d_iC_n(F)=C_{n-1}(F)d_i\) for every \(i\), hence
    \[
      \partial C_n(F)
        =\sum_{i=0}^{n}(-1)^id_iC_n(F)
        =\sum_{i=0}^{n}(-1)^iC_{n-1}(F)d_i
        =C_{n-1}(F)\partial ,
    \]
    so the two composites of the square of chain maps induced on the graded
    pieces,
    \[
    \begin{tikzcd}[row sep=1.7em, column sep=2.6em]
      \MC_{n,\ell}(X) \arrow[r, "\partial"] \arrow[d, "{\MC_{n,\ell}(F)}"']
        & \MC_{n-1,\ell}(X) \arrow[d, "{\MC_{n-1,\ell}(F)}"]\\
      \MC_{n,\ell}(Y) \arrow[r, "\partial"']
        & \MC_{n-1,\ell}(Y)
    \end{tikzcd}
    \]
    coincide, and the square commutes in every degree \(n\) and at every length
    \(\ell\).
  \item \textbf{Functoriality.} For \(G\colon Y\to Z\) in \(\Met\) and
    \(\mathbf{x}\in B_{n,\ell}(X)\), \Cref{lem:filtration}\Cref{itm:filt-nerve} gives
    \[
    \begin{aligned}
      \MC_{n,\ell}(1_X)[\mathbf{x}]&=[\Nv(1_X)\mathbf{x}]=[\mathbf{x}],\\
      \MC_{n,\ell}(GF)[\mathbf{x}]&=[\Nv(G)\Nv(F)\mathbf{x}]
        =\MC_{n,\ell}(G)\MC_{n,\ell}(F)[\mathbf{x}],
    \end{aligned}
    \]
    so \(\MC_{\bullet,\ell}\colon\Met\to\mathrm{Ch}(\Ab)\) is a functor, and so
    is \(\MH_{n,\ell}=H_n\circ\MC_{\bullet,\ell}\).
  \end{itemize}
\item The splitting comes from a partition of the basis \(B_{n,\ell}(X)\) by the
  classes of \(\pi_0X\), a partition that the faces and the differential of
  \Cref{prop:mc-explicit} respect.
  \begin{itemize}
  \item \textbf{All Entries Lie in One Class.} Let
    \(\mathbf{x}\in B_{n,\ell}(X)\). For \(1\le i\le n\) the distance
    \(d(x_0,x_i)\) is bounded by a partial sum of \(\lambda\mathbf{x}\), and the
    distance of \(\Met\) is symmetric, so
    \[
      d(x_0,x_i)\le\sum_{j=1}^{i}d(x_{j-1},x_j)\le\lambda\mathbf{x}=\ell<\infty,
      \qquad
      d(x_i,x_0)=d(x_0,x_i)<\infty ,
    \]
    which is \(x_0\sim x_i\). By transitivity \(x_i\sim x_j\) for all \(i,j\),
    and \(c(\mathbf{x})=[x_0]\in\pi_0X\).
  \item \textbf{The Basis Is Partitioned.} A tuple with all entries in \(c\) is a
    tuple of \(X_c\) of the same length, the distance of \(X_c\) being the
    restriction of \(d\), so \(\mathbf{x}\mapsto c(\mathbf{x})\)
    gives
    \[
      B_{n,\ell}(X)=\coprod_{c\in\pi_0X}B_{n,\ell}(X_c) .
    \]
  \item \textbf{The Partition Respects \(\partial\).} The entries of
    \(d_i\mathbf{x}\) are entries of \(\mathbf{x}\), so
    \(c(d_i\mathbf{x})=c(\mathbf{x})\) and \(\partial_i\) of
    \Cref{prop:mc-explicit} carries \(\bigl\langle B_{n,\ell}(X_c)\bigr\rangle_{\mathbb{Z}}\)
    into \(\bigl\langle B_{n-1,\ell}(X_c)\bigr\rangle_{\mathbb{Z}}\), hence so does
    \(\partial=\sum_{i=1}^{n-1}(-1)^i\partial_i\). Splitting the basis
    splits the whole complex as a direct sum,
    \[
      \MC_{n,\ell}(X)=\bigl\langle B_{n,\ell}(X)\bigr\rangle_{\mathbb{Z}}
        =\bigoplus_{c\in\pi_0X}\bigl\langle B_{n,\ell}(X_c)\bigr\rangle_{\mathbb{Z}}
        =\bigoplus_{c\in\pi_0X}\MC_{n,\ell}(X_c),
    \]
    compatibly with \(\partial\) in every degree, which is the isomorphism
    \eqref{eq:splitting}.
  \item \textbf{Finite Distances.} For \(x,y\in c\) the relation \(x\sim y\)
    holds by the definition of \(\pi_0X\), so \(d(x,y)<\infty\) and \(X_c\) has
    finite distances, as the equivalence of \Cref{sec:base} states.\qedhere
  \end{itemize}
\end{enumerate}
\end{proof}

\section{Proofs of \texorpdfstring{\Cref{sec:comparison}}{Section 5}}\label{app:comparison}

\begin{proof}[Proof of \Cref{cor:critical}]~
\begin{enumerate}
\item\label{itm:crit-ladder} \textbf{The Ladder.} Let \(X\) be an object of
  \(\Met\), let \(n\ge 0\) and let \(\ell\in[0,\infty)\). Write
  \(\iota_\ell\colon F_{<\ell}C_\bullet(X)\hookrightarrow F_\ell C_\bullet(X)\) for the
  inclusion and \(\pi_\ell\colon F_\ell C_\bullet(X)\twoheadrightarrow\MC_{\bullet,\ell}(X)\)
  for the quotient map of \Cref{def:magnitude}. Reading the short exact sequence
  \eqref{eq:ses} of chain complexes in three consecutive degrees, with
  \(C_m(X)=0\) for \(m<0\), gives the ladder
  \begin{equation}\label{eq:ladder}
  \begin{tikzcd}[row sep=2.1em, column sep=2.0em, ampersand replacement=\&]
    0\arrow[r]
      \arrow[from=1-3,to=1-4,trace]\arrow[from=1-3,to=2-3,trace]
      \arrow[from=1-4,to=2-4,trace]\arrow[from=2-2,to=2-3,trace]
      \arrow[from=2-3,to=2-4,trace]
      \&F_{<\ell}C_{m+1}(X)\arrow[r,hook,"\iota_\ell"]\arrow[d,"\partial"']
      \&F_\ell C_{m+1}(X)\arrow[r,two heads,"\pi_\ell"]\arrow[d,"\partial"']
      \&\MC_{m+1,\ell}(X)\arrow[r]\arrow[d,"\partial"']\&0\\
    0\arrow[r]\&F_{<\ell}C_{m}(X)\arrow[r,hook,"\iota_\ell"]\arrow[d,"\partial"']
      \&F_\ell C_{m}(X)\arrow[r,two heads,"\pi_\ell"]\arrow[d,"\partial"']
      \&\MC_{m,\ell}(X)\arrow[r]\arrow[d,"\partial"']\&0\\
    0\arrow[r]\&F_{<\ell}C_{m-1}(X)\arrow[r,hook,"\iota_\ell"]
      \&F_\ell C_{m-1}(X)\arrow[r,two heads,"\pi_\ell"]
      \&\MC_{m-1,\ell}(X)\arrow[r]\&0
  \end{tikzcd}
  \end{equation}
  \looseness=1
  for every degree \(m\). Its rows are exact, its squares commute because
  \(\iota_\ell\) and \(\pi_\ell\) are chain maps. The map \(\pi_\ell\) is
  surjective in each degree, so lifts along it exist, and \(\iota_\ell\) is
  injective with image \(\ker\pi_\ell\), so a chain killed by \(\pi_\ell\) is
  \(\iota_\ell\) of one chain and of no other. The identifications
  \eqref{eq:les-terms} of \Cref{thm:les} read
  \(H_n\bigl(F_{<\ell}C_\bullet(X)\bigr)\) as \(\PH_n(X)(\ell^-)\), the group
  \(H_n\bigl(F_\ell C_\bullet(X)\bigr)\) as \(\PH_n(X)(\ell)\) and
  \(H_m\bigl(\MC_{\bullet,\ell}(X)\bigr)\) as \(\MH_{m,\ell}(X)\), and under
  them \(\tau_{n,\ell}=(\iota_\ell)_*\) and \(\rho_{n,\ell}=(\pi_\ell)_*\).
  Every group named in \Cref{cor:critical} is thereby a homology group of one of
  the three columns of \eqref{eq:ladder}, and the three chases below run along
  its traced path, the rows \(m=n+1\) and \(m=n\) apart from
  \(F_{<\ell}C_{n+1}(X)\).
\item\label{itm:crit-delta}\looseness=-1 \textbf{The Connecting Map.} Let
  \(w\in\MC_{n+1,\ell}(X)\) be a cycle. Choose a lift
  \(\tilde{w}\in F_\ell C_{n+1}(X)\) with \(\pi_\ell\tilde{w}=w\). Its boundary
  satisfies \(\pi_\ell\partial\tilde{w}=\partial w=0\), so
  \(\partial\tilde{w}=\iota_\ell u\) for a unique \(u\in F_{<\ell}C_n(X)\),
  which the row \(m=n\) of \eqref{eq:ladder} supplies. The four elements sit in
  the two rows \(m=n+1\) and \(m=n\) of \eqref{eq:ladder} as
  \begin{equation}\label{eq:delta-chase}
  \begin{tikzcd}[row sep=1.9em, column sep=2.6em, ampersand replacement=\&]
    {}\arrow[from=1-2,to=1-3,trace]\arrow[from=1-2,to=2-2,trace]
      \arrow[from=1-3,to=2-3,trace]\arrow[from=2-1,to=2-2,trace]
      \&\tilde{w}\arrow[r,mapsto,"\pi_\ell"]\arrow[d,mapsto,"\partial"']
      \&w\arrow[d,mapsto,"\partial"]\\
    u\arrow[r,mapsto,"\iota_\ell"']\&\partial\tilde{w}\&0
  \end{tikzcd}
  \end{equation}
  and \(\iota_\ell\partial u=\partial\partial\tilde{w}=0\) with \(\iota_\ell\)
  injective makes \(u\) a cycle. Put \(\delta_{n+1,\ell}[w]=[u]\). Two lifts of
  \(w\) differ by \(\iota_\ell a\) with \(a\in F_{<\ell}C_{n+1}(X)\), so their
  cycles obey \(\iota_\ell(u_1-u_2)=\iota_\ell\partial a\) and
  \(u_1-u_2=\partial a\), which leaves \([u]\) alone. For a second
  representative \(w+\partial y\) of \([w]\), with \(y\in\MC_{n+2,\ell}(X)\),
  lift \(y\) to \(\tilde{y}\in F_\ell C_{n+2}(X)\). The chain
  \(\tilde{w}+\partial\tilde{y}\) lifts \(w+\partial y\) and has boundary
  \(\partial\tilde{w}=\iota_\ell u\) again, and the independence just proved
  makes every lift of \(w+\partial y\) return \([u]\). Lifts add, so
  \(\delta_{n+1,\ell}\) is a homomorphism
  \(\MH_{n+1,\ell}(X)\to\PH_n(X)(\ell^-)\). The three maps
  \(\delta_{n+1,\ell}\), \(\tau_{n,\ell}\) and \(\rho_{n,\ell}\) named so far
  join their four groups into the segment
  \begin{equation}\label{eq:segment}
  \begin{tikzcd}[column sep=2.6em, ampersand replacement=\&]
    \MH_{n+1,\ell}(X)\arrow[r,"\delta_{n+1,\ell}"]
      \&|[fill=tintleft,rounded corners=2pt,inner sep=3pt]|\PH_n(X)(\ell^-)
        \arrow[r,"\tau_{n,\ell}"]
      \&|[fill=tintright,rounded corners=2pt,inner sep=3pt]|\PH_n(X)(\ell)
        \arrow[r,"\rho_{n,\ell}"]
      \&\MH_{n,\ell}(X)
  \end{tikzcd}
  \end{equation}
  whose four groups are four consecutive terms of the long exact sequence of
  \Cref{thm:les}. Warm shading runs through the shaded term on the left, the
  exactness of \eqref{eq:segment} there, the injectivity of \(\tau_{n,\ell}\)
  that follows from it and the group \(\MH_{n+1,\ell}(X)\) that controls that
  injectivity. Cool shading runs through the shaded term on the right, exactness
  there, surjectivity and \(\MH_{n,\ell}(X)\), and \Cref{fig:critical} uses
  both colours.
\item\label{itm:crit-exact-left} \textbf{Exactness at
  \hlleft{\PH_n(X)(\ell^-)}.} Reading \eqref{eq:delta-chase} along its bottom
  row and up its middle column gives
  \(\tau_{n,\ell}[u]=[\iota_\ell u]=[\partial\tilde{w}]=0\), the class of a
  boundary of \(F_\ell C_\bullet(X)\), so
  \(\operatorname{im}\delta_{n+1,\ell}\subseteq\ker\tau_{n,\ell}\). For the
  reverse inclusion let \(z\in F_{<\ell}C_n(X)\) be a cycle with
  \(\tau_{n,\ell}[z]=0\), so that \(\iota_\ell z=\partial g\) for some
  \(g\in F_\ell C_{n+1}(X)\), and put \(v=\pi_\ell g\in\MC_{n+1,\ell}(X)\). The
  chase
  \begin{equation}\label{eq:left-chase}
  \begin{tikzcd}[row sep=1.9em, column sep=2.6em, ampersand replacement=\&]
    {}\arrow[from=1-2,to=1-3,trace]\arrow[from=1-2,to=2-2,trace]
      \arrow[from=1-3,to=2-3,trace]\arrow[from=2-1,to=2-2,trace]
      \&g\arrow[r,mapsto,"\pi_\ell"]\arrow[d,mapsto,"\partial"']
      \&v\arrow[d,mapsto,"\partial"]\\
    z\arrow[r,mapsto,"\iota_\ell"']\&\partial g\&0
  \end{tikzcd}
  \end{equation}
  commutes, since \(\partial v=\pi_\ell\partial g=\pi_\ell\iota_\ell z=0\). So
  \(v\) is a cycle of \(\MC_{n+1,\ell}(X)\), and \(g\) is a lift of \(v\) with
  \(\partial g=\iota_\ell z\). Comparing \eqref{eq:left-chase} with
  \eqref{eq:delta-chase} gives \(\delta_{n+1,\ell}[v]=[z]\), so
  \(\operatorname{im}\delta_{n+1,\ell}\) and \(\ker\tau_{n,\ell}\) are one and
  the same subgroup of the group \(\PH_n(X)(\ell^-)\).
\item\label{itm:crit-exact-right} \textbf{Exactness at \hlright{\PH_n(X)(\ell)}.}
  One inclusion is \(\rho_{n,\ell}\tau_{n,\ell}[z]=[\pi_\ell\iota_\ell z]=0\),
  valid for every cycle \(z\in F_{<\ell}C_n(X)\). For the other let
  \(g\in F_\ell C_n(X)\) be a cycle with \(\rho_{n,\ell}[g]=0\), so that
  \(\pi_\ell g=\partial q\) for some \(q\in\MC_{n+1,\ell}(X)\), and lift \(q\)
  to \(\tilde{q}\in F_\ell C_{n+1}(X)\). Subtracting boundaries from \(g\) and
  from \(\pi_\ell g\) turns the rows \(m=n+1\) and \(m=n\) of the ladder
  \eqref{eq:ladder} into the chase of elements
  \begin{equation}\label{eq:right-chase}
  \begin{tikzcd}[row sep=2.1em, column sep=3.2em, ampersand replacement=\&]
    {}\arrow[from=1-2,to=1-3,trace]\arrow[from=1-2,to=2-2,trace]
      \arrow[from=1-3,to=2-3,trace]\arrow[from=2-1,to=2-2,trace]
      \arrow[from=2-2,to=2-3,trace]
      \&\tilde{q}\arrow[r,mapsto,"\pi_\ell"]\arrow[d,mapsto,"g-\partial(-)"']
      \&q\arrow[d,mapsto,"\pi_\ell g-\partial(-)"]\\
    z\arrow[r,mapsto,"\iota_\ell"']\&g-\partial\tilde{q}
      \arrow[r,mapsto,"\pi_\ell"']\&0
  \end{tikzcd}
  \end{equation}
  whose two vertical assignments subtract a boundary from \(g\) and from
  \(\pi_\ell g\). It commutes because \(\pi_\ell\) is a chain map, and its
  right-hand column reads \(\pi_\ell g-\partial q=0\) by the choice of \(q\).
  The bottom middle entry is killed by \(\pi_\ell\), so it is
  \(\iota_\ell z\) for the one chain \(z\in F_{<\ell}C_n(X)\) named at the
  bottom left, and \(\iota_\ell\partial z=\partial g-\partial\partial\tilde{q}=0\)
  with \(\iota_\ell\) injective makes \(z\) a cycle. Its image is
  \(\tau_{n,\ell}[z]=[g-\partial\tilde{q}]=[g]\). Part~\Cref{itm:crit-ladder} and
  everything after it use of \eqref{eq:ses} only that it is a short exact
  sequence of chain complexes, together with the identifications
  \eqref{eq:les-terms}. The last step of the proof of \Cref{thm:les} supplies
  both with field coefficients, so \eqref{eq:ladder} and \eqref{eq:segment}
  stay exact over every field.
\item\label{itm:crit-iso} \textbf{Vanishing Forces an Isomorphism.} Assume
  \(\MH_{n,\ell}(X)=0=\MH_{n+1,\ell}(X)\). Exactness at
  \hlleft{\PH_n(X)(\ell^-)}, from \Cref{itm:crit-exact-left}, gives
  \(\ker\tau_{n,\ell}=\operatorname{im}\delta_{n+1,\ell}\), and the source of
  \(\delta_{n+1,\ell}\) is the zero group, so \(\tau_{n,\ell}\) is injective.
  Exactness at \hlright{\PH_n(X)(\ell)}, from \Cref{itm:crit-exact-right}, gives
  \(\operatorname{im}\tau_{n,\ell}=\ker\rho_{n,\ell}\), and the target of
  \(\rho_{n,\ell}\) is the zero group, so \(\tau_{n,\ell}\) is surjective. A
  bijective homomorphism of abelian groups is an isomorphism of abelian groups. The
  vanishing of \hlleft{\MH_{n+1,\ell}(X)} alone makes \(\tau_{n,\ell}\)
  injective, and the vanishing of \hlright{\MH_{n,\ell}(X)} alone makes it
  surjective. Both readings hold with coefficients in a field \(\mathbb{F}\), by
  the last sentence of \Cref{itm:crit-exact-right}.
\item\label{itm:crit-blocks}\looseness=-1 \textbf{Blocks of Constancy.} Let \(X\) be an
  object of \(\Met\) with \(\Ob X\) finite, let \(\mathbb{F}\) be a field and fix
  \(n\ge 0\). Each \(\Nv(X)_m\) is finite, an \(m\)-simplex being a map
  \([m]\to\Ob X\), so the finite lengths carried by the nondegenerate simplices
  up to degree \(n+1\) collect into the set
  \begin{equation}\label{eq:crit-lengths}
    \Lambda=\{0\}\cup\bigl\{\lambda\mathbf{x}\bigm|
      0\le m\le n+1,\ \mathbf{x}\in\Nv(X)_m\ \text{nondegenerate},\
      \lambda\mathbf{x}<\infty\bigr\},
  \end{equation}
  a finite subset of \([0,\infty)\), being a union of one singleton and
  \(n+2\) finite sets. Enumerate it as \(\ell_0<\dots<\ell_r\) with
  \(\ell_0=0\), put \(\ell_{r+1}=\infty\) and call \(B_i=[\ell_i,\ell_{i+1})\),
  for \(0\le i\le r\), the \emph{blocks} of \(\Lambda\). They are nonempty,
  pairwise disjoint and cover \([0,\infty)\). Now let \(\ell\le\ell'\) lie in one
  block \(B_i\) and let \(0\le m\le n+1\). By \eqref{eq:filtered-chains} the
  group \(F_\ell C_m(X)\) is spanned by the nondegenerate
  \(\mathbf{x}\in\Nv(X)_m\) with \(\lambda\mathbf{x}\le\ell\) and
  \(F_{\ell'}C_m(X)\) by those with \(\lambda\mathbf{x}\le\ell'\), two subsets of
  the one basis of \(C_m(X)\). A tuple \(\mathbf{x}\) in the second has
  \(\lambda\mathbf{x}<\infty\), hence \(\lambda\mathbf{x}\in\Lambda\) by
  \eqref{eq:crit-lengths}, and \(\lambda\mathbf{x}\le\ell'<\ell_{i+1}\) leaves
  only the values \(\ell_0,\dots,\ell_i\), all of them at most \(\ell_i\le\ell\).
  The two spanning sets agree, so \(F_\ell C_m(X)=F_{\ell'}C_m(X)\) for every
  \(m\le n+1\), and the same holds after \(-\otimes_{\mathbb{Z}}\mathbb{F}\).
  The group \(\PH_n(X;\mathbb{F})(\ell)\) is computed from the degrees
  \(n+1\), \(n\), \(n-1\) and the two differentials between them,
  restrictions of the differential of \(C_\bullet(X)\). All three degrees are
  at most \(n+1\), so \(\PH_n(X;\mathbb{F})(\ell)=\PH_n(X;\mathbb{F})(\ell')\)
  and the transition map between them is the identity of that space.
  \Cref{fig:critical} has \(\Lambda=\{0,1,2,3\}\) and four blocks.
\item\label{itm:crit-summands} \textbf{A Direct Sum of Maps.} Let
  \((f_s\colon V_s\to W_s)_{s\in S}\) be a family of \(\mathbb{F}\)-linear maps.
  The kernel of \(\bigoplus_{s\in S}f_s\) is \(\bigoplus_{s\in S}\ker f_s\) and
  its image is \(\bigoplus_{s\in S}\operatorname{im}f_s\). The first vanishes if
  and only if every \(\ker f_s\) does, and the second is all of
  \(\bigoplus_{s\in S}W_s\) if and only if \(\operatorname{im}f_s=W_s\) holds
  for every \(s\in S\). So \(\bigoplus_{s\in S}f_s\) is injective, surjective or
  bijective if and only if every one of the summand maps \(f_s\) has the
  property in question, for every index \(s\in S\).
\item\label{itm:crit-bars} \textbf{Every Bar Is a Union of Blocks.} Keep \(X\),
  \(\mathbb{F}\) and \(n\) of \Cref{itm:crit-blocks}, and write \(M\) for the
  module \(\PH_n(X;\mathbb{F})\colon\Vfin^{\op}\to\Vect_{\mathbb{F}}\). Each \(M\ell\)
  is a subquotient of \(F_\ell C_n(X)\otimes_{\mathbb{Z}}\mathbb{F}\), whose
  dimension is bounded by the finite number of nondegenerate \(n\)-simplices of
  \(\Nv(X)\), so \(M\) is tame. \Cref{thm:barcode} then provides a set \(S\), a
  family \((J_s)_{s\in S}\) of intervals indexing the multiset \(\mathcal{B}\),
  and an isomorphism of persistence modules
  \(\Theta\colon\bigoplus_{s\in S}\mathbb{F}_{J_s}\to M\). Let \(\ell\le\ell'\)
  lie in one block \(B_i\) and put \(f_s=(\mathbb{F}_{J_s})(\ell\le\ell')\), a
  family of \(\mathbb{F}\)-linear maps as in \Cref{itm:crit-summands}.
  Naturality of \(\Theta\) on the morphism \(\ell\le\ell'\) of \(\Vfin^{\op}\)
  is then the commuting square
  \begin{equation}\label{eq:theta-natural}
  \begin{tikzcd}[row sep=2.1em, column sep=4.0em, ampersand replacement=\&]
    \bigoplus_{s\in S}(\mathbb{F}_{J_s})\ell
      \arrow[r,"\bigoplus_{s\in S}f_s"]\arrow[d,"\Theta_\ell"',"\cong"]
      \&\bigoplus_{s\in S}(\mathbb{F}_{J_s})\ell'
        \arrow[d,"\Theta_{\ell'}","\cong"']\\
    M\ell\arrow[r,"1"']\&M\ell'
  \end{tikzcd}
  \end{equation}
  in which the lower map is the identity by \Cref{itm:crit-blocks} and the two
  vertical maps are isomorphisms. The upper map is an isomorphism, and by
  \Cref{itm:crit-summands} so is every \(f_s\). Its source is \(\mathbb{F}\)
  when \(\ell\in J_s\) and \(0\) otherwise, its target \(\mathbb{F}\) when
  \(\ell'\in J_s\) and \(0\) otherwise, and an isomorphism forces the two
  dimensions to agree. So \(\ell\in J_s\) holds if and only if \(\ell'\in J_s\)
  does, each \(J_s\) meets a block in the empty set or in the whole block, and
  each \(J_s\) is a union of the blocks of \(\Lambda\) listed in
  \Cref{itm:crit-blocks}.
\item\label{itm:crit-shape} \textbf{The Shape of a Bar.} Write
  \(I_s=\{i\mid B_i\subseteq J_s\}\) for \(s\in S\). Intervals are nonempty by
  the definition in \Cref{sec:barcodes}, and a block meeting \(J_s\) lies inside
  it by \Cref{itm:crit-bars}, so \(I_s\) is a nonempty subset of
  \(\{0,\dots,r\}\). Let \(i<j<k\) with \(i,k\in I_s\) and pick \(\ell\in B_i\),
  \(\ell'\in B_j\) and \(\ell''\in B_k\). Then
  \(\ell<\ell_{i+1}\le\ell_j\le\ell'\) and
  \(\ell'<\ell_{j+1}\le\ell_k\le\ell''\), and convexity of the interval \(J_s\)
  puts \(\ell'\) in \(J_s\), so the block \(B_j\) meets \(J_s\) and \(j\in I_s\)
  again by \Cref{itm:crit-bars}. The set \(I_s\) is an interval
  \(\{i_0,\dots,i_1\}\) of consecutive integers, and the bar \(J_s\) is the
  union of the blocks it names,
  \begin{equation}\label{eq:bar-shape}
    J_s=\bigcup_{i=i_0}^{i_1}B_i
      =\begin{cases}
         [\ell_{i_0},\ell_{i_1+1}), & i_1<r,\\
         [\ell_{i_0},\infty), & i_1=r,
       \end{cases}
  \end{equation}
  which is the shape of a bar asserted in the statement, its left endpoint and
  its finite right endpoint lying in the set \(\Lambda\) of
  \eqref{eq:crit-lengths}. The running example of \Cref{fig:critical} has the
  three bars \([0,\infty)\), \([0,2)\) and \([0,1)\).
\item\label{itm:crit-colim} \textbf{The Left Limit of a Direct Sum.} For a
  persistence module \(P\colon\Vfin^{\op}\to\Vect_{\mathbb{F}}\) and
  \(\ell\in[0,\infty)\) write
  \(P(\ell^-)=\operatorname*{colim}_{\ell'<\ell}P\ell'\), with structure maps
  \(c^P_{\ell'}\colon P\ell'\to P(\ell^-)\), and let
  \(\tau^P_\ell\colon P(\ell^-)\to P\ell\) be the map induced by the cocone
  \(\bigl(P(\ell'\le\ell)\bigr)_{\ell'<\ell}\). At \(\ell=0\) the diagram is
  empty and \(P(0^-)=0\). The step \emph{\(\tau_{n,\ell}\) Is the Transition
  Map} in the proof of \Cref{thm:les} identifies \(\tau_{n,\ell}\), read with
  coefficients in \(\mathbb{F}\) from here on, with \(\tau^M_\ell\) for the
  module \(M\) of \Cref{itm:crit-bars}. Every morphism \(\theta\colon P\to P'\)
  of persistence modules and every length \(\ell\in[0,\infty)\) put the two
  left limits and the two values into the square
  \begin{equation}\label{eq:tau-natural}
  \begin{tikzcd}[row sep=2.1em, column sep=3.0em, ampersand replacement=\&]
    P(\ell^-)\arrow[r,"\tau^P_\ell"]
      \arrow[d,"\operatorname*{colim}_{\ell'<\ell}\theta_{\ell'}"']
      \&P\ell\arrow[d,"\theta_\ell"]\\
    P'(\ell^-)\arrow[r,"\tau^{P'}_\ell"']\&P'\ell
  \end{tikzcd}
  \end{equation}
  commute, because both composites precomposed with \(c^P_{\ell'}\) give
  \(P'(\ell'\le\ell)\circ\theta_{\ell'}\), and a map out of a colimit is
  determined by its composites with the structure maps. A coproduct is a colimit
  over a discrete index category and colimits commute with colimits, so for a
  family \((P_s)_{s\in S}\) of persistence modules the canonical comparison
  \(\bigoplus_{s\in S}P_s(\ell^-)\to\bigl(\bigoplus_{s\in S}P_s\bigr)(\ell^-)\)
  is an isomorphism carrying \(\bigoplus_{s\in S}\tau^{P_s}_\ell\) to
  \(\tau^{\oplus_{s}P_s}_\ell\). Take for \(\theta\) the isomorphism
  \(\Theta\) of \Cref{itm:crit-bars}, whose colimit is again an isomorphism by
  functoriality. Both vertical maps of \eqref{eq:tau-natural} are then
  invertible, so \(\tau_{n,\ell}\) is injective, surjective or bijective if and
  only if \(\bigoplus_{s\in S}\tau^{\mathbb{F}_{J_s}}_\ell\) is, and
  \Cref{itm:crit-summands} turns each into a condition on the
  summands, one for each index \(s\in S\) of the barcode multiset
  \(\mathcal{B}\) of \(M\).
\item\label{itm:crit-endpoint} \textbf{What Happens at an Endpoint.} Let
  \(J=[b,b')\subseteq[0,\infty)\) with \(0\le b<b'\le\infty\), the case
  \(b'=\infty\) being \([b,\infty)\), and let \(\ell\in[0,\infty)\). The interval
  module has \((\mathbb{F}_J)\ell=\mathbb{F}\) when \(b\le\ell<b'\) and \(0\)
  otherwise, by definition. Its left limit has
  \((\mathbb{F}_J)(\ell^-)=\mathbb{F}\) when \(b<\ell\le b'\) and \(0\) otherwise,
  which three cases give. For \(\ell\le b\) every \(\ell'<\ell\) has
  \((\mathbb{F}_J)\ell'=0\), so the colimit vanishes. For \(b<\ell\le b'\) the set
  \(\{\ell'\mid b\le\ell'<\ell\}\) is nonempty and cofinal in
  \(\{\ell'\mid\ell'<\ell\}\), and the diagram restricted to it is constant at
  \(\mathbb{F}\) with identity transition maps, so the colimit is \(\mathbb{F}\)
  and every \(c^{\mathbb{F}_J}_{\ell'}\) on that set is an isomorphism. For
  \(b'<\ell\) the set \(\{\ell'\mid b'\le\ell'<\ell\}\) is nonempty and cofinal and
  the diagram is constant at \(0\) on it. Five cases in \(\ell\) exhaust the
  possibilities, and the two values computed above settle the map
  \(\tau^{\mathbb{F}_J}_\ell\) in every one of them but the third row of the
  table that follows:
  \begin{equation}\label{eq:endpoint-table}
  \begin{array}{@{}llll@{}}
    \toprule
    \ell & (\mathbb{F}_J)(\ell^-) & (\mathbb{F}_J)\ell
      & \tau^{\mathbb{F}_J}_\ell\\
    \midrule
    \ell<b & 0 & 0 & \text{bijective}\\
    \rowcolor{tintright}\ell=b & 0 & \mathbb{F} & \text{not surjective}\\
    b<\ell<b' & \mathbb{F} & \mathbb{F} & \text{bijective}\\
    \rowcolor{tintleft}\ell=b'<\infty & \mathbb{F} & 0 & \text{not injective}\\
    b'<\ell & 0 & 0 & \text{bijective}\\
    \bottomrule
  \end{array}
  \end{equation}
  In that third row the map is bijective because
  \(\tau^{\mathbb{F}_J}_\ell\circ c^{\mathbb{F}_J}_{\ell'}=(\mathbb{F}_J)(\ell'\le\ell)\)
  is the identity of \(\mathbb{F}\) whenever \(b\le\ell'<\ell\), and
  \(c^{\mathbb{F}_J}_{\ell'}\) is an isomorphism there. Every \(J_s\) has the
  shape assumed for \(J\) here, by \Cref{itm:crit-shape}, so the births of \(M\)
  named after \Cref{thm:barcode} are the left endpoints of the bars, the deaths
  are their finite right endpoints, and a bar has no further endpoints. Combining \eqref{eq:endpoint-table} with
  \Cref{itm:crit-colim}, a birth at \(\ell\) contributes a summand that is not
  surjective and a death at \(\ell\) a summand that is not injective, and a
  direct sum of maps is surjective or injective only if every summand is, whence
  the two implications
  \begin{equation}\label{eq:endpoint-split}
  \begin{aligned}
    \ell\ \text{a birth of}\ M
      &\;\Longrightarrow\; \tau_{n,\ell}\ \text{is not surjective},\\
    \ell\ \text{a death of}\ M
      &\;\Longrightarrow\; \tau_{n,\ell}\ \text{is not injective}.
  \end{aligned}
  \end{equation}
  The running example of \Cref{fig:critical} has one birth, at \(\ell=0\), and
  two deaths, at \(\ell=1\) and \(\ell=2\), one for each of its two bars of
  finite length and none elsewhere.
\item\label{itm:crit-inclusion} \textbf{The Inclusion.} Let \(X\) be an object
  of \(\Met\) with \(\Ob X\) finite, let \(\mathbb{F}\) be a field and let
  \(n\ge 0\). A birth \(\ell\) of \(\PH_n(X;\mathbb{F})\) makes
  \(\tau_{n,\ell}\) fail to be surjective by \eqref{eq:endpoint-split}, so the
  second half of \Cref{itm:crit-iso} fails and
  \hlright{\MH_{n,\ell}(X;\mathbb{F})\neq0}. A death \(\ell\) makes
  \(\tau_{n,\ell}\) fail to be injective, so the first half fails and
  \hlleft{\MH_{n+1,\ell}(X;\mathbb{F})\neq0}. Every birth lies in the first set
  of the union of \Cref{cor:critical} and every death in the second, so every
  endpoint lies in that union. In the running example of \Cref{fig:critical} the
  two sets are \(\{0\}\) and \(\{1,2\}\), and they meet the births and the
  deaths. At \(\ell=3\) both magnitude groups vanish and the
  barcode does not change, although \(3\) is a block boundary.\qedhere
\end{enumerate}
\end{proof}

\section{Proofs of \texorpdfstring{\Cref{sec:barcodes}}{Section 6}}\label{app:barcodes}

\begin{proof}[Proof of \Cref{thm:barcode}]~
\begin{itemize}
\item \textbf{Existence.} A functor \(M\colon\Vfin^{\op}\to\Vect_{\mathbb{F}}\)
  is a persistence module over the totally ordered set \([0,\infty)\), because
  \(\Vfin^{\op}\) is that set read as a category. Its subset
  \(\mathbb{Q}\cap[0,\infty)\) is countable and dense in the order topology,
  which is the standing hypothesis of
  \cite[\S1]{CrawleyBoevey2015Decomposition}, tameness of \(M\) says that \(M\)
  is pointwise finite-dimensional, and the interval modules of
  \Cref{sec:barcodes} are the interval modules of that paper. So \(M\) is a
  direct sum of interval modules \cite[Thm.~1.1]{CrawleyBoevey2015Decomposition}:
  there are a set \(S\), a family \((J_s)_{s\in S}\) of intervals and an
  isomorphism \(\bigoplus_{s\in S}\mathbb{F}_{J_s}\cong M\), and \(\mathcal{B}\)
  is the multiset that the family \((J_s)_{s\in S}\) indexes. What remains is to
  pin that family down up to a bijection of index sets.
\item \textbf{The Endomorphism Ring of an Interval Module.} Let
  \(J\subseteq[0,\infty)\) be an interval and let
  \(\theta\colon\mathbb{F}_J\to\mathbb{F}_J\) be a natural transformation. For
  \(\ell\notin J\) the space \((\mathbb{F}_J)\ell\) is zero and \(\theta_\ell=0\).
  For \(\ell\in J\) the map \(\theta_\ell\colon\mathbb{F}\to\mathbb{F}\) is
  multiplication by the scalar \(c_\ell=\theta_\ell(1)\). On a morphism
  \(\ell\le\ell'\) of \(\Vfin^{\op}\) whose two ends lie in \(J\), naturality of
  \(\theta\) is the commuting square
  \begin{equation}\label{eq:end-natural}
  \begin{tikzcd}[row sep=2.1em, column sep=3.2em, ampersand replacement=\&]
    (\mathbb{F}_J)\ell\arrow[r,"1_{\mathbb{F}}"]\arrow[d,"\theta_\ell"']
      \&(\mathbb{F}_J)\ell'\arrow[d,"\theta_{\ell'}"]\\
    (\mathbb{F}_J)\ell\arrow[r,"1_{\mathbb{F}}"']\&(\mathbb{F}_J)\ell'
  \end{tikzcd}
  \end{equation}
  whose two horizontal maps are the identity of \(\mathbb{F}\), because the
  transition maps of \(\mathbb{F}_J\) between nonzero terms are. Reading
  \eqref{eq:end-natural} on \(1\in\mathbb{F}\) gives \(c_\ell=c_{\ell'}\). Any
  two elements of \(J\) are comparable and \(J\) is nonempty, so one scalar
  \(c\in\mathbb{F}\) carries every component. In the other direction each
  \(c\in\mathbb{F}\) defines such a \(\theta\), since on a morphism with both
  ends in \(J\) the two composites of \eqref{eq:end-natural} are multiplication
  by \(c\), and otherwise one of the two spaces is zero and both composites
  vanish. Taken together, the two directions make the map of scalars into
  endomorphisms of the interval module \(\mathbb{F}_J\) given by
  \begin{equation}\label{eq:end-interval}
    \Xi\colon\mathbb{F}\longrightarrow\End(\mathbb{F}_J),
    \qquad
    c\longmapsto\bigl(c\cdot 1_{(\mathbb{F}_J)\ell}\bigr)_{\ell\in[0,\infty)}
  \end{equation}
  a bijection, surjective by the first direction and injective because \(c\) is
  recovered from \(\Xi(c)\) as its value on \(1\) at any \(\ell\in J\). It is
  additive, it sends \(1\) to the identity, and it sends a product \(cc'\) to the
  composite \(\Xi(c)\circ\Xi(c')\), the two agreeing component by component and
  multiplying by \(cc'\) where \((\mathbb{F}_J)\ell=\mathbb{F}\), so the map
  \(\Xi\) is itself an isomorphism of rings.
\item \textbf{Locality and Indecomposability.} A ring is local when it has a
  unique maximal proper left ideal. Every nonzero element of the field
  \(\mathbb{F}\) is a unit, so the only left ideals of \(\mathbb{F}\) are \(0\)
  and \(\mathbb{F}\), and \(0\) is the only proper one. By
  \eqref{eq:end-interval} the ring \(\End(\mathbb{F}_J)\) is local. Suppose now
  \(\mathbb{F}_J\cong N\oplus N'\) in \(\Vect_{\mathbb{F}}^{\Vfin^{\op}}\), with
  the four maps of that biproduct arranged as
  \begin{equation}\label{eq:biproduct}
  \begin{tikzcd}[column sep=3.6em, ampersand replacement=\&]
    N\arrow[r,shift left=0.55ex,hook,"\kappa"]
      \&\mathbb{F}_J\arrow[l,shift left=0.55ex,two heads,"p"]
        \arrow[r,shift left=0.55ex,two heads,"p'"]
      \&N'\arrow[l,shift left=0.55ex,hook',"\kappa'"]
  \end{tikzcd}
  \qquad
  p\kappa=1_N,\quad p'\kappa'=1_{N'},\quad
  \kappa p+\kappa'p'=1_{\mathbb{F}_J}.
  \end{equation}
  Put \(\varepsilon=\kappa p\in\End(\mathbb{F}_J)\), idempotent because
  \(\varepsilon^2=\kappa(p\kappa)p=\kappa p\), whose scalar
  \(c=\Xi^{-1}\varepsilon\) satisfies \(c^2=c\), so \(c(c-1)=0\) and
  \(c\in\{0,1\}\) in a field. If \(c=0\) then \(\varepsilon=0\) and
  \(1_N=p\kappa p\kappa=p\varepsilon\kappa=0\), so the summand \(N\) is zero.
  If \(c=1\) then \(\kappa'p'=1_{\mathbb{F}_J}-\varepsilon=0\) and
  \(1_{N'}=p'\kappa'p'\kappa'=0\), so the summand \(N'\) is zero. One of the two
  is zero in either case, and \(\mathbb{F}_J\) is itself nonzero because \(J\)
  is nonempty, so the interval module \(\mathbb{F}_J\) is indecomposable in
  the category \(\Vect_{\mathbb{F}}^{\Vfin^{\op}}\) of persistence modules.
\item \textbf{Uniqueness.} Let two decompositions of the module \(M\) into
  interval modules be given,
  \begin{equation}\label{eq:two-decompositions}
  \begin{tikzcd}[column sep=3.0em, ampersand replacement=\&]
    \bigoplus_{s\in S}\mathbb{F}_{J_s}\arrow[r,"\cong"]\&M
      \&\bigoplus_{t\in T}\mathbb{F}_{J'_t}\arrow[l,"\cong"'] .
  \end{tikzcd}
  \end{equation}
  Every summand on either side has local endomorphism ring and is
  indecomposable, by the two steps above, so the
  Krull--Remak--Schmidt--Azumaya theorem
  \cite[Thm.~1]{Azumaya1950KrullRemakSchmidt}, applied in the additive category
  \(\Vect_{\mathbb{F}}^{\Vfin^{\op}}\), supplies a bijection
  \(\beta\colon S\to T\) with
  \(\mathbb{F}_{J_s}\cong\mathbb{F}_{J'_{\beta(s)}}\) for every \(s\in S\).
  That theorem is invoked for persistence modules in this way in
  \cite[\S1]{CrawleyBoevey2015Decomposition}, and
  Botnan and Crawley-Boevey~\cite[Thm.~1.1]{BotnanCrawleyBoevey2020Decomposition} prove the existence of
  a decomposition into indecomposables with local endomorphism rings for
  pointwise finite-dimensional persistence modules over any small indexing
  category, which covers \(\Vect_{\mathbb{F}}^{\Vfin^{\op}}\). An isomorphism
  \(\theta\colon\mathbb{F}_J\to\mathbb{F}_{J'}\) of persistence modules has every
  component \(\theta_\ell\) an isomorphism of vector spaces, so
  \(\dim_{\mathbb{F}}(\mathbb{F}_J)\ell=\dim_{\mathbb{F}}(\mathbb{F}_{J'})\ell\)
  for every \(\ell\in[0,\infty)\), the left side being \(1\) or \(0\) according
  as \(\ell\) lies in \(J\) or not and the right side according as \(\ell\) lies
  in \(J'\) or not. The two intervals agree, the bijection
  \(\beta\) matches equal intervals, and the multiset \(\mathcal{B}\) is
  determined by \(M\) up to bijection of index sets.\qedhere
\end{itemize}
\end{proof}

\begin{proof}[Proof of \Cref{prop:interleaving-vcat}]~
\begin{itemize}
\item \textbf{Notation.} For \(M,N\colon\Vfin^{\op}\to\A\) collect the
  interleaving parameters into the set
  \begin{equation}\label{eq:int-set}
    S(M,N)=\bigl\{\delta\in[0,\infty)\bigm|
      \text{a }\delta\text{-interleaving of }M\text{ and }N\text{ exists}\bigr\},
  \end{equation}
  so that \(d_I(M,N)=\inf S(M,N)\) with \(\inf\varnothing=\infty\), which is
  \Cref{def:interleaving} read as one set. For \(\delta\in[0,\infty)\) let
  \(\Sigma_\delta\colon\Vfin^{\op}\to\Vfin^{\op}\) be the shift functor, with
  \(\Sigma_\delta\ell=\ell+\delta\) on objects and
  \(\Sigma_\delta(\ell\le\ell')=(\ell+\delta\le\ell'+\delta)\) on morphisms.
  Shifts compose strictly, \(\Sigma_\delta\Sigma_\varepsilon
  =\Sigma_{\delta+\varepsilon}\) and \(\Sigma_0=1\), and for each \(\delta\)
  there is one and only one natural transformation
  \(\sigma_\delta\colon 1\Rightarrow\Sigma_\delta\), the one whose component at
  \(\ell\) is the morphism \(\ell\le\ell+\delta\) of \(\Vfin^{\op}\): its
  naturality squares commute, as every diagram in a thin category does. A family
  \(\varphi_\ell\colon M\ell\to N(\ell+\delta)\) natural in \(\ell\) is a
  natural transformation \(\varphi\colon M\Rightarrow N\Sigma_\delta\), and \Cref{def:interleaving} reads
  \begin{equation}\label{eq:int-conditions}
    (\psi\Sigma_\delta)\cdot\varphi=M\sigma_{2\delta},
    \qquad
    (\varphi\Sigma_\delta)\cdot\psi=N\sigma_{2\delta},
  \end{equation}
  where \(\cdot\) is vertical composition, \(\psi\Sigma_\delta\) has the
  components \(\psi_{\ell+\delta}\colon N(\ell+\delta)\to M(\ell+2\delta)\),
  and \(M\sigma_{2\delta}\) has the components
  \(M(\ell\le\ell+2\delta)\colon M\ell\to M(\ell+2\delta)\), one for each \(\ell\).
\item \textbf{Diagonal.} Let \(M\colon\Vfin^{\op}\to\A\). Take \(\delta=0, \varphi=\psi=1_M\colon M\Rightarrow M=M\Sigma_0\). \eqref{eq:int-conditions} holds,
  \[
    (1_M\Sigma_0)\cdot 1_M=1_M=M1=M\sigma_0 ,
  \]
  because the component of \(\sigma_0\) at \(\ell\) is the identity of \(\ell\)
  and \(M\), a functor, sends it to \(1_{M\ell}\). So \(0\in S(M,M)\), and the
  infimum of a subset of \([0,\infty)\) that contains \(0\) is \(0\), which
  computes the diagonal, the first of the three conditions, as
  \(d_I(M,M)=\inf S(M,M)=0\).
\item \textbf{Symmetry.} The assignment \((\varphi,\psi)\mapsto(\psi,\varphi)\)
  sends a \(\delta\)-interleaving of \(M\) and \(N\) to a pair that satisfies
  the two conditions \eqref{eq:int-conditions} in the other order, which is a
  \(\delta\)-interleaving of \(N\) and \(M\), and it is its own inverse. So
  \(S(M,N)=S(N,M)\) as subsets of \([0,\infty)\), and applying \(\inf\) to both
  sides gives \(d_I(M,N)=\inf S(M,N)=\inf S(N,M)=d_I(N,M)\).
\item \textbf{Interleavings Compose.} For \(M,N,P\colon\Vfin^{\op}\to\A\) let
  \(\delta\in S(M,N)\) through the pair \((\varphi,\psi)\) and
  \(\varepsilon\in S(N,P)\) through the pair \((\varphi',\psi')\). Put the two
  vertical composites
  \[
    \Phi=(\varphi'\Sigma_\delta)\cdot\varphi\colon
      M\Longrightarrow P\Sigma_\varepsilon\Sigma_\delta
      =P\Sigma_{\delta+\varepsilon},
    \qquad
    \Psi=(\psi\Sigma_\varepsilon)\cdot\psi'\colon
      P\Longrightarrow M\Sigma_\delta\Sigma_\varepsilon
      =M\Sigma_{\delta+\varepsilon},
  \]
  with components \(\Phi_\ell=\varphi'_{\ell+\delta}\varphi_\ell\) and
  \(\Psi_\ell=\psi_{\ell+\varepsilon}\psi'_\ell\), vertical composites of
  whiskered natural transformations and so natural. The first condition of
  \eqref{eq:int-conditions} for the pair \((\Phi,\Psi)\) is the chain of
  equalities of natural transformations \(M\Rightarrow M\Sigma_{2(\delta+\varepsilon)}\),
  \begin{align}
    (\Psi\Sigma_{\delta+\varepsilon})\cdot\Phi
    &=(\psi\Sigma_\varepsilon\Sigma_\varepsilon\Sigma_\delta)
      \cdot(\psi'\Sigma_\varepsilon\Sigma_\delta)
      \cdot(\varphi'\Sigma_\delta)\cdot\varphi
      &&\text{definition of }\Phi,\Psi,\notag\\
    &=(\psi\Sigma_{2\varepsilon}\Sigma_\delta)
      \cdot(N\sigma_{2\varepsilon}\Sigma_\delta)\cdot\varphi
      &&\text{interleaving of }N\text{ and }P,\notag\\
    &=(M\Sigma_\delta\sigma_{2\varepsilon}\Sigma_\delta)
      \cdot(\psi\Sigma_\delta)\cdot\varphi
      &&\text{interchange},\label{eq:int-comp}\\
    &=(M\Sigma_\delta\sigma_{2\varepsilon}\Sigma_\delta)
      \cdot(M\sigma_{2\delta})
      &&\text{interleaving of }M\text{ and }N,\notag\\
    &=M\sigma_{2(\delta+\varepsilon)},
      &&\Vfin^{\op}\text{ is thin,}\notag
  \end{align}
  whose interchange step slides the two 2-cells \(\psi\) and
  \(\sigma_{2\varepsilon}\), which sit on different wires, past one another, and
  whose last step composes two transformations \(1\Rightarrow\Sigma_{2(\delta
  +\varepsilon)}\) of which there is only one. The same five lines with
  \(M,\varphi,\psi,\delta\) and \(P,\psi',\varphi',\varepsilon\) exchanged give
  \((\Phi\Sigma_{\delta+\varepsilon})\cdot\Psi=P\sigma_{2(\delta+\varepsilon)}\).
  So \((\Phi,\Psi)\) is a \((\delta+\varepsilon)\)-interleaving of \(M\) and
  \(P\), and \(\delta+\varepsilon\in S(M,P)\). In the calculus of the proof of
  \Cref{lem:term-monad}, with the codomain \(\A\) left of the module wire and
  the shift wires of \(\Vfin^{\op}\) right of it, the transformation
  \(\sigma_{2\varepsilon}\colon 1\Rightarrow\Sigma_\varepsilon\Sigma_\varepsilon\)
  is a cap, a dot with the two wires of \(\Sigma_\varepsilon\Sigma_\varepsilon\)
  below it and none above. Cool tracks the pair \((\varphi,\psi)\) of \(M\) and
  \(N\), warm the pair \((\varphi',\psi')\) of \(N\) and \(P\), and
  \eqref{eq:int-comp} reads
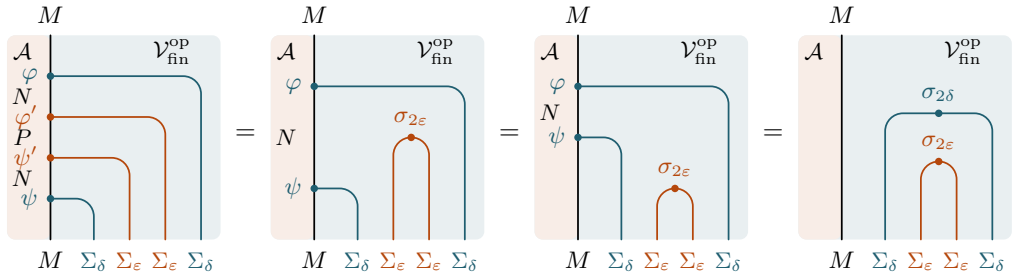
  \[
  \renewcommand\sdxunit{0.675}\renewcommand\sdyunit{0.675}
  \begin{sdiag}{4}
    \sdcanvas{sddomain}{-0.47}{2.95}
    \sdstrip{sdcodomain}{-0.85}{0}
    \sdwire{0}{0}
    \node[sddotsty,fill=complement] at (0,-0.8) {};
    \node[sdlab,complement,left=\sdsidegap*\sdxunit cm] at (0,-0.8) {$\varphi$};
    \sdleg[sd,complement]{0}{-0.8}{2.95}
    \node[sddotsty,fill=accent] at (0,-1.6) {};
    \node[sdlab,accent,left=\sdsidegap*\sdxunit cm] at (0,-1.6) {$\varphi'$};
    \sdleg[sd,accent]{0}{-1.6}{2.25}
    \node[sddotsty,fill=accent] at (0,-2.4) {};
    \node[sdlab,accent,left=\sdsidegap*\sdxunit cm] at (0,-2.4) {$\psi'$};
    \sdleg[sd,accent]{0}{-2.4}{1.55}
    \node[sddotsty,fill=complement] at (0,-3.2) {};
    \node[sdlab,complement,left=\sdsidegap*\sdxunit cm] at (0,-3.2) {$\psi$};
    \sdleg[sd,complement]{0}{-3.2}{0.85}
    \node[sdlab,font=\footnotesize] at (-0.55,-1.2) {$N$};
    \node[sdlab,font=\footnotesize] at (-0.55,-2.0) {$P$};
    \node[sdlab,font=\footnotesize] at (-0.55,-2.8) {$N$};
    \sdtop{0}{M}
    \sdbot{0}{M}
    \node[sdlab,font=\footnotesize,complement,below=\sdlabelgap*\sdyunit cm] at (0.85,\sdfloor) {$\Sigma_\delta$};
    \node[sdlab,font=\footnotesize,accent,below=\sdlabelgap*\sdyunit cm] at (1.55,\sdfloor) {$\Sigma_\varepsilon$};
    \node[sdlab,font=\footnotesize,accent,below=\sdlabelgap*\sdyunit cm] at (2.25,\sdfloor) {$\Sigma_\varepsilon$};
    \node[sdlab,font=\footnotesize,complement,below=\sdlabelgap*\sdyunit cm] at (2.95,\sdfloor) {$\Sigma_\delta$};
    \node[sdname] at (-0.55,-0.32) {$\A$};
    \node[sdname] at (2.38,-0.32) {$\Vfin^{\op}$};
  \end{sdiag}
  \sdeqq
  \begin{sdiag}{4}
    \sdcanvas{sddomain}{-0.47}{2.95}
    \sdstrip{sdcodomain}{-0.85}{0}
    \sdwire{0}{0}
    \node[sddotsty,fill=complement] at (0,-1.0) {};
    \node[sdlab,complement,left=\sdsidegap*\sdxunit cm] at (0,-1.0) {$\varphi$};
    \sdleg[sd,complement]{0}{-1.0}{2.95}
    \node[sddotsty,fill=accent] at (1.9,-2.0) {};
    \node[sdlab,accent,above=\sdsidegap*\sdxunit cm] at (1.9,-2.0) {$\sigma_{2\varepsilon}$};
    \sdleg[sd,accent]{1.9}{-2.0}{1.55}
    \sdleg[sd,accent]{1.9}{-2.0}{2.25}
    \node[sddotsty,fill=complement] at (0,-3.0) {};
    \node[sdlab,complement,left=\sdsidegap*\sdxunit cm] at (0,-3.0) {$\psi$};
    \sdleg[sd,complement]{0}{-3.0}{0.85}
    \node[sdlab,font=\footnotesize] at (-0.55,-2.0) {$N$};
    \sdtop{0}{M}
    \sdbot{0}{M}
    \node[sdlab,font=\footnotesize,complement,below=\sdlabelgap*\sdyunit cm] at (0.85,\sdfloor) {$\Sigma_\delta$};
    \node[sdlab,font=\footnotesize,accent,below=\sdlabelgap*\sdyunit cm] at (1.55,\sdfloor) {$\Sigma_\varepsilon$};
    \node[sdlab,font=\footnotesize,accent,below=\sdlabelgap*\sdyunit cm] at (2.25,\sdfloor) {$\Sigma_\varepsilon$};
    \node[sdlab,font=\footnotesize,complement,below=\sdlabelgap*\sdyunit cm] at (2.95,\sdfloor) {$\Sigma_\delta$};
    \node[sdname] at (-0.55,-0.32) {$\A$};
    \node[sdname] at (2.38,-0.32) {$\Vfin^{\op}$};
  \end{sdiag}
  \sdeqq
  \begin{sdiag}{4}
    \sdcanvas{sddomain}{-0.47}{2.95}
    \sdstrip{sdcodomain}{-0.85}{0}
    \sdwire{0}{0}
    \node[sddotsty,fill=complement] at (0,-1.0) {};
    \node[sdlab,complement,left=\sdsidegap*\sdxunit cm] at (0,-1.0) {$\varphi$};
    \sdleg[sd,complement]{0}{-1.0}{2.95}
    \node[sddotsty,fill=complement] at (0,-2.0) {};
    \node[sdlab,complement,left=\sdsidegap*\sdxunit cm] at (0,-2.0) {$\psi$};
    \sdleg[sd,complement]{0}{-2.0}{0.85}
    \node[sddotsty,fill=accent] at (1.9,-3.0) {};
    \node[sdlab,accent,above=\sdsidegap*\sdxunit cm] at (1.9,-3.0) {$\sigma_{2\varepsilon}$};
    \sdleg[sd,accent]{1.9}{-3.0}{1.55}
    \sdleg[sd,accent]{1.9}{-3.0}{2.25}
    \node[sdlab,font=\footnotesize] at (-0.55,-1.5) {$N$};
    \sdtop{0}{M}
    \sdbot{0}{M}
    \node[sdlab,font=\footnotesize,complement,below=\sdlabelgap*\sdyunit cm] at (0.85,\sdfloor) {$\Sigma_\delta$};
    \node[sdlab,font=\footnotesize,accent,below=\sdlabelgap*\sdyunit cm] at (1.55,\sdfloor) {$\Sigma_\varepsilon$};
    \node[sdlab,font=\footnotesize,accent,below=\sdlabelgap*\sdyunit cm] at (2.25,\sdfloor) {$\Sigma_\varepsilon$};
    \node[sdlab,font=\footnotesize,complement,below=\sdlabelgap*\sdyunit cm] at (2.95,\sdfloor) {$\Sigma_\delta$};
    \node[sdname] at (-0.55,-0.32) {$\A$};
    \node[sdname] at (2.38,-0.32) {$\Vfin^{\op}$};
  \end{sdiag}
  \sdeqq
  \begin{sdiag}{4}
    \sdcanvas{sddomain}{-0.47}{2.95}
    \sdstrip{sdcodomain}{-0.85}{0}
    \sdwire{0}{0}
    \node[sddotsty,fill=complement] at (1.9,\sdgoldup) {};
    \node[sdlab,complement,above=\sdsidegap*\sdxunit cm] at (1.9,\sdgoldup) {$\sigma_{2\delta}$};
    \sdleg[sd,complement]{1.9}{\sdgoldup}{0.85}
    \sdleg[sd,complement]{1.9}{\sdgoldup}{2.95}
    \node[sddotsty,fill=accent] at (1.9,\sdgold) {};
    \node[sdlab,accent,above=\sdsidegap*\sdxunit cm] at (1.9,\sdgold) {$\sigma_{2\varepsilon}$};
    \sdleg[sd,accent]{1.9}{\sdgold}{1.55}
    \sdleg[sd,accent]{1.9}{\sdgold}{2.25}
    \sdtop{0}{M}
    \sdbot{0}{M}
    \node[sdlab,font=\footnotesize,complement,below=\sdlabelgap*\sdyunit cm] at (0.85,\sdfloor) {$\Sigma_\delta$};
    \node[sdlab,font=\footnotesize,accent,below=\sdlabelgap*\sdyunit cm] at (1.55,\sdfloor) {$\Sigma_\varepsilon$};
    \node[sdlab,font=\footnotesize,accent,below=\sdlabelgap*\sdyunit cm] at (2.25,\sdfloor) {$\Sigma_\varepsilon$};
    \node[sdlab,font=\footnotesize,complement,below=\sdlabelgap*\sdyunit cm] at (2.95,\sdfloor) {$\Sigma_\delta$};
    \node[sdname] at (-0.55,-0.32) {$\A$};
    \node[sdname] at (2.38,-0.32) {$\Vfin^{\op}$};
  \end{sdiag}
  \,.
  \]
  No two wires cross: a crossing would name a 2-cell exchanging the two
  functors it swaps, and the calculus supplies none, even for shifts that
  commute. Each step slides a dot past a dot on another wire, the interchange
  law, and each pair contracts to the cap of its own shift. The last picture
  is \(M\) beside the nested caps \(\sigma_{2\delta}\) and
  \(\sigma_{2\varepsilon}\), one 2-cell \(M\sigma_{2(\delta+\varepsilon)}\)
  because \(\Vfin^{\op}\) is thin, with component
  \(M(\ell\le\ell+2(\delta+\varepsilon))\) at \(\ell\).
\item \textbf{Triangle Inequality.} If \(S(M,N)\) or \(S(N,P)\) is empty, then
  \(d_I(M,N)+d_I(N,P)=\infty\) and there is nothing to prove. Otherwise the
  previous item gives the containment
  \(S(M,N)+S(N,P)\subseteq S(M,P)\) of subsets of \([0,\infty)\), where the sum
  of two sets is the set
  \(\{\delta+\varepsilon\mid\delta\in S(M,N),\,\varepsilon\in S(N,P)\}\) of
  their sums, and the infima line up as
  \begin{align}
    d_I(M,P)=\inf S(M,P)
      &\le\inf\bigl(S(M,N)+S(N,P)\bigr)
      &&\text{containment},\notag\\
      &=\inf S(M,N)+\inf S(N,P)
      &&\text{infima add},\label{eq:int-inf}\\
      &=d_I(M,N)+d_I(N,P),
      &&\text{definition of }d_I.\notag
  \end{align}
  Infima of nonempty subsets \(A,B\subseteq[0,\infty)\) add: \(\inf A+\inf B\)
  bounds every sum \(\delta+\varepsilon\) with \(\delta\in A\) and
  \(\varepsilon\in B\) from below, and for every \(\eta>0\) there are
  \(\delta<\inf A+\eta/2\) and \(\varepsilon<\inf B+\eta/2\), so
  \(\inf(A+B)\le\inf A+\inf B+\eta\) for every \(\eta>0\).
\item \textbf{Enrichment.} By the unwinding in \Cref{sec:base}, a symmetric
  \(\V\)-category structure on the class of persistence modules in \(\A\) is a
  function \(d\colon\Ob(\A^{\Vfin^{\op}})\times\Ob(\A^{\Vfin^{\op}})\to\Rp\)
  with \(d(M,M)=0\), with \(d(M,P)\le d(M,N)+d(N,P)\) and with
  \(d(M,N)=d(N,M)\), the identity, the composition and the symmetry of the
  structure. The parts \emph{Diagonal}, \emph{Triangle Inequality} and
  \emph{Symmetry} supply the three conditions for \(d=d_I\).
\item \textbf{Separatedness Fails.} In \(\A=\Vect_{\mathbb{F}}\) take the
  interval modules \(M=\mathbb{F}_{[0,1)}\) and \(N=\mathbb{F}_{[0,1]}\). They
  are not isomorphic, since an isomorphism of persistence modules has
  isomorphisms as components and \(M(1)=0\) while \(N(1)=\mathbb{F}\). For
  \(\delta>1\) the zero families form a \(\delta\)-interleaving, both sides of
  each condition \eqref{eq:int-conditions} being zero maps, because \(\ell\) and
  \(\ell+2\delta\) never lie in one of the two intervals together. For
  \(0<\delta\le 1\) put
  \[
    \varphi_\ell=
    \begin{cases}
      1_{\mathbb{F}}, & \ell\le 1-\delta,\\
      0, & \text{otherwise,}
    \end{cases}
    \qquad
    \psi_\ell=
    \begin{cases}
      1_{\mathbb{F}}, & \ell<1-\delta,\\
      0, & \text{otherwise.}
    \end{cases}
  \]
  Both families are natural and satisfy \eqref{eq:int-conditions}: on each
  square and each condition, either every space in sight is \(\mathbb{F}\) and
  every map in sight is \(1_{\mathbb{F}}\), or the target of the composite is
  the zero space, or the first factor of the composite is a zero map, and the
  three cases are decided by comparing \(\ell\), \(\ell+\delta\) and
  \(\ell+2\delta\) with \(1\). So \(S(M,N)=(0,\infty)\) and \(d_I(M,N)=0\) with
  \(M\not\cong N\), and \(d_I\) becomes a metric only after passing to the
  quotient by the relation \(d_I=0\).\qedhere
\end{itemize}
\end{proof}

\begin{proof}[Proof of \Cref{thm:distortion}]~
\begin{itemize}
\item \textbf{One Set of Simplices, Two Length Functions.} The objects \(X\)
  and \(Y\) share their underlying set, \(\Ob X=\Ob Y\). An \(m\)-simplex of
  either length nerve is a map \(\mathbf{x}\colon[m]\to\Ob X\), written
  \(\mathbf{x}=\langle x_0,\dots,x_m\rangle\), and it is nondegenerate when
  \(x_{i-1}\neq x_i\) for \(1\le i\le m\), a condition that mentions no metric.
  So \(\Nv(X)_m=\Nv(Y)_m\) for every \(m\ge0\), with the same degeneracies, and
  on this one set of simplices live the two length functions
  \begin{gather*}
    \lambda_X,\lambda_Y\colon\textstyle\bigcup_{m\ge0}\Nv(X)_m\longrightarrow\Rp,\\
    \lambda_X\mathbf{x}=\sum_{i=1}^{m}d_X(x_{i-1},x_i),
    \qquad
    \lambda_Y\mathbf{x}=\sum_{i=1}^{m}d_Y(x_{i-1},x_i).
  \end{gather*}
\item \textbf{The Two Lengths Differ by at Most \(m\delta\).} Let
  \(\mathbf{x}=\langle x_0,\dots,x_m\rangle\) be an \(m\)-simplex. It has
  \(m\) consecutive pairs, so each of the two lengths is a sum of
  \(m\) distances, and adding the hypothesis
  \(d_Y(x_{i-1},x_i)\le d_X(x_{i-1},x_i)+\delta\) over \(1\le i\le m\) in
  \(\Rp\) gives
  \begin{equation}\label{eq:length-shift}
    \lambda_Y\mathbf{x}=\sum_{i=1}^{m}d_Y(x_{i-1},x_i)
      \le\sum_{i=1}^{m}\bigl(d_X(x_{i-1},x_i)+\delta\bigr)
      =\lambda_X\mathbf{x}+m\delta ,
  \end{equation}
  and the other hypothesis gives \(\lambda_X\mathbf{x}\le\lambda_Y\mathbf{x}+m\delta\)
  the same way. Put \(\Delta=(n+1)\delta\). Since \(m\delta\le\Delta\) for
  \(m\le n+1\), and since a bound in \(\Rp\) only weakens when the added
  constant grows, both bounds hold with \(\Delta\) in place of \(m\delta\) in
  every simplicial degree \(m\le n+1\) at once, which is the range of degrees
  the next part keeps.
\item \textbf{Truncation.} Within this proof abbreviate by
  \(C'_\bullet(X)\subseteq C_\bullet(X)\) the subcomplex with
  \(C'_m(X)=C_m(X)\) for \(m\le n+1\) and \(C'_m(X)=0\) for \(m>n+1\), a
  subcomplex because \(\partial\) lowers the degree, and let
  \(F_\ell C'_\bullet(X)=F_\ell C_\bullet(X)\cap C'_\bullet(X)\) be the
  filtration \eqref{eq:filtered-chains} read inside it. Cycles and boundaries
  in degree \(n\) are computed from the degrees \(n\) and \(n+1\) and the
  differentials between them and into degree \(n-1\), which the truncation
  leaves unchanged, so
  \(H_n\bigl(F_\ell C'_\bullet(X)\bigr)=H_n\bigl(F_\ell C_\bullet(X)\bigr)
  =\PH_n(X)(\ell)\) for every \(\ell\), and the same equality holds for \(Y\).
  The truncation therefore changes neither \(\PH_n(X)\) nor \(\PH_n(Y)\).
\item \textbf{The Interleaving Diagram.} For \(\ell\in[0,\infty)\) define on
  the common basis the two maps
  \[
    \varphi_\ell\colon F_\ell C'_\bullet(X)\longrightarrow
      F_{\ell+\Delta}C'_\bullet(Y),
    \qquad
    \psi_\ell\colon F_\ell C'_\bullet(Y)\longrightarrow
      F_{\ell+\Delta}C'_\bullet(X)
  \]
  by \(\varphi_\ell\mathbf{x}=\mathbf{x}\) and
  \(\psi_\ell\mathbf{x}=\mathbf{x}\). Both are well defined: a nondegenerate
  \(\mathbf{x}\) of degree \(m\le n+1\) with \(\lambda_X\mathbf{x}\le\ell\) has
  \(\lambda_Y\mathbf{x}\le\ell+\Delta\) by the second part, and with the roles
  of \(X\) and \(Y\) exchanged for \(\psi_\ell\). Both are chain maps: the
  differential of either side is \(\partial=\sum_{i=0}^{m}(-1)^id_i\) with the
  same face maps of the common simplices, and faces do not raise either length,
  by \Cref{lem:filtration}\Cref{itm:len-ops-stmt} for \(X\) and for \(Y\). Together with the
  inclusions \(\iota\colon F_\ell C'_\bullet\subseteq F_{\ell+\Delta}C'_\bullet\)
  of the filtration \eqref{eq:filtered-chains}, the maps \(\varphi_\ell\) and
  \(\psi_\ell\) form, for every \(\ell\), the diagram
  \begin{equation}\label{eq:dist-diagram}
  \begin{tikzcd}[row sep=3.7em, column sep=2.8em, ampersand replacement=\&,
                 labels={inner sep=1.6pt}]
    F_\ell C'_\bullet(X)\arrow[r,hook,"\iota"]
      \arrow[dr,accent,"\varphi_\ell"{pos=0.15, anchor=north east}]
    \&F_{\ell+\Delta}C'_\bullet(X)\arrow[r,hook,"\iota"]
      \arrow[dr,accent,"\varphi_{\ell+\Delta}"{pos=0.15, anchor=north east}]
    \&F_{\ell+2\Delta}C'_\bullet(X)\\
    F_\ell C'_\bullet(Y)\arrow[r,hook,"\iota"']
      \arrow[ur,complement,"\psi_\ell"{pos=0.15, anchor=south east},
             crossing over]
    \&F_{\ell+\Delta}C'_\bullet(Y)\arrow[r,hook,"\iota"']
      \arrow[ur,complement,"\psi_{\ell+\Delta}"{pos=0.15, anchor=south east},
             crossing over]
    \&F_{\ell+2\Delta}C'_\bullet(Y),
  \end{tikzcd}
  \end{equation}
  in which every map, horizontal or diagonal, sends a basis simplex
  \(\mathbf{x}\) to the simplex \(\mathbf{x}\). Every triangle of
  \eqref{eq:dist-diagram} therefore commutes, and exchanging the two rows and
  the two letters \(\varphi\) and \(\psi\) is the symmetry that exchanges the
  two hypotheses on \(d_X\) and \(d_Y\). The two triangles that cross
  \eqref{eq:dist-diagram} from \(F_\ell\) to \(F_{\ell+2\Delta}\) read
  \[
    \psi_{\ell+\Delta}\circ\varphi_\ell=\iota\colon
      F_\ell C'_\bullet(X)\lhook\joinrel\longrightarrow F_{\ell+2\Delta}C'_\bullet(X),
    \qquad
    \varphi_{\ell+\Delta}\circ\psi_\ell=\iota\colon
      F_\ell C'_\bullet(Y)\lhook\joinrel\longrightarrow F_{\ell+2\Delta}C'_\bullet(Y),
  \]
  and the two quadrilaterals of \eqref{eq:dist-diagram} with two horizontal and
  two diagonal sides state that \(\varphi\) and \(\psi\) are natural in
  \(\ell\), one quadrilateral for each of the two families.
\item \textbf{Homology.} Apply \(H_n\) to \eqref{eq:dist-diagram} and read its
  corners through the truncation part. The upper row becomes the transition
  maps of \(\PH_n(X)\), the lower row those of \(\PH_n(Y)\), the diagonals two
  families of group homomorphisms, and the diagram itself
  \begin{equation}\label{eq:dist-homology}
  \begin{tikzcd}[row sep=3.7em, column sep=2.2em, ampersand replacement=\&,
                 labels={inner sep=1.6pt}]
    \PH_n(X)(\ell)\arrow[r,"\tau"]
      \arrow[dr,accent,"H_n(\varphi_\ell)"{pos=0.14, anchor=north east}]
    \&\PH_n(X)(\ell+\Delta)\arrow[r,"\tau"]
      \arrow[dr,accent,"H_n(\varphi_{\ell+\Delta})"{pos=0.14, anchor=north east}]
      \arrow[d,comm]
    \&\PH_n(X)(\ell+2\Delta)\\
    \PH_n(Y)(\ell)\arrow[r,"\tau"']
      \arrow[ur,complement,"H_n(\psi_\ell)"{pos=0.14, anchor=south east},
             crossing over]
    \&\PH_n(Y)(\ell+\Delta)\arrow[r,"\tau"']
      \arrow[ur,complement,"H_n(\psi_{\ell+\Delta})"{pos=0.14, anchor=south east},
             crossing over]
    \&\PH_n(Y)(\ell+2\Delta)\mathrlap{\,,}
  \end{tikzcd}
  \end{equation}
  in which every \(\tau\) is the transition map of the module named at its
  source. Every cell of \eqref{eq:dist-homology} is the \(H_n\)-image of a cell
  of \eqref{eq:dist-diagram} and therefore commutes, the circled arrow marking
  the two quadrilaterals, which share that centre and say that the two families
  are natural in \(\ell\). The two triangles read
  \(H_n(\psi_{\ell+\Delta})\circ H_n(\varphi_\ell)=\PH_n(X)(\ell\le\ell+2\Delta)\)
  and
  \(H_n(\varphi_{\ell+\Delta})\circ H_n(\psi_\ell)=\PH_n(Y)(\ell\le\ell+2\Delta)\),
  so the two families are a \(\Delta\)-interleaving of \(\PH_n(X)\) and
  \(\PH_n(Y)\). With the set \eqref{eq:int-set} of the preceding proof
  \((n+1)\delta=\Delta\in S\bigl(\PH_n(X),\PH_n(Y)\bigr)\), and the infimum
  \(d_I\bigl(\PH_n(X),\PH_n(Y)\bigr)\) is at most
  \((n+1)\delta\).\qedhere
\end{itemize}
\end{proof}

\section{Proofs of \texorpdfstring{\Cref{sec:composite}}{Section 7}}\label{app:composite}

\begin{proof}[Proof of \Cref{thm:composite}]~
\begin{enumerate}
\item\label{itm:comp-functor} The three factors of the claimed composite are
  functors. \(T_U=G_U\Free_U\colon\Met\to\Met\) is one by \Cref{thm:free}\Cref{itm:forgetfulfunctor},
  being a composite of the two adjoints. \(\Nv(-)^{\le-}\colon\Met\to
  \SSet^{\Vfin^{\op}}\) is one by \Cref{lem:filtration}: its part \Cref{itm:filt-functor} makes
  \(\Nv(X)^{\le-}\) an object of \(\SSet^{\Vfin^{\op}}\), its part \Cref{itm:filt-map} makes
  \(\Nv(F)^{\le-}\) a morphism there, and its part \Cref{itm:filt-nerve} gives the functor
  identities. Passing to normalised chains and taking homology are the functors
  \(C_\bullet\colon\SSet\to\mathrm{Ch}(\Ab)\) and
  \(H_n\colon\mathrm{Ch}(\Ab)\to\Ab\) applied pointwise in \(\ell\), so
  postcomposition with \(H_nC_\bullet\) is a functor
  \(\SSet^{\Vfin^{\op}}\to\Ab^{\Vfin^{\op}}\). A composite of functors is a
  functor, and its value at an object \(A\) of \(\Met\) is
  \[
    \ell\longmapsto H_n\bigl(C_\bullet(\Nv(T_UA)^{\le\ell})\bigr)
    =H_n\bigl(F_\ell C_\bullet(T_UA)\bigr)=\PH_n(T_UA)(\ell),
  \]
  the persistence module \(\PH_n(T_UA)\), since \(F_\ell C_\bullet(X)\) is the
  normalised chain complex of \(\Nv(X)^{\le\ell}\) for every \(X\) in \(\Met\).
\item\label{itm:comp-graded} The space \(T_UA\) is an object of \(\Met\) by
  \Cref{thm:free}\Cref{itm:forgetfulfunctor}, so \Cref{def:magnitude} applies to \(X=T_UA\) and gives
  the two equalities
  \[
    \gr_\ell C_\bullet(T_UA)=\MC_{\bullet,\ell}(T_UA),
    \qquad
    H_n\bigl(\gr_\ell C_\bullet(T_UA)\bigr)=\MH_{n,\ell}(T_UA),
  \]
  the second being the first claim of this part, an equality of groups rather
  than a mere isomorphism. Naturality in \(A\) is the statement that the
  assignment is the composite of the two functors \(T_U\colon\Met\to\Met\) and
  \(\MH_{n,\ell}\colon\Met\to\Ab\) of \Cref{lem:complex}\Cref{itm:mc-functor}: a morphism
  \(h\colon A\to B\) of \(\Met\) is sent to
  \(\MH_{n,\ell}(T_Uh)\colon\MH_{n,\ell}(T_UA)\to\MH_{n,\ell}(T_UB)\), and
  composites go to composites. For morphisms \(h\colon A\to B\) and
  \(g\colon B\to C\) of \(\Met\) the functor laws of \(T_U\) and of
  \(\MH_{n,\ell}\) give
  \begin{align*}
    \MH_{n,\ell}(T_Ug)\circ\MH_{n,\ell}(T_Uh)
      &=\MH_{n,\ell}(T_Ug\circ T_Uh)
      =\MH_{n,\ell}\bigl(T_U(g\circ h)\bigr),\\
    \MH_{n,\ell}(T_U1_A)&=1_{\MH_{n,\ell}(T_UA)},
  \end{align*}
  so the assignment \(A\mapsto\MH_{n,\ell}(T_UA)\) is a functor
  \(\Met\to\Ab\).\qedhere
\end{enumerate}
\end{proof}

\begin{proof}[Proof of \Cref{thm:theorymap}]~
\begin{enumerate}
\item\label{itm:tmap-restrict} \textbf{Restriction.} Let \(B\) be an object of
  \(\Mod(U')\), so that \(B\) is a quantitative algebra with \(B\models\varphi\)
  for every \(\varphi\in U'\). From \(U\subseteq U'\) it satisfies every
  \(\varphi\in U\), so \(B\) is an object of \(\Mod(U)\). Both categories are
  full subcategories of \(\QAlg\) by \Cref{def:qalg}, so \(\Mod(U')\) is a full
  subcategory of \(\Mod(U)\), and the two forgetful functors agree on it, both
  sending an algebra \((B,\Omega^B,d^B)\) to the space \((B,d^B)\). In
  particular \(\Free_{U'}A\) is an object of \(\Mod(U)\) with
  \(G_U\Free_{U'}A=G_{U'}\Free_{U'}A=T_{U'}A\), for every \(A\) in \(\Met\).
\item\label{itm:tmap-transpose} \textbf{The Transpose.} Let \(A\) be an object
  of \(\Met\) and \(B\) one of \(\Mod(U)\). The adjunction
  \(\Free_U\dashv G_U\) of \Cref{thm:free}\Cref{itm:forgetfulfunctor} is the bijection
  \begin{equation}\label{eq:tmap-transpose}
    \Met(A,G_UB)\;\cong\;\Mod(U)(\Free_UA,B),
    \qquad
    h\longmapsto h^\dagger,
  \end{equation}
  natural in \(A\) and \(B\), in which \(h^\dagger\) is the one and only
  morphism \(\Free_UA\to B\) of \(\Mod(U)\) whose underlying map fills the
  triangle of \(\Met\)
  \[
  \begin{tikzcd}[row sep=1.7em, column sep=3.2em, ampersand replacement=\&]
    A\arrow[r,"\eta^U_A"]\arrow[dr,"h"']
      \&T_UA\arrow[d,dashed,"G_U(h^\dagger)"]\\
    \&G_UB,
  \end{tikzcd}
  \]
  so that two morphisms \(g,g'\colon\Free_UA\to B\) of \(\Mod(U)\) with
  \(G_U(g)\circ\eta^U_A=G_U(g')\circ\eta^U_A\) are equal. The dagger $\dagger$ keeps this
  transpose apart from the \(\Omega\)-algebra extension \((-)^\sharp\) of
  \eqref{eq:extension}, which extends along \(\eta_W\) over \(\Set\) rather
  than along \(\eta^U_A\) over \(\Met\). Define
  \[
    q_A=G_U\bigl((\eta^{U'}_A)^\dagger\bigr)\colon T_UA\longrightarrow T_{U'}A,
  \]
  the transpose taken for \(B=\Free_{U'}A\), an object of \(\Mod(U)\) by
  \Cref{itm:tmap-restrict}. The triangle above reads
  \(q_A\circ\eta^U_A=\eta^{U'}_A\). For the uniqueness claimed in the
  statement, let \(q'\) be a natural transformation whose component at \(A\)
  underlies a morphism \(g_A\colon\Free_UA\to\Free_{U'}A\) of \(\Mod(U)\) with
  \(G_U(g_A)\circ\eta^U_A=\eta^{U'}_A\). Then \(g_A=(\eta^{U'}_A)^\dagger\) by
  the uniqueness in \eqref{eq:tmap-transpose}, so
  \(q'_A=G_U(g_A)=G_U\bigl((\eta^{U'}_A)^\dagger\bigr)=q_A\) one object at a
  time, and \(q'=q\).
\item\label{itm:tmap-generation} \textbf{Generation.} Let \(V\) be a
  quantitative equational theory over \((\Omega,\ar)\) and \(A\) an object of
  \(\Met\). We claim that \(T_VA\) is generated by \(\eta^V_A(A)\) as an
  \(\Omega\)-algebra: the only subset of \(T_VA\) that contains
  \(\eta^V_A(A)\) and is closed under every operation \(f^{T_VA}\) with
  \(f\in\Omega\) is \(T_VA\) itself. Let \(S\) be the smallest such subset, the
  intersection of the family of all of them, a family that contains \(T_VA\)
  and whose intersection again contains \(\eta^V_A(A)\) and is closed under
  every operation. Restricting the operations and the metric makes
  \(\bar S=(S,(f^{T_VA}|_{S^{\ar f}})_{f\in\Omega},d^{T_VA}|_{S\times S})\) a
  quantitative algebra, with nonexpansiveness inherited. It lies in
  \(\Mod(V)\): an assignment \(\iota\colon\Xi\to S\) is an assignment into
  \(T_VA\), and its extensions to \(\bar S\) and to \(T_VA\) agree, both being
  \(\Omega\)-homomorphisms \(T_\Omega\Xi\to T_VA\) that restrict to \(\iota\)
  on \(\Xi\), which \Cref{lem:term-monad}\Cref{itm:tm-free} makes equal, so every satisfaction
  condition of \Cref{def:qalg} quantifies for \(\bar S\) over a subset of the
  assignments quantified over for \(T_VA\), and \(T_VA\in\Mod(V)\) forces
  \(\bar S\models\varphi\) for every \(\varphi\in V\). The inclusion
  \(\iota_S\colon\bar S\hookrightarrow\Free_VA\) is an \(\Omega\)-homomorphism and
  nonexpansive, so a morphism of \(\Mod(V)\), and \(\eta^V_A\) corestricts to a
  nonexpansive map \(\bar\eta\colon A\to G_V\bar S\) with
  \(G_V(\iota_S)\circ\bar\eta=\eta^V_A\). Now transpose \(\bar\eta\) along
  \(\Free_V\dashv G_V\) as in \eqref{eq:tmap-transpose} and chase the triangle
  \[
  \begin{tikzcd}[row sep=1.7em, column sep=3.6em, ampersand replacement=\&]
    A\arrow[r,"\eta^V_A"]\arrow[dr,"\bar\eta"']
      \&T_VA\arrow[d,"G_V(\bar\eta^\dagger)"]
        \arrow[dr,bend left=18,"1_{T_VA}"]\&\\
    \&G_V\bar S\arrow[r,hook,"G_V(\iota_S)"']\&T_VA\mathrlap{\,:}
  \end{tikzcd}
  \]
  the composite \(\iota_S\circ\bar\eta^\dagger\colon\Free_VA\to\Free_VA\) is a
  morphism of \(\Mod(V)\) with
  \(G_V(\iota_S\circ\bar\eta^\dagger)\circ\eta^V_A
  =G_V(\iota_S)\circ\bar\eta=\eta^V_A=G_V(1_{\Free_VA})\circ\eta^V_A\), so the
  uniqueness in the transposition bijection for \(V\) gives
  \(\iota_S\circ\bar\eta^\dagger=1_{\Free_VA}\). Every \(t\in T_VA\) is then
  the image \(\iota_S(s)\) of an element \(s\) of \(S\), so the inclusion
  \(\iota_S\) is surjective and \(S=T_VA\).
\item\label{itm:tmap-monad} \textbf{Naturality and the Monad Axioms.} Let
  \(f\colon A\to B\) be a morphism of \(\Met\). Both
  \(q_B\circ T_Uf\) and \(T_{U'}f\circ q_A\) underlie morphisms
  \(\Free_UA\to\Free_{U'}B\) of \(\Mod(U)\), namely
  \((\eta^{U'}_B)^\dagger\circ\Free_Uf\) and
  \(\Free_{U'}f\circ(\eta^{U'}_A)^\dagger\), with \(\Free_{U'}f\) a morphism of
  \(\Mod(U)\) by \Cref{itm:tmap-restrict}, and
  \begin{align}
    q_B\circ T_Uf\circ\eta^U_A
      &=q_B\circ\eta^U_B\circ f
      &&\eta^U\text{ natural at }f,\notag\\
      &=\eta^{U'}_B\circ f,
      &&\text{triangle for }q_B,\notag\\
    T_{U'}f\circ q_A\circ\eta^U_A
      &=T_{U'}f\circ\eta^{U'}_A
      &&\text{triangle for }q_A,\label{eq:tmap-nat}\\
      &=\eta^{U'}_B\circ f,
      &&\eta^{U'}\text{ natural at }f,\notag
  \end{align}
  so the uniqueness in \eqref{eq:tmap-transpose} makes the naturality square of
  \(q\) at \(f\) commute. Of the two axioms of a morphism of monads, the unit
  axiom \(q_A\circ\eta^U_A=\eta^{U'}_A\) is the triangle of
  \Cref{itm:tmap-transpose}. The multiplication axiom is the commutativity,
  for every \(A\) in \(\Met\), of the square
  \begin{equation}\label{eq:tmap-mult}
  \begin{tikzcd}[row sep=2.2em, column sep=3.4em, ampersand replacement=\&]
    T_UT_UA\arrow[r,"T_Uq_A"]\arrow[d,"\mu^U_A"']
      \&T_UT_{U'}A\arrow[r,"q_{T_{U'}A}"]
      \&T_{U'}T_{U'}A\arrow[d,"\mu^{U'}_A"]\\
    T_UA\arrow[rr,"q_A"']\&\&T_{U'}A
  \end{tikzcd}
  \end{equation}
  in which \(\mu^U_A=G_U\bigl(\varepsilon^U_{\Free_UA}\bigr)\) and
  \(\mu^{U'}_A=G_{U'}\bigl(\varepsilon^{U'}_{\Free_{U'}A}\bigr)\) are the
  multiplications of the two induced monads, with \(\varepsilon^U\) and
  \(\varepsilon^{U'}\) the counits of the two adjunctions. Both composites of
  \eqref{eq:tmap-mult} underlie morphisms \(\Free_U(T_UA)\to\Free_{U'}A\) of
  \(\Mod(U)\): the clockwise one is built from \(\Free_U(q_A)\), from
  \((\eta^{U'}_{T_{U'}A})^\dagger\) and from \(\varepsilon^{U'}_{\Free_{U'}A}\),
  a morphism of \(\Mod(U')\) and so one of \(\Mod(U)\), the anticlockwise one
  from \(\varepsilon^U_{\Free_UA}\) and \((\eta^{U'}_A)^\dagger\). They agree
  after precomposition with the unit component
  \(\eta^U_{T_UA}\colon T_UA\to T_UT_UA\) of the monad \(T_U\),
  \begin{align}
    q_A\circ\mu^U_A\circ\eta^U_{T_UA}
      &=q_A,
      &&\text{triangle id. of }\Free_U\dashv G_U,\notag\\
    \mu^{U'}_A\circ q_{T_{U'}A}\circ T_Uq_A\circ\eta^U_{T_UA}
      &=\mu^{U'}_A\circ q_{T_{U'}A}\circ\eta^U_{T_{U'}A}\circ q_A
      &&\eta^U\ \text{nat.\ at}\ q_A,\label{eq:tmap-mult-check}\\
      &=\mu^{U'}_A\circ\eta^{U'}_{T_{U'}A}\circ q_A
      &&\text{triangle for }q_{T_{U'}A},\notag\\
      &=q_A,
      &&\text{triangle id.\ of }\Free_{U'}\dashv G_{U'},\notag
  \end{align}
  so the uniqueness in \eqref{eq:tmap-transpose} makes \eqref{eq:tmap-mult}
  commute, and \(q\) is a morphism of monads.
\item\label{itm:tmap-surj} \textbf{Surjectivity and Nonexpansiveness.} The map
  \(q_A\) underlies a morphism of \(\QAlg\), so it is nonexpansive by
  \Cref{def:qalg}. Its image \(q_A(T_UA)\subseteq T_{U'}A\) contains
  \(\eta^{U'}_A(A)=q_A(\eta^U_A(A))\) and is closed under every operation,
  since \(q_A\) is an \(\Omega\)-homomorphism and
  \(f^{T_{U'}A}(q_At_1,\dots,q_At_n)=q_A\bigl(f^{T_UA}(t_1,\dots,t_n)\bigr)\)
  exhibits each value on elements of the image as an element of the image. By
  \Cref{itm:tmap-generation} applied with \(V=U'\), the only subset of
  \(T_{U'}A\) that contains \(\eta^{U'}_A(A)\) and is closed under every
  operation is \(T_{U'}A\) itself, so \(q_A(T_UA)=T_{U'}A\), and \(q_A\) is
  surjective.
\item\label{itm:tmap-ph} \textbf{Persistent Homology of the Comparison.} For
  part \Cref{itm:tmap-homology} apply the functor
  \(\PH_n\colon\Met\to\Ab^{\Vfin^{\op}}\) of
  \Cref{thm:composite}\Cref{itm:comp-ph} to the nonexpansive map \(q_A\), and
  let \(\PH_n(q)_A\) be the resulting map
  \(\PH_n(q_A)\colon\PH^U_n(A)\to\PH^{U'}_n(A)\). For every
  morphism \(f\colon A\to B\) of \(\Met\), that functor carries the left
  of the two squares below, the naturality square of \(q\) at \(f\)
  established in \Cref{itm:tmap-monad}, to the homology square on its right,
  \[
  \begin{tikzcd}[row sep=2.1em, column sep=2.9em, ampersand replacement=\&]
    T_UA\arrow[r,"T_Uf"]\arrow[d,two heads,"q_A"']
      \&T_UB\arrow[d,two heads,"q_B"]\\
    T_{U'}A\arrow[r,"T_{U'}f"']\&T_{U'}B,
  \end{tikzcd}
  \qquad
  \begin{tikzcd}[row sep=2.1em, column sep=2.9em, ampersand replacement=\&]
    \PH^U_n(A)\arrow[r,"\PH_n(T_Uf)"]\arrow[d,"\PH_n(q_A)"']
      \&\PH^U_n(B)\arrow[d,"\PH_n(q_B)"]\\
    \PH^{U'}_n(A)\arrow[r,"\PH_n(T_{U'}f)"']\&\PH^{U'}_n(B),
  \end{tikzcd}
  \]
  because a functor preserves commuting squares. The right squares are the
  naturality squares of the family \((\PH_n(q_A))_A\), which is therefore a
  natural transformation \(\PH_n(q)\colon\PH^U_n\Rightarrow\PH^{U'}_n\) between
  functors \(\Met\to\Ab^{\Vfin^{\op}}\).
\item\label{itm:tmap-bound} \textbf{The Bound.} Assume \(q_A\) is bijective
  and \(d^{T_UA}(t,s)\le d^{T_{U'}A}(q_At,q_As)+\delta\) for all
  \(t,s\in T_UA\). Let \(Y\) be the pair consisting of the set \(\Ob T_UA\)
  and the function
  \[
    d_Y\colon\Ob T_UA\times\Ob T_UA\longrightarrow\Rp,
    \qquad
    d_Y(t,s)=d^{T_{U'}A}(q_At,q_As).
  \]
  It is an object of \(\Met\): the function \(d_Y\) inherits
  \(d_Y(t,t)=0\), symmetry and the triangle inequality from
  \(d^{T_{U'}A}\), and \(d_Y(t,s)=0\) forces \(q_At=q_As\), since
  \(T_{U'}A\) is separated, and then \(t=s\), since \(q_A\) is injective. The
  objects \(X=T_UA\) and \(Y\) of \(\Met\) share their underlying set, and the
  two comparisons of \Cref{thm:distortion} hold for them,
  \begin{align}
    d_X(t,s)&\le d^{T_{U'}A}(q_At,q_As)+\delta=d_Y(t,s)+\delta,
      &&\text{hypothesis},\notag\\
    d_Y(t,s)&=d^{T_{U'}A}(q_At,q_As)\le d_X(t,s)\le d_X(t,s)+\delta,
      &&q_A\text{ nonexpansive},\label{eq:tmap-bound}
  \end{align}
  so \(d_I\bigl(\PH_n(X),\PH_n(Y)\bigr)\le(n+1)\delta\). That factor is a count
  of summands, and here it counts as follows. An \(m\)-simplex of the length
  nerve the two objects share is a tuple \(\mathbf{t}=\langle t_0,\dots,t_m\rangle\)
  of elements of \(T_UA\); it has \(m\) consecutive pairs, its two lengths are the
  \(m\)-term sums \(\lambda_X\mathbf{t}=\sum_{i=1}^{m}d_X(t_{i-1},t_i)\) and
  \(\lambda_Y\mathbf{t}=\sum_{i=1}^{m}d_Y(t_{i-1},t_i)\), and
  \eqref{eq:tmap-bound} bounds each of the \(m\) differences by \(\delta\), so
  \eqref{eq:length-shift} reads
  \[
    \lambda_Y\mathbf{t}-\lambda_X\mathbf{t}
      =\sum_{i=1}^{m}\bigl(d^{T_{U'}A}(q_At_{i-1},q_At_i)-d^{T_UA}(t_{i-1},t_i)\bigr)
      \in[-m\delta,\,m\delta].
  \]
  The group \(\PH_n\) is computed from the degrees \(n-1\), \(n\) and \(n+1\)
  alone, and the truncation in the proof of \Cref{thm:distortion} discards
  everything above \(n+1\), so \(m\le n+1\) throughout and the largest of these
  counts is \(m=n+1\): a simplex \(\langle t_0,\dots,t_{n+1}\rangle\) carries
  \(n+1\) distances, each of which the axioms of \(U'\smallsetminus U\) may
  shorten by \(\delta\), and its length can therefore drop by
  \((n+1)\delta\) and by no more. That worst case is the shift for which the
  two maps \(\varphi_\ell\) and \(\psi_\ell\) of \eqref{eq:dist-diagram} are
  defined, so \(\Delta=(n+1)\delta\), and no smaller shift makes them well
  defined once one \((n+1)\)-simplex has all \(n+1\) of its distances moved by
  \(\delta\). The map \(q_A\), read
  as a map \(\bar q\colon Y\to T_{U'}A\), preserves distances by the definition
  of \(d_Y\) and is bijective, so \(\bar q\) and its inverse are morphisms of
  \(\Met\), and the pair with the components
  \(\varphi_\ell=\PH_n(\bar q)(\ell)\) and
  \(\psi_\ell=\PH_n(\bar q^{-1})(\ell)\) is a \(0\)-interleaving of
  \(\PH_n(Y)\) and \(\PH_n(T_{U'}A)\), its two composites being the identities,
  which are the transition maps along \(\ell\le\ell+0\). So
  \(d_I\bigl(\PH_n(Y),\PH_n(T_{U'}A)\bigr)=0\), and the triangle inequality of
  \Cref{prop:interleaving-vcat} through the midpoint \(\PH_n(Y)\) closes the
  computation of \Cref{itm:boundedInterleaving},
  \begin{align*}
    d_I\bigl(\PH^U_n(A),\PH^{U'}_n(A)\bigr)
    &\le d_I\bigl(\PH_n(X),\PH_n(Y)\bigr)
      +d_I\bigl(\PH_n(Y),\PH_n(T_{U'}A)\bigr)\\
    &\le(n+1)\delta+0.\qedhere
  \end{align*}
\end{enumerate}
\end{proof}

\begin{figure}[p]
\centering
\scalebox{0.92}{\input{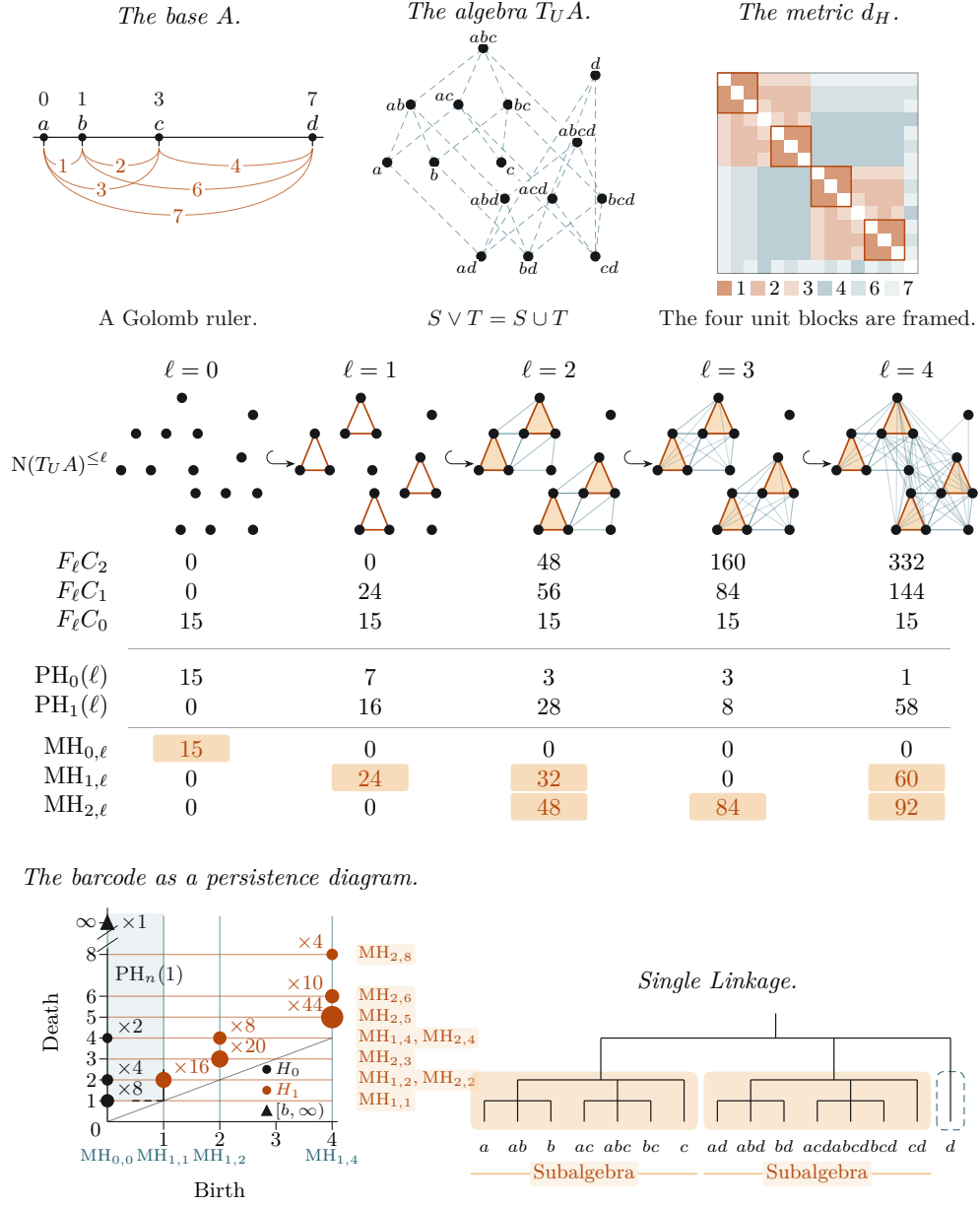}}
\caption{The base \(A\) is the four marks
\(0,1,3,7\) of a Golomb ruler, so its six distances \(1,2,3,4,6,7\) differ,
and \(T_UA\) is the fifteen nonempty subsets under union and the Hausdorff metric
\(d_H\). The points keep their places in every panel: two copies of one
seven-point picture, the subsets without \(d\) and those with it, and the outlier
\(\{d\}\). Heavy edges are pairs at distance \(1\), fainter ones the longer pairs,
shading marks \(2\)-simplices, ranks are over \(\mathbb{F}_2\). By
\eqref{eq:endpoint-split} births lie on cool rules and deaths on warm ones, each
rule carries an endpoint, so the inclusion of \Cref{cor:critical} is an equality
in both degrees. The tinted quadrant holds the bars alive at \(\ell=1\), the
\(\ell=1\) column of the table
\cite[\S{}VII.1]{EdelsbrunnerHarer2010ComputationalTopology}. The dendrogram is
\(\PH_0\) as single linkage by \Cref{cor:degree-zero}, its merge heights the
deaths \(1,2,4\), every cluster a subalgebra, and the sixteen \(H_1\)-classes
at \((1,2)\) are four per unit triangle, three digons of \Cref{lem:digon} and a
directed triangle, one triangle over each of the four subsets \(S\subseteq\{c,d\}\),
on \(S\cup\{a\}\), \(S\cup\{b\}\) and \(S\cup\{a,b\}\),
the only unit triangles of \(T_UA\).}
\label{fig:pipeline}
\end{figure}

\end{document}